\documentclass{ws-ijbc}
\usepackage{ws-rotating}     
\usepackage{graphicx}
\usepackage{epstopdf}
\usepackage{subfigure}
\usepackage{color}
\begin{document}

\catchline{}{}{}{}{} 

\markboth{Yong Yao}{Dynamics of a Prey-predator System with Foraging Facilitation in Predators}

\title{Dynamics of a Prey-predator System with Foraging\\ Facilitation in Predators}

\author{Yong Yao\footnote{ Author for correspondence}}

\address{Department of Mathematics, Sichuan University,\\
Chengdu, Sichuan 610064, PR China\\
mathyaoyong@163.com}

\maketitle

\begin{history}
\received{(to be inserted by publisher)}
\end{history}

\begin{abstract}
The dynamics of a prey-predator system with foraging facilitation among predators
are investigated. The analysis involves the computation of many semi-algebraic systems of
large degrees. We apply the pseudo-division reduction, real-root isolation technique and
complete discrimination system of polynomial to obtain parameter conditions for the
exact number of equilibria and their qualitative properties as well as a complete investigation of
bifurcations including saddle-node, transcritical, pitchfork, Hopf and Bogdanov-Takens bifurcations.
Moreover, numerical simulations are presented to support our theoretical results.
\end{abstract}


\section{Introduction}
\noindent Populations rarely exist in isolation, which results in ecological systems are characterized by the interaction between species and environment.
Mathematical models play important roles in understanding population interactions (\cite{Freedman,Kot}).
An important type of interaction is predation,
which leads to prey-predator models that have great importance in ecology.
One of the classic prey-predator models, the Rosenzweig-MacArthur model(\cite{Rosenzweig}), is given by
\begin{eqnarray}
\begin{array}{l}
\frac{dN}{dt}=rN(1-\frac{N}{K})-\frac{E NP}{1+HE N},~~~
\frac{dP}{dt}=e\frac{ENP}{1+HE N}-mP,
\end{array}
\label{(00)}
\end{eqnarray}
where $N(t)$ and $P(t)$ represent densities of the prey and predator at time $t$ respectively, $r$  stands for
the intrinsic growth rate of prey,  $K$ is the carrying capacity of prey, $e$ is the conversion rate,  $m$ is the mortality rate of predator, $E$ is the encounter rate of predator with the prey and $H$ is the  predator handling time of a prey individual. Some researchers (\cite{Hsu,Cheng,Huang,Turchin,Kot}) have studied the dynamical behaviors of system (\ref{(00)}), which has a coexistence equilibrium rose from transcritical bifurcation and
a unique limit cycle induced by Hopf bifurcation.  They also have shown
both the prey and predator populations survive either to the coexistence equilibrium or the limit cycle.
Another widespread type of interaction in ecological systems is cooperation among individuals (\cite{Dugatkin}),
which seems to be an important evolutionary cause of sociality and a key factor for exploring and understanding many aspects of how organisms are designed.
There are a great variety of cooperative behaviors in nature such as cooperative defence against predators (\cite{Garay}), cooperative breeding (\cite{CourchampBerec}), alarm calling (\cite{Lehmann}) and cooperative hunting (\cite{Boesch,Packer}).
The behaviour of cooperation during prey hunting has been observed in many different species, for instance, some species of tuna hunt in a linear school and aggregate when they encounter a school of prey (\cite{Partridge}) and
wolves can hunt animals bigger or faster than themselves by cooperative hunting (\cite{Schmidt}).
Foraging facilitation or hunting cooperation embraces a number of specific mechanisms such as
locating and capturing the prey in a bigger group (\cite{Cosner}), protecting any of members from predation (\cite{Krause})
and intraspecific cooperation (\cite{Courchamp}).
Recently, the foraging facilitation has been taken into consideration in some mathematical literatures (\cite{Berec,Cosner,Kimun,Pribylova,Alves,Saheb}).
Foraging facilitation can be depicted in mathematical models by functional response, which means the per capita feeding rate of predators on their prey. The independence of the Holling type II functional response in system (\ref{(00)}) from predator density is hardly always true in reality because it  reflects that any single predator affects the growth rate of prey independently of its conspecifics. Therefore,
functional response might depend on predator density and is increasing with respect to predator density for the case of foraging facilitation.
That is to say, when any of the foraging facilitation mechanisms operates, $E$ in Holling type II functional response no longer is a constant, but rather an increasing function of predator density.

Berec (\cite{Berec}) extended the classical Rosenzweig-MacArthur system by including foraging facilitation and proposed the following prey-predator system
\begin{eqnarray}
\begin{array}{l}
\frac{dN}{dt}=rN(1-\frac{N}{K})-\frac{E(P) NP}{1+HE(P) N},~~
\frac{dP}{dt}=e\frac{E(P) NP}{1+HE(P) N}-mP
\end{array}
\label{(01)}
\end{eqnarray}
with the encounter-driven functional response $E(P):=e_1/(e_2+P)^{\omega}$, where $e_1>0$, $e_2\geq0$ and $\omega\leq0$.
Clearly, the above model is exactly the Rosenzweig-MacArthur model as $\omega=0$, and it characterizes the
foraging facilitation as $\omega<0$.
Berec gave a brief overview on the number and stabilities of coexistence equilibria of system (\ref{(01)}), and later Pribylova and Peniaskova (\cite{Pribylova}) considered the bifurcation behaviors through qualitative analysis combined with numerical simulations. In the special case $e_2=0$ and $\omega=-1$, the functional response happens to be the one considered by Cosner (\cite{Cosner}), which actually describes the foraging facilitation in a spatially linear formation and aggregation when the predators encounter a cluster of prey. Kimun {\it et al} (\cite{Kimun}) analyzed system (\ref{(01)}) with the special functional response. Furthermore, Alves and Hilker (\cite{Alves}) investigated both of the two special cases $\omega=-1$, $e_2>0$ with $H=0$ and $H>0$ respectively and derived the result that the hunting cooperation in the prey-predator system induces Allee effects in predators.
In the case $\omega=-1$, $e_2>0$ and $H=0$,  they investigated the stabilities of equilibria and saddle-node, Hopf and Bogdanov-Takens bifurcations.
In the case $\omega=-1$, $e_2>0$ and $H>0$, by dimensionless transformations
$x=\frac{ee_1e_2}{m}N$, $y=\frac{e_1e_2}{m}P$, $\tau=mt$, $\sigma=\frac{r}{m}$, $k=\frac{ee_1e_2K}{m}$, $\alpha=\frac{ m}{e_1e_2^2}$ and $h=\frac{mH}{e}$
system (\ref{(01)}) can be written into
\begin{eqnarray}
\begin{array}{l}
\frac{dx}{d\tau}=x\{\sigma(1-\frac{x}{k})-\frac{(1+\alpha y)y}{1+h(1+\alpha y)x}\},~~~
\frac{dy}{d\tau}=y\{\frac{(1+\alpha y)x}{1+h(1+\alpha y)x}-1\},
\end{array}
\label{(2)}
\end{eqnarray}
where $\alpha$ describes the intensity of predator cooperation in hunting. System (\ref{(2)}) is a direct extension of the Rosenzweig-MacArthur model by considering the foraging facilitation.
Alves and Hilker (\cite{Alves})  presented a two-parameter bifurcation diagram of system (\ref{(2)}) for
special parameter values $h=0.1$ and $k=0.8$.
Therefore,  further carrying out a detailed study of system (\ref{(2)}) is the task of
this paper.

Note that system (\ref{(2)})  is orbitally equivalent to the following quartic system
\begin{eqnarray}
\begin{array}{l}
\frac{dx}{dt}=x\{\sigma(k-x)(1+h(1+\alpha y)x)-ky(1+\alpha y)\},~~
\frac{dy}{dt}=ky\{x(1+\alpha y)(1-h)-1\}.
\end{array}
\label{(3)}
\end{eqnarray}
In this paper,  we investigate the dynamics of
the above system with positive parameters $h$, $k$, $\sigma$ and $\alpha$  in the closure of the first
quadrant $\overline{\mathbb{R}_+^2}:=\{(x,y)\in\mathbb{R}^2:x\geq0, y\geq0\}$.
$\overline{\mathbb{R}_+^2}$ is positively invariant under the flow generated by system (\ref{(3)}). In fact, the origin $(0, 0)$ is an equilibrium, the positive $y$-axis is an orbital and the positive $x$-axis consists of three orbitals, i.e., $0<x<k$, $x>k$ and the equilibrium $(k, 0)$.
Notice that the abscissas of equilibria of the above system are decided by those positive roots of a cubic polynomial with complicated coefficients. However, generically we cannot obtain the analytic expressions of
those equilibria.
In Section 2,  we qualitatively analyse the cubic polynomial equilibrium function and investigate the relative positions of those roots for the equilibrium
function and the trace of the Jacobian matrix. Consequently, we obtain the parameter conditions for the exact number of equilibria and their qualitative properties.  Section 3 is devoted to equilibria with exact one zero eigenvalue. Restricting on
the center manifold, we obtain parameter conditions for transcritical, pitchfork and
saddle-node bifurcations. In Section 4, we apply the pseudo-division reduction (\cite{Winkler}) and
real-root isolation technique to determine the sign of the first quantity of focus, which
is a quartic polynomial with complex coefficients. It is proved that at most one limit
cycle bifurcates via Hopf bifurcation. In Section 5, we investigate the Bogdanov-Takens
bifurcation and show that it is codimension 2. Furthermore, the complete discrimination system of polynomial (\cite{YangLu}) is applied to verify the transversal condition.
In Section 6, we verify the results by numerical simulations and end the paper with a
brief biological implications.

\section{Equilibria and Their Properties}
In order to state our results conveniently, we consider the partition
$\mathbb{R}_+^3:=\{(k,\sigma,\alpha)\in\mathbb{R}^3:k>0, \sigma>0, \alpha>0\}=\mathcal{P}_1\cup\mathcal{S}_1\cup\mathcal{P}_2\cup\mathcal{S}_2\cup\mathcal{L}_1\cup\mathcal{S}_3\cup\mathcal{P}_3
\cup\mathcal{S}_4\cup\mathcal{P}_4\cup\mathcal{S}_5\cup\mathcal{P}_5$,
where
\begin{eqnarray*}
\left.
\begin{array}{l}
\mathcal{P}_1:=\{(k,\sigma,\alpha)\in\mathbb{R}_+^3: k>k_1, \alpha<\frac{1}{\sigma k}\},~
\mathcal{S}_1:=\{(k,\sigma,\alpha)\in\mathbb{R}_+^3: k=k_1, \alpha<\frac{1}{\sigma k}\},\\
\mathcal{P}_2:=\{(k,\sigma,\alpha)\in\mathbb{R}_+^3: k<k_1, \alpha<\frac{1}{\sigma k}\},~
\mathcal{S}_2:=\{(k,\sigma,\alpha)\in\mathbb{R}_+^3: k>k_1, \alpha=\frac{1}{\sigma k}\},\\
\mathcal{L}_1:=\{(k,\sigma,\alpha)\in\mathbb{R}_+^3: k=k_1, \alpha=\frac{1}{\sigma k}\},~
\mathcal{S}_3:=\{(k,\sigma,\alpha)\in\mathbb{R}_+^3: k<k_1, \alpha=\frac{1}{\sigma k}\},\\
\mathcal{P}_3:=\{(k,\sigma,\alpha)\in\mathbb{R}_+^3: k>k_1, \alpha>\frac{1}{\sigma k}\},~
\mathcal{S}_4:=\{(k,\sigma,\alpha)\in\mathbb{R}_+^3: k=k_1, \alpha>\frac{1}{\sigma k}\},\\
\mathcal{S}_5:=\{(k,\sigma,\alpha)\in\mathbb{R}_+^3: k<k_1, \alpha=\alpha_1\},~
\mathcal{P}_5:=\{(k,\sigma,\alpha)\in\mathbb{R}_+^3: k<k_1, \alpha>\alpha_1\},\\
\mathcal{P}_4:=\{(k,\sigma,\alpha)\in\mathbb{R}_+^3: k<k_1, \frac{1}{\sigma k}<\alpha<\alpha_1\}\\
\end{array}
\right.
\end{eqnarray*}
with $k_1:=1/(1-h)$ and
\begin{eqnarray}
\left.
\begin{array}{l}
\alpha_1:=\frac{-(h-1)^2k^2+18(h-1)k+27+(hk-k+9)\sqrt{(hk-k+1)(hk-k+9)}}{8\sigma(1-h)k^2}.
\end{array}
\right.
\label{(07)}
\end{eqnarray}
We further consider partitions
$\mathcal{P}_1=\mathcal{P}_{11}\cup\mathcal{P}_{12}\cup\mathcal{S}_{11}$,
$\mathcal{S}_2=\mathcal{S}_{21}\cup\mathcal{S}_{22}\cup\mathcal{L}_{21}$,
$\mathcal{P}_3=\mathcal{P}_{31}\cup\mathcal{P}_{32}\cup\mathcal{S}_{31}$,
$\mathcal{S}_4=\mathcal{S}_{41}\cup\mathcal{L}_{41}\cup\mathcal{S}_{42}$
and $\mathcal{P}_5=\mathcal{P}_{51}\cup\mathcal{P}_{52}\cup\mathcal{S}_{51}$,
where
\begin{eqnarray*}
\begin{array}{l}
\mathcal{P}_{11}:=\{(k,\sigma,\alpha)\in\mathbb{R}_+^3:k\geq k_2, \alpha<\frac{1}{\sigma k} ~\mbox{or}~ k_3\leq k<k_2, \alpha_2<\alpha<\frac{1}{\sigma k} ~\mbox{or}~ k_1<k<k_3, \sigma<\sigma_1, \alpha_2<\alpha<\frac{1}{\sigma k}\},\\
\mathcal{P}_{12}:=\{(k,\sigma,\alpha)\in\mathbb{R}_+^3:k_3\leq k<k_2, \alpha<\alpha_2 ~\mbox{or}~  k_1<k<k_3, \sigma\geq\sigma_1, \alpha<\frac{1}{\sigma k} ~\mbox{or}~ k_1<k<k_3, \sigma<\sigma_1, \\
\phantom{\mathcal{P}_{12}:=}\alpha<\alpha_2\},~
\mathcal{S}_{11}:=\{(k,\sigma,\alpha)\in\mathbb{R}_+^3:k_3\leq k<k_2, \alpha=\alpha_2 ~\mbox{or}~ k_1<k<k_3, \sigma<\sigma_1, \alpha=\alpha_2\},\\
\mathcal{S}_{21}:=\{(k,\sigma,\alpha)\in\mathbb{R}_+^3:k\geq k_3, \alpha=\frac{1}{\sigma k} ~\mbox{or}~ k_1<k<k_3, \sigma<\sigma_1, \alpha=\frac{1}{\sigma k}\},\\
\mathcal{S}_{22}:=\{(k,\sigma,\alpha)\in\mathbb{R}_+^3:k_1<k<k_3, \sigma>\sigma_1, \alpha=\frac{1}{\sigma k}\},
\mathcal{L}_{21}:=\{(k,\sigma,\alpha)\in\mathbb{R}_+^3:k_1<k<k_3, \sigma=\sigma_1, \alpha=\frac{1}{\sigma k}\},\\
\mathcal{P}_{31}:=\{(k,\sigma,\alpha)\in\mathbb{R}_+^3:k\geq k_3, \alpha>\frac{1}{\sigma k} ~\mbox{or}~ k_1<k<k_3, \sigma\geq\sigma_1, \alpha>\alpha_2   ~\mbox{or}~  k_1<k<k_3, \sigma<\sigma_1, \alpha>\frac{1}{\sigma k}\},\\
\mathcal{P}_{32}:=\{(k,\sigma,\alpha)\in\mathbb{R}_+^3:k_1<k<k_3, \sigma>\sigma_1, \frac{1}{\sigma k}<\alpha<\alpha_2\},\\
\mathcal{S}_{31}:=\{(k,\sigma,\alpha)\in\mathbb{R}_+^3:k_1<k<k_3, \sigma>\sigma_1, \alpha=\alpha_2\},
\mathcal{S}_{41}:=\{(k,\sigma,\alpha)\in\mathbb{R}_+^3:k=k_1, \frac{1}{\sigma k}<\alpha<\alpha_2\},\\
\mathcal{L}_{41}:=\{(k,\sigma,\alpha)\in\mathbb{R}_+^3:k=k_1, \alpha=\alpha_2\},~
\mathcal{S}_{42}:=\{(k,\sigma,\alpha)\in\mathbb{R}_+^3:k=k_1, \alpha>\alpha_2\},\\
\mathcal{P}_{51}:=\{(k,\sigma,\alpha)\in\mathbb{R}_+^3:k<k_1, \sigma\leq\sigma_2, \alpha>\alpha_1 ~\mbox{or}~ k<k_1, \sigma>\sigma_2, \alpha>\alpha_2\},\\
\mathcal{P}_{52}:=\{(k,\sigma,\alpha)\in\mathbb{R}_+^3:k<k_1, \sigma>\sigma_2, \alpha_1<\alpha<\alpha_2\},~
\mathcal{S}_{51}:=\{(k,\sigma,\alpha)\in\mathbb{R}_+^3:k<k_1, \sigma>\sigma_2, \alpha=\alpha_2\}
\end{array}
\end{eqnarray*}
with $k_2:=(1+h)/h(1-h)$, $k_3:=(h+1)^3/\{h(1-h)(h^2+3 h+1)\}$,
\begin{eqnarray}
\left.
\begin{array}{l}
\alpha_2:=\frac{\{kh(h-1)+h+1\}\{k(h-1)^2+h\sigma+\sigma\}^2}{k^2(1-h)(h\sigma+1-h)^2\{k(h-1)^2+h+\sigma-1\}},\\
\sigma_2:=\frac{(1-h)(h^2k-hk+3h+3)(hk-k+1)+(1-h)(h^2k-hk+h+1)\sqrt{(hk-k+9)(hk-k+1)}}{2\{h(-h^2+2h-1)k-h^2+h+2\}}
\end{array}
\right.
\label{(7)}
\end{eqnarray}
and $\sigma_1$ is the unique positive root of the following function
\begin{eqnarray}
\left.
\begin{array}{l}
f(\sigma):=\{hk(h-1)(h^2+3h+1)k+(h+1)^3\}\sigma^3+\{k(h-1)^2(h(3h+2)(h-1)k+3h^2+2h+2)\}\sigma^2\\
\phantom{f(\sigma):=}+\{k(h-1)^3(kh-k+1)(h^2k-hk-2h+1)\}\sigma+k(h-1)^4(hk-k+1)
\end{array}
\right.
\label{(8)}
\end{eqnarray}
as $0<h<1$ and $k_1<k<k_3$.  The following theorem is devoted to the number of equilibria of system (\ref{(3)}) and
their qualitative properties.
\begin{theorem}
System (\ref{(3)}) has at most four equilibria. The exact number and  qualitative properties of equilibria are described in
Table 1.
\label{thm1}
\end{theorem}
\begin{table}[h]
\tbl{Qualitative properties for various parameters.}
{\begin{tabular}{l l l l}\\[-2pt]
\toprule
$h$& $(k,\sigma,\alpha)$& Number &Equilibrium \\[6pt]
\hline\\[-2pt]
{}&$\mathcal{P}_{11}\cup\mathcal{S}_{21}\cup\mathcal{P}_{31}$& $3$ &$E_0$(saddle)  $E_k$(saddle) $E_1$(unstable focus or node)\\[1pt]
{} & $\mathcal{S}_{11}\cup\mathcal{L}_{21}\cup\mathcal{S}_{31}$
& $3$ &$E_0$(saddle)  $E_k$(saddle) $E_1$(center type)\\[1pt]
{}& $\mathcal{P}_{12}\cup\mathcal{S}_{22}\cup\mathcal{P}_{32}$
&$3$&$E_0$(saddle)  $E_k$(saddle) $E_1$(stable focus or node)\\[1pt]
{} &$\mathcal{S}_1\cup\mathcal{L}_1$
&$2$&$E_0$(saddle)  $E_k$(degenerate)\\[1pt]
$(0, 1)$ & $\mathcal{P}_2\cup\mathcal{S}_3\cup\mathcal{P}_{4}$
&$2$&$E_0$(saddle)  $E_k$(stable node)\\[1pt]
{} &$\mathcal{S}_{41}$
&$3$&$E_0$(saddle)  $E_k$(degenerate) $E_1$(stable focus or node)\\[1pt]
{} &$\mathcal{L}_{41}$
&$3$&$E_0$(saddle)  $E_k$(degenerate)  $E_1$(center type)\\[1pt]
{} &$\mathcal{S}_{42}$
&$3$&$E_0$(saddle)  $E_k$(degenerate)  $E_1$(unstable focus or node)\\[1pt]
{} &$\mathcal{P}_{51}$
&$4$&$E_0$(saddle)  $E_k$(stable node)  $E_1$(unstable focus or node) $E_2$(saddle)\\[1pt]
{} &$\mathcal{S}_{51}$
&$4$&$E_0$(saddle)  $E_k$(stable node)  $E_1$(center type) $E_2$(saddle)\\[1pt]
{} &$\mathcal{P}_{52}$
&$4$&$E_0$(saddle)  $E_k$(stable node)  $E_1$(stable focus or node) $E_2$(saddle)\\[1pt]
{} &$\mathcal{S}_{5}$
&$3$&$E_0$(saddle)  $E_k$(stable node)  $E_*$(degenerate)\\[1pt]
\hline\\[-2pt]
$[1,+\infty)$&$\mathbb{R}_+^3$&2&$E_0$(saddle)  $E_k$(stable node)\\[1pt]
\botrule
\end{tabular}}
\end{table}
\begin{proof}
Equilibria of system (\ref{(3)}) are determined by the algebraic equations
\begin{eqnarray}
\begin{array}{l}
x\{\sigma(k-x)(1+h(1+\alpha y)x)-ky(1+\alpha y)\}=0,~~
ky\{x(1+\alpha y)(1-h)-1\}=0.
\end{array}
\label{(4)}
\end{eqnarray}
For $y=0$, we can find two equilibria $E_0:(0, 0)$ and $E_k:(k, 0)$. For $y>0$, from the second equation in (\ref{(4)}) system (\ref{(3)}) has no other equilibrium if $h\geq1$. If $h<1$, substituting equality $1+hx(1+\alpha y)=x(1+\alpha y)$ into the first equation in (\ref{(4)}), we conclude that all equilibria lie on the curve
\begin{eqnarray}
\left.
\begin{array}{l}
y=\frac{\sigma x(k-x)}{k},~~ 0<x<k.
\end{array}
\right.
\label{(5)}
\end{eqnarray}
Substituting (\ref{(5)}) into the second equation in (\ref{(4)}), we
obtain
\begin{eqnarray}
\left.
\begin{array}{l}
F(x):=\alpha\sigma(h-1)x^3+\alpha\sigma k(1-h)x^2+k(1-h)x-k,
\end{array}
\right.
\label{(6)}
\end{eqnarray}
whose zeros in the interval $(0,k)$ determine all equilibria of system (\ref{(3)}).
The derivative of $F(x)$ is
$F'(x)=(1-h)(-3\alpha\sigma x^2+2\alpha\sigma kx+k)$,
which has a unique positive root
\begin{eqnarray}
\left.
\begin{array}{l}
x_*:=\frac{k\alpha\sigma+\sqrt{k\alpha\sigma(k\alpha\sigma+3)}}{3\alpha\sigma}.
\end{array}
\right.
\label{(10)}
\end{eqnarray}
It is easily seen that $F(x)$ is monotonically increasing for $0<x<x_*$ and monotonically decreasing for $x>x_*$.
We need to discuss the zeros of $F(x)$ in the interval $(0,k)$ for $0<h<1$ in two cases: $x_*\geq k$ and $x_*< k$.
(I). For the case $x_*\geq k$, i.e., $\alpha\leq\frac{1}{\sigma k}$, the discussion is divided into the following two subcases. (I.1) If $F(k)>0$, i.e., $k>k_1$, then $F(x)=0$ has a unique root in the interval $(0, k)$ denoted by $x_1$ (see Fig. \ref{Figure 1} (a)). The corresponding parameters  $(k, \sigma, \alpha)$ locate in $\mathcal{P}_1\cup\mathcal{S}_2$. (I.2) If $F(k)\leq0$, i.e., $k\leq k_1$, then $F(x)=0$ has no root in the interval $(0, k)$. The corresponding parameters $(k, \sigma, \alpha)$ locate in $\mathcal{S}_1\cup\mathcal{P}_2\cup\mathcal{L}_1\cup\mathcal{S}_3$.
(II). For the case $x_*< k$, i.e., $\alpha>\frac{1}{\sigma k}$, we need to discuss in the following two subcases. (II.1) If $F(k)\geq0$, i.e., $k\geq k_1$, then $F(x)=0$ has a unique root in the interval $(0, k)$ denoted by $x_1$ (see Fig. \ref{Figure 1} (a)). The corresponding parameters $(k, \sigma, \alpha)$ locate in  $\mathcal{P}_3\cup\mathcal{S}_4$. (II.2) If $F(k)<0$, i.e., $k<k_1$, it should be clear that we need only
account for the sign of $F(x_*)$ to determine the number of zeros of $F(x)$.
Since $F'(x_*)=0$, we can use Maple command ``prem"  to get the pseudo-remainder of $F(x)$ divided by $F'(x)$ at $x_*$, i.e.,
$$prem(F(x), F'(x), x,`m\mbox{'})=k\alpha^2\sigma^2(1-h)^2\{2(k\alpha\sigma+3)(1-h)x+k(1-h)-9\},$$
where $m=9\alpha^2\sigma^2(h-1)^2$.
Thus, at $x_*$ we have $F(x)=\frac{k}{9}\tilde{F}(x)$  with
\begin{eqnarray}
\left.
\begin{array}{l}
\tilde{F}(x):=2(k\alpha\sigma+3)(1-h)x+k(1-h)-9.
\end{array}
\right.
\label{(0010)}
\end{eqnarray}
Substituting $x_*$ given by (\ref{(10)}) into $\tilde{F}(x)$ leads to
\begin{eqnarray}
\left.
\begin{array}{l}
\tilde{F}(x_*)=\frac{2(k\alpha\sigma+3)(1-h)\sqrt{k\alpha\sigma(k\alpha\sigma+3)}+\alpha\sigma\{2k^2\sigma\alpha(1-h)+9k(1-h)-27\}}{3\alpha\sigma},
\end{array}
\right.
\label{(010)}
\end{eqnarray}
in which the sign of $2(k\alpha\sigma+3)(1-h)\sqrt{k\alpha\sigma(k\alpha\sigma+3)}$ is positive but that of $2k^2\sigma\alpha(1-h)+9k(1-h)-27$ is indeterminate. If $2k^2\sigma\alpha(1-h)+9k(1-h)-27\geq0$, i.e., $\alpha\geq\frac{27-9k(1-h)}{2k^2\sigma(1-h)}$, then  $\tilde{F}(x_*)$ is positive.
If $2k^2\sigma\alpha(1-h)+9k(1-h)-27<0$, i.e., $\frac{1}{\sigma k}<\alpha<\frac{27-9k(1-h)}{2k^2\sigma(1-h)}$, then
the sign of $\tilde{F}(x_*)$  is same as that of
\begin{eqnarray}
\left.
\begin{array}{l}
F_1(\alpha):=4k^2\sigma^2(1-h)\alpha^2+\sigma\{(1-h)^2k^2+18(1-h)k-27\}\alpha+4k(1-h)^2,
\end{array}
\right.
\label{(11)}
\end{eqnarray}
which is deduced from that $2(\alpha k\sigma+3)(1-h)\sqrt{k\alpha\sigma(\alpha k\sigma+3)}$ square minus $\alpha \sigma \{2k^2\sigma\alpha(1-h)+9k(1-h)-27\}$ square.
Since the leading coefficient of $F_1(\alpha)$ and $F_1(\frac{27-9k(1-h)}{2k^2\sigma(1-h)})=\frac{(hk-k+9)^3}{2k^2(1-h)}$
are positive and $F_1(\frac{1}{\sigma k})=\frac{(hk-k+1)(5hk-5k-27)}{k}$ is negative under the conditions $0<h<1$ and $0<k<k_1$, $F_1(\alpha)$ has one root $\alpha_1$ given in (\ref{(07)}) in the interval $(\frac{1}{\sigma k}, \frac{27-9k(1-h)}{2k^2\sigma(1-h)})$.
Hence, we can immediately obtain that $\tilde{F}(x_*)<0$ if $\frac{1}{\sigma k}<\alpha<\alpha_1$,  $\tilde{F}(x_*)>0$ if $\alpha>\alpha_1$ and $\tilde{F}(x_*)=0$ if $\alpha=\alpha_1$.
Accordingly, the distribution of roots of $F(x)$ in the interval $(0, k)$ is displayed as follows.
$F(x)$ has no root in the interval $(0, k)$ if $(k, \sigma, \alpha)\in\mathcal{P}_{4}$;
$F(x)$ has two roots in the interval $(0, k)$ denoted by $x_1$, $x_2$ and $x_1<x_2$ if $(k, \sigma, \alpha)\in\mathcal{P}_{5}$ (see Fig. \ref{Figure 1} (a));
$F(x)$ has one multiple root $x_*$ in the interval $(0, k)$ if $(k, \sigma, \alpha)\in\mathcal{S}_{5}$ (see Fig. \ref{Figure 1} (b)). Furthermore, $x_*$ also can be expressed as
\begin{eqnarray}
\left.
\begin{array}{l}
x_*=\frac{9-k(1-h)}{2(k\alpha_1\sigma+3)(1-h)}
\end{array}
\right.
\label{(012)}
\end{eqnarray}
if $\alpha=\alpha_1$ because $\tilde{F}(x_*)=0$ in (\ref{(0010)}).
Corresponding to the roots of $F(x)$ in the interval $(0, k)$, the positive equilibria of system (\ref{(3)}) are $E_1:(x_1, y_1)$, $E_2:(x_2, y_2)$ or $E_*:(x_*, y_*)$, where $y_i=\frac{\sigma x_i(k-x_i)}{k}$, $i=1, 2, *$.
From the above discussion we obtain the number of equilibria of system (\ref{(3)}) as shown in Table 1.
\begin{figure}[h]
\centering\subfigure[]{\label{fig:subfig:a}
\includegraphics[width=5.5cm,height=5.5cm]{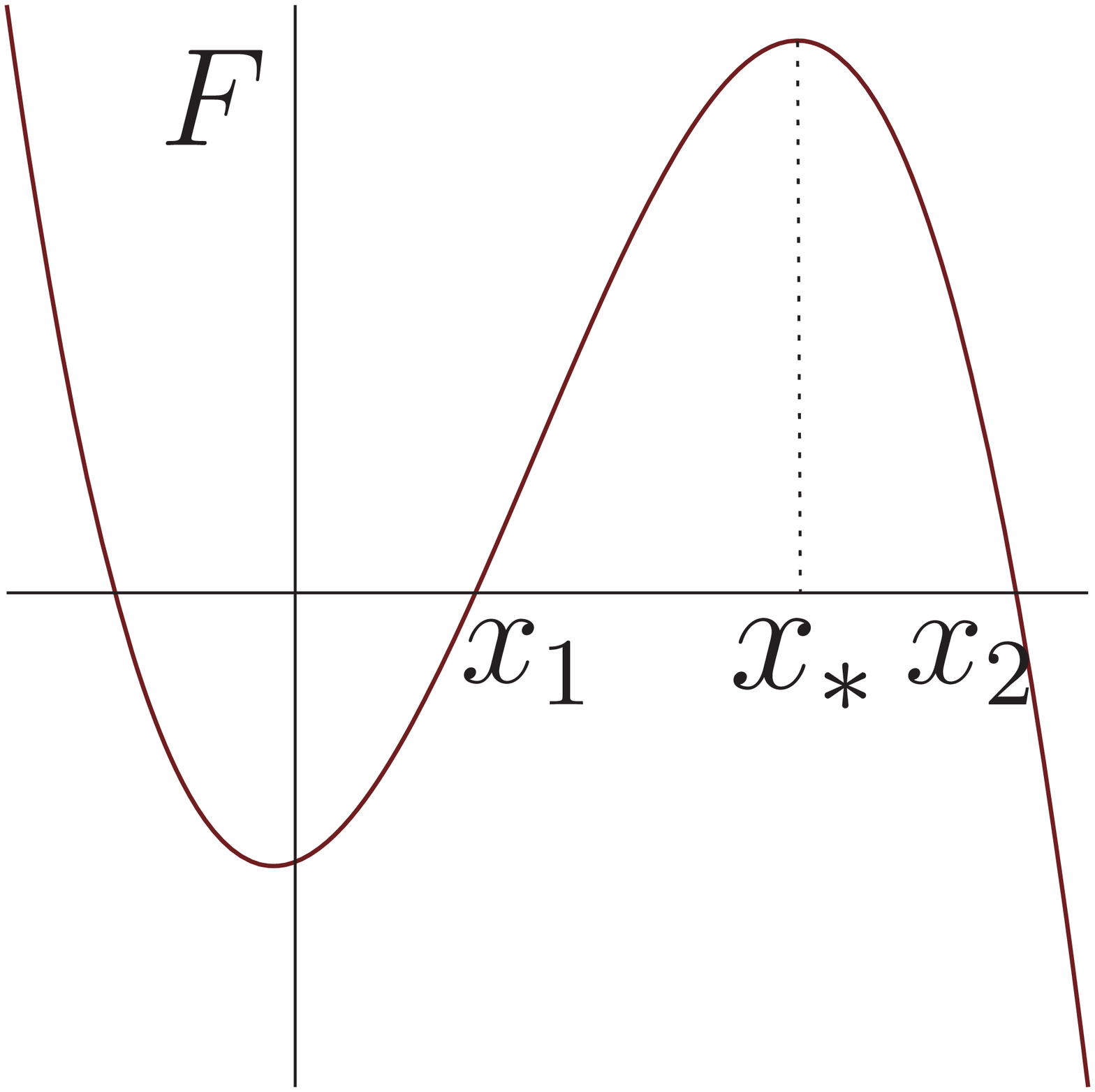}}~~~
\subfigure[]{\label{fig:subfig:b}
\includegraphics[width=5.5cm,height=5.5cm]{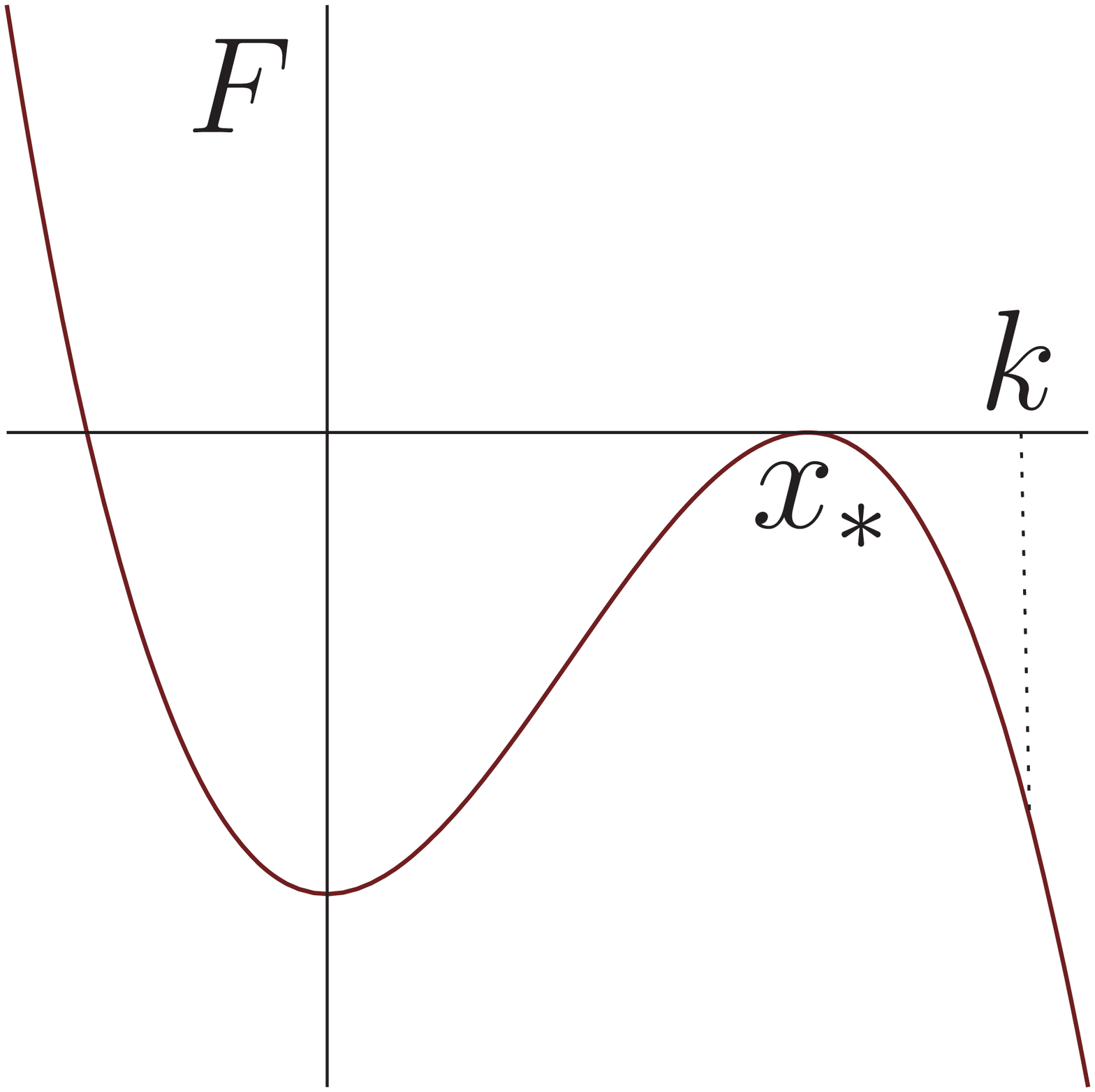}}
\caption{Graphs of the equilibrium equation $F(x)$ when the positive equilibria exist. (a): $F(x)$ has a root $x_1$ for $0<x<k$ when $F(k)>0$ or $k>x_{*}$ and $F(k)=0$; $F(x)$ has two roots $x_1$ and $x_2$ for $0<x<k$ when $k>x_{*}$, $F(k)<0$ and $F(x_{*})>0$.
(b): $F(x)$ has a multiple root $x_{*}$ for $0<x<k$ when $k>x_{*}$, $F(k)<0$ and $F(x_{*})=0$.}
\label{Figure 1}
\end{figure}

In what follows, we study the dynamical behaviors of equilibria. Compute the Jacobian matrix of vector
field (\ref{(3)})
$$
J:=
\left(\begin{array}{lr}J_{11}& x\{(hx(k-x)\sigma-2ky)\alpha-k\}\\
-ky(h-1)(\alpha y+1)  & -k\{x(h-1)(2\alpha y+1)+1\}
\end{array}\right),
$$
where $J_{11}:=\sigma(k-2x)\{1+h(\alpha y+1)x\}+\{h\sigma x(k-x)-yk\}(\alpha y+1)$
and let $T$ and $D$ denote its trace and determinant respectively.  $E_0$ is a saddle because of
$D|_{E_0}=-\sigma k^2<0$. At $E_k$,
$D|_{E_k}=k^2\sigma(hk+1)(hk-k+1)$, $T|_{E_k}=-k\{(\sigma+1)(hk+1)-k\}$,
$T|_{E_k}-4D|_{E_k}=k^2\{(\sigma-1)(hk+1)+k\}^2$.
When $h\geq1$, $D|_{E_k}>0$, $T|_{E_k}<0$ and $T|_{E_k}-4D|_{E_k}>0$, implying that $E_k$ is a stable node.
When $h<1$, the qualitative properties of $E_k$ are displayed as follows.
$D|_{E_k}<0$ if $k>k_1$, implying that $E_k$ is a saddle,
$D|_{E_k}>0$, $T|_{E_k}<0$ and $T|_{E_k}-4D|_{E_k}>0$ if $k<k_1$, implying that $E_k$ is a stable node,
$T|_{E_k}<0$ and $D|_{E_k}=0$ if  $k=k_1$, implying that $E_k$ is  degenerate.
 At $E_i$, $i=1, 2, *$, we obtain determinant $D|_{E_i}$ and trace $T|_{E_i}$ of Jacobian matrix $J$ as follows
\begin{eqnarray}
\left.
\begin{array}{l}
D|_{E_i}=\frac{\sigma(k-x_i)\{\alpha \sigma x_i(k-x_i)+k\}x_i^2}{k}F'(x_i),~~
T|_{E_i}=\frac{(-h^2k+2hk-h\sigma-k-\sigma)x_i+k(h\sigma-h+1)}{1-h}.
\end{array}
\right.
\label{(13)}
\end{eqnarray}
To obtain the afore-given expressions $D|_{E_i}$ and $T|_{E_i}$, we have used the branch  $1+hx_i(1+\alpha y_i)=x_i(1+\alpha y_i)$ and the expression of $y_i$.
Furthermore, $T|_{E_i}$
is the pseudo-remainder of trace $\tilde{T}|_{E_i}$ of Jacobian matrix $J$ at $E_i$ divided by $F(x_i)$, where
$\tilde{T}|_{E_i}:=\sigma x_i\{\alpha\sigma(h+1)x_i^3-\alpha k(2h\sigma-h+\sigma+1)x_i^2+k(\alpha hk\sigma-\alpha hk+\alpha k-h-1)x_i+hk^2\}/k$. Using the MAPLE command ``prem", we can simplify trace $\tilde{T}|_{E_i}$ as
$T|_{E_i}$ since $F(x_i)=0$.
The above discussion of the existence of equilibria shows that $F'(x_1)>0$, $F'(x_*)=0$ and $F'(x_2)<0$.
Thus, we obtain $D|_{E_1}>0$, $D|_{E_*}=0$ and $D|_{E_2}<0$, which imply that $E_*$ is degenerate, $E_2$ is a saddle and
$E_1$ can be neither a saddle nor a degenerate equilibrium. We only need to discuss
the sign of $T|_{E_1}$ in the following lemma.
\begin{lemma}
For $0<h<1$,
$T|_{E_1}>0$ if $(k, \sigma, \alpha)\in\mathcal{P}_{11}\cup\mathcal{S}_{21}\cup\mathcal{P}_{31}\cup\mathcal{S}_{42}\cup\mathcal{P}_{51}$,  $T|_{E_1}<0$ if $(k, \sigma, \alpha)\in\mathcal{P}_{12}\cup\mathcal{S}_{22}\cup\mathcal{P}_{32}\cup\mathcal{S}_{41}\cup\mathcal{P}_{52}$ and  $T|_{E_1}=0$ if $(k, \sigma, \alpha)\in\mathcal{S}_{11}\cup\mathcal{L}_{21}\cup\mathcal{S}_{31}\cup\mathcal{L}_{41}\cup\mathcal{S}_{51}$.
\label{lem1}
\end{lemma}
{\bf Proof of Lemma~\ref{lem1}: }
The above discussion shows that the equilibrium $E_1$ exists for $0<h<1$ and $(k, \sigma, \alpha)\in \mathcal{P}_1\cup\mathcal{S}_2\cup\mathcal{P}_3\cup\mathcal{S}_4\cup\mathcal{P}_5$.
Determining the sign of $T|_{E_1}$ is a difficulty because the explicit solution $x_1$ can not be obtained from equilibrium equation (\ref{(6)}), which is a cubic equation. In order to overcome it, we need to discuss the sign of $T|_{E_1}$ indirectly via the relative position of the roots of equilibrium equation (\ref{(6)}) and $T|_{E_1}$ together with the monotonicity of $T|_{E_1}$.
Function $T|_{E_1}$ is monotonically decreasing and has one positive root.
Let the root of $T|_{E_1}=0$ be
\begin{eqnarray}
\left.
\begin{array}{l}
x_0:={\frac {k \left( h\sigma-h+1 \right) }{k \left( h-1 \right) ^{2}+\sigma\, \left( h+1 \right)}}.
\end{array}
\right.
\label{(013)}
\end{eqnarray}
Substituting $x_0$ into $F(x)$, we get
\begin{eqnarray*}
\left.
\begin{array}{l}
F(x_0)=\frac{k\sigma\{k^2(1-h)(h\sigma-h+1)^2(k(h-1)^2+h+\sigma-1)\alpha+(-h^2k+hk-h-1)(h^2k-2hk+h\sigma+k+\sigma)^2\}}{\{k(h-1)^2+\sigma(h+1)\}^3}.
\end{array}
\right.
\end{eqnarray*}
The concrete strategy is described as follows. For $0<h<1$ and $(k, \sigma, \alpha)\in \mathcal{P}_1\cup\mathcal{S}_2\cup\mathcal{P}_3\cup\mathcal{S}_4$, $F(x)$ has a unique root $x_1$ in the interval $(0, k)$ as well as $F(x)<0$ in the interval $(0, x_1)$ and $F(x)>0$ in the interval $(x_1, k)$.
The relative position of $x_1$ and $x_0$ is determined by the sign of $F(x_0)$ together with relationship $x_0<k$.
Concretely, $x_1>x_0$ if $F(x_0)<0$, $x_1<x_0$ if $F(x_0)>0$ and $x_1=x_0$ if $F(x_0)=0$
(see Fig. \ref{Figure 1} (a)). In addition that $T|_{E_1}$ is monotonically decreasing, then $T|_{E_1}<0$ if $F(x_0)<0$, $T|_{E_1}>0$ if $F(x_0)>0$ and $T|_{E_1}=0$ if $F(x_0)=0$.
For $0<h<1$ and $(k, \sigma, \alpha)\in \mathcal{P}_5$, $F(x)$ has two roots $x_1$ and $x_2$ ($x_1<x_*<x_2$) as well as $F(x)<0$ in the intervals $(0, x_1)\cup(x_2, k)$ and $F(x)>0$ in the interval $(x_1, x_2)$.
The relative position of $x_1$ and $x_0$ is determined by the sign of $F(x_0)$ together with relative position of $x_0$ and $x_*$.
Concretely, $x_1>x_0$ if $F(x_0)<0$ and $x_0<x_*$; $x_1<x_0$ if $F(x_0)>0$ or  $F(x_0)\leq0$ and $x_0>x_*$; $x_1=x_0$ if $F(x_0)=0$ and $x_0<x_*$
(see Fig. \ref{Figure 1} (a)).  In addition that $T|_{E_1}$ is monotonically decreasing, then
$T|_{E_1}<0$ if $F(x_0)<0$ and $x_0<x_*$; $T|_{E_1}>0$ if $F(x_0)>0$ or $F(x_0)\leq0$ and $x_0>x_*$; $T|_{E_1}=0$ if $F(x_0)=0$ and $x_0<x_*$.
Thus, to obtain the parameter condition for each case is the subsequent task.

Because of space cause, we just give the proof in detail for  $0<h<1$ and $(k, \sigma, \alpha)\in \mathcal{S}_4$, but
omit the verbose proof of the rest cases.
By analyzing $F(x_0)$ for $0<h<1$ and $(k,\sigma,\alpha)\in\mathbb{R}_+^3$ we have the sign of $F(x_0)$ as follows.
$F(x_0)>0$ if
$(k,\sigma,\alpha)\in\{(k,\sigma,\alpha)\in\mathbb{R}_+^3:k\geq k_2 ~\mbox{or}~ k_1\leq k<k_2, \alpha>\alpha_2 ~\mbox{or}~ k<k_1, \sigma>1-h-k(1-h)^2, \alpha>\alpha_2\}$;
$F(x_0)<0$ if
$(k,\sigma,\alpha)\in\{(k,\sigma,\alpha)\in\mathbb{R}_+^3:k_1\leq k<k_2, \alpha<\alpha_2~\mbox{or}~k<k_1, \sigma\leq 1-h-k(1-h)^2~\mbox{or}~ k<k_1, \sigma> 1-h-k(1-h)^2, \alpha<\alpha_2\}$;
$F(x_0)=0$ if
$(k,\sigma,\alpha)\in\{(k,\sigma,\alpha)\in\mathbb{R}_+^3:k_1\leq k<k_2, \alpha=\alpha_2~\mbox{or}~k<k_1, \sigma> 1-h-k(1-h)^2, \alpha=\alpha_2\}$,
where $\alpha_2$ is given by (\ref{(7)}).
Furthermore,  $x_0<k$ for $0<h<1$ and $(k, \sigma, \alpha)\in \mathcal{P}_1\cup\mathcal{S}_2\cup\mathcal{P}_3\cup\mathcal{S}_4$ since
$x_0-k=\frac{k\{(1-h)(hk-k+1)-\sigma\}}{k(1-h)^2+\sigma(1+h)}<0$. Now we need to find the intersections of set $\mathcal{S}_4$ and sets of $F(x_0)>0$, $F(x_0)=0$ and $F(x_0)<0$ respectively.
In order to compare the endpoints $\alpha_2$ with $\frac{1}{\sigma k}$, we denote $\alpha_2-\frac{1}{\sigma k}$ by $f(\sigma)$ given in (\ref{(8)}), where
$f(\sigma)=\frac{(1-h)\{(2h+1)\sigma-2h+2\}}{(h\sigma+1-h)^2}>0$
for $0<h<1$ and $(k, \sigma, \alpha)\in \mathcal{S}_4$, implying $\alpha_2>\frac{1}{\sigma k}$.
For $0<h<1$, we can get $F(x_0)>0$ as $(k, \sigma, \alpha)\in\mathcal{S}_{42}$,  $F(x_0)<0$ as $(k, \sigma, \alpha)\in\mathcal{S}_{41}$ and  $F(x_0)=0$ as $(k, \sigma, \alpha)\in\mathcal{L}_{41}$.
Thus, we obtain the corresponding sign of $T|_{E_1}$ for this case.

Although the proof for the case $(k, \sigma, \alpha)\in \mathcal{P}_1\cup\mathcal{S}_2\cup\mathcal{P}_3\cup\mathcal{P}_5$ is omitted, we should account for the two quantities $\sigma_1$ and $\sigma_2$. In the case $(k, \sigma, \alpha)\in \mathcal{P}_1\cup\mathcal{S}_2\cup\mathcal{P}_3$, we still need to compare the endpoints $\alpha_2$ and $\frac{1}{\sigma k}$ so that function $f(\sigma)$ given in (\ref{(8)}) need to be discussed for $(h, k, \sigma)\in\{(h, k, \sigma)\in\mathbb{R}_+^3:h<1, k_1<k<k_2\}$, the properties of which are displayed as follows.
$f(\sigma)>0$ if  $(h, k, \sigma)\in\{(h, k, \sigma)\in\mathbb{R}_+^3:h<1, k_1<k<k_3, \sigma>\sigma_1\}$;
$f(\sigma)<0$ if  $(h, k, \sigma)\in\{(h, k, \sigma)\in\mathbb{R}_+^3:h<1, k_1<k<k_3, \sigma<\sigma_1 ~\mbox{or}~ h<1, k_3\leq k<k_2\}$;
$f(\sigma)=0$ if  $(h, k, \sigma)\in\{(h, k, \sigma)\in\mathbb{R}_+^3:h<1, k_1<k<k_3, \sigma=\sigma_1\}$, where $\sigma_1$ is the unique positive root of $f(\sigma)$ for $0<h<1$ and $k_1<k<k_3$.
In the case $(k, \sigma, \alpha)\in \mathcal{P}_5$, we need to compare the endpoints $\alpha_1$ and $\alpha_2$ for $(h, k, \sigma)\in\{(h, k, \sigma)\in\mathbb{R}_+^3:h<1, k<k_1, \sigma>1-h-k(1-h)^2\}$.
Substituting $\alpha_2$ into $F_1(\alpha)$, we get
$F_1(\alpha_2)=f_1(\sigma)f_2^2(\sigma)/\{k^2(1-h)(h\sigma-h+1)^4\{k(1-h)^2+h+\sigma-1\}^2\}$,
where
\begin{eqnarray*}
\left.
\begin{array}{l}
f_1(\sigma):=\{hk(h-1)+4 h^2+5 h+1\}\sigma^2+(h-1)\{(h-1)^2hk^2+(6h^2-5h-1)k-3h-3\}\sigma\\
\phantom{f_1(\sigma):=}-4k(h-1)^3,\\
f_2(\sigma):=\{-h(h-1)^2k-h^2+h+2\}\sigma^2+(h-1)(h^2k-hk+3h+3)(hk-k+1)\sigma\\
\phantom{f_1(\sigma):=}+k(h-1)^3(hk-k+1).
\end{array}
\right.
\end{eqnarray*}
$f_1(\sigma)>0$ for $\sigma>0$ because all the coefficients  are positive for $0<h<1$ and $0<k<k_1$.
Furthermore, the leading coefficient of $f_2(\sigma)$ is positive and
the constant term of which is negative, implying that $f_2(\sigma)$ is monotonically increasing for $\sigma>0$ and has a unique positive root $\sigma_2$ given in (\ref{(7)}).
Since $f_2(1-h-k(1-h)^2)=-(1-h)^2(kh-k+1)(h^2k-hk+h+1)^2<0$ for $0<h<1$ and $0<k<k_1$,
we have $\sigma_2>1-h-k(1-h)^2$.
Hence, we obtain $\alpha_2=\alpha_1$ for $0<h<1$, $0<k<k_1$ and $\sigma=\sigma_2$ as well as  $\alpha_2>\alpha_1$ for $0<h<1$, $0<k<k_1$, $\sigma>1-h-k(1-h)^2$ and $\sigma\neq\sigma_2$.
The proof of Lemma \ref{lem1} is completed.

The determinant of $E_1$ is positive and  the sign of the trace of $E_1$ is shown in Lemma \ref{lem1},
the qualitative properties of $E_1$ can be derived, namely, $E_1$ is an unstable node or focus if $T|_{E_1}>0$,
$E_1$ is a stable node or focus if $T|_{E_1}<0$ and $E_1$ is center type if $T|_{E_1}=0$.
The stability and topological classification for the equilibria are presented in Table 1.
The proof of Theorem \ref{thm1} is completed.
\end{proof}

\subsection{Bifurcations at $E_k$ and $E_*$}
In this section, we show that both transcritical and pitchfork bifurcations may occur at $E_k$ and
a saddle-node bifurcation may occur at $E_*$.
Table 1 of Theorem \ref{thm1} indicates that system (\ref{(3)}) has a degenerate equilibrium $E_k$ with $T|_{E_k}<0$ and $D|_{E_k}=0$ if $0<h<1$ and
$(k, \sigma, \alpha)\in\mathcal{S}_{1}\cup\mathcal{L}_{1}\cup\mathcal{S}_{41}\cup\mathcal{L}_{41}\cup\mathcal{S}_{42}$, i.e., $k=k_1$. The following theorem displays the bifurcations at $E_k$.
\begin{theorem}
For $0<h<1$ and
$(k, \sigma, \alpha)\in\mathcal{S}_{1}\cup\mathcal{L}_{1}\cup\mathcal{S}_{41}\cup\mathcal{L}_{41}\cup\mathcal{S}_{42}$, $E_k$ is a saddle-node
of system (\ref{(3)}). Moreover,\\
(i) as $(k, \sigma, \alpha)$  crosses $\mathcal{S}_{1}$, i.e., $(k, \sigma, \alpha)$ varies from $\mathcal{P}_{11}\cup\mathcal{P}_{12}$ to $\mathcal{P}_{2}$, a transcritical bifurcation happens at $E_k$ such that
a stable (resp., unstable) node $E_1$ and two saddles $E_0$ and $E_k$ change into a stable node $E_k$ and a saddle $E_0$ for $(k, \sigma, \alpha)\in\mathcal{P}_{12}$ (resp., $(k, \sigma, \alpha)\in\mathcal{P}_{11}$).\\
(ii) as $(k, \sigma, \alpha)$  crosses $\mathcal{S}_{41}$, i.e., $(k, \sigma, \alpha)$ varies from $\mathcal{P}_{32}$ to $\mathcal{P}_{52}$, a transcritical bifurcation happens at $E_k$ such that a stable node $E_1$ and two saddles $E_0$ and $E_k$ change into two stable nodes $E_1$ and $E_k$ and two saddles $E_0$ and $E_2$.\\
(iii) as $(k, \sigma, \alpha)$  crosses $\mathcal{L}_{41}$, i.e., $(k, \sigma, \alpha)$ varies from $\mathcal{S}_{31}$ to $\mathcal{S}_{51}$, a transcritical bifurcation happens at $E_k$ such that
a center type equilibrium $E_1$ and two saddles $E_0$ and $E_k$ change into a center type equilibrium $E_1$, a stable node $E_k$ and two saddles $E_0$ and $E_2$.\\
(iv) as $(k, \sigma, \alpha)$  crosses $\mathcal{S}_{42}$, i.e., $(k, \sigma, \alpha)$ varies from $\mathcal{P}_{31}$ to $\mathcal{P}_{51}$, a transcritical bifurcation happens at $E_k$ such that
two saddles $E_0$ and $E_k$ and an unstable node $E_1$ change into an unstable node $E_1$, a stable node $E_k$ and two saddles $E_0$ and $E_2$.\\
(v) as $(k, \sigma, \alpha)$  crosses $\mathcal{L}_{1}$, i.e., $(k, \sigma, \alpha)$ varies from $\mathcal{S}_{21}\cup\mathcal{S}_{22}$ to $\mathcal{S}_{3}$, a pitchfork bifurcation happens at $E_k$ such that two saddles $E_0$ and $E_k$ and a stable (resp., unstable) node $E_1$ change into a saddle $E_0$ and a stable node $E_k$ for $(k, \sigma, \alpha)\in\mathcal{S}_{22}$ (resp., $(k, \sigma, \alpha)\in\mathcal{S}_{21}$).
\label{thm2}
\end{theorem}
\begin{proof}
Let $\epsilon=k-k_1$. For sufficiently small $|\epsilon|$,
consider system (\ref{(3)}) suspected by the parameter $\epsilon$.
Using the linear transformation $x=u+v+k$, $y=-\sigma u$ and
time-rescaling $\tau:=\frac{-\sigma}{(h-1)^2}t$ to translate $E_k$ to the origin $(0,0)$ and diagonalize the linear part of the suspected system,
we can change the system into the follows
\begin{eqnarray}
\left\{
\begin{array}{l}
\frac{du}{d\tau}=\alpha(h-1)^2u^3+\alpha(h-1)^2u^2v+\frac{(2\alpha\sigma+h-1)(h-1)^2}{\sigma}u^2\epsilon
-\frac{(\alpha\sigma+h-1)(h-1)}{\sigma}u^2+\frac{(h-1)^3}{\sigma}uv\epsilon\\
\phantom{\frac{du}{d\tau}=}-\frac{(h-1)^2}{\sigma}uv+\frac{(h-1)^3}{\sigma}u\epsilon^2-\frac{(h-1)^2}{\sigma}u\epsilon
+O(\parallel(u,v,\epsilon)\parallel^3),\\
\frac{dv}{d\tau}=v-(h^2-1)v\epsilon-(h^2-1)v^2-\frac{(\sigma-1)(h-1)^2}{\sigma}u\epsilon-\frac{(h-1)\{(\sigma-1)(\alpha\sigma+h)+1\}}{\sigma}u^2\\
\phantom{\frac{du}{d\tau}=}-\frac{\alpha h\sigma^2+(2h^2-h-1)\sigma-(h-1)^2}{\sigma}uv+\{((2h-1)\sigma-h+1)\alpha+h(h-1)\}(h-1)u^3\\
\phantom{\frac{du}{d\tau}=}+\{((4h-1)\sigma-h+1)\alpha+3h(h-1)\}(h-1)u^2v+\frac{(\sigma-1)(h-1)^3}{\sigma}u\epsilon^2+h(h-1)^2v^3\\
\phantom{\frac{du}{d\tau}=}+\frac{\{2\sigma(\sigma-1)\alpha+(2h-1)\sigma-h+1\}(h-1)^2}{\sigma}u^2\epsilon
+h\{2\alpha\sigma+3(h-1)\}(h-1)uv^2+2h(h-1)^2v^2\epsilon\\
\phantom{\frac{du}{d\tau}=}+\frac{\{2\alpha h\sigma^2+(4h^2-5h+1)\sigma-(h-1)^2\}(h-1)}{\sigma}uv\epsilon+h(h-1)^2v\epsilon^2
+O(\parallel(u,v,\epsilon)\parallel^3),\\
\frac{d\epsilon}{d\tau}=0.
\end{array}
\right.
\label{(15)}
\end{eqnarray}
By Theorem 1 of \cite{Carr},
system (\ref{(15)}) has a two-dimensional center manifold $W^c: v=h(u,\epsilon)$ near the origin,
which is $C^\infty$ and tangent to the plane $v=0$ at the origin in the $(u,v,\epsilon)$-space. Let
\begin{equation}
v=h(u,\epsilon)=au^2+bu\epsilon+c\epsilon^2+O(\parallel(u,\epsilon)\parallel^3).
\label{(16)}
\end{equation}
Since it is invariant to solutions $(u(t), v(t), \epsilon(t))$ of system (\ref{(15)}),
we can differentiate both sides of (\ref{(16)}),  which leads to the equality
$\dot{v}=h_u\dot{u}+h_\epsilon\dot{\epsilon}$.
Substituting equations of (\ref{(15)}) into the equality
and comparing the coefficients
of $u^2$, $\epsilon^2$ and $u\epsilon$, we get
$a={ \left( \alpha\,{\sigma}^{2}-\alpha\,\sigma+h\sigma-h+1
 \right)  \left( -1+h \right) /\sigma}$, $b=0$ and
 $c={\left( \sigma-1 \right)  \left( -1+h \right) ^{2}/\sigma}$.
Thus,
system (\ref{(15)}) restricted to center manifold (\ref{(16)}) can be written as
\begin{eqnarray}
\begin{array}{l}
\frac{du}{dt}=-{\frac { \left( -1+h \right) ^{2}\epsilon\,u}{\sigma}}-{\frac {
 \left( \alpha\,\sigma+h-1 \right)  \left( -1+h \right) {u}^{2}}{
\sigma}}+{\frac { \left( -1+h \right) ^{3}u{\epsilon}^{2}}{\sigma}}
+c_1{u}^{2}\epsilon+c_2{u}^{3}+O(\parallel(u,\epsilon)\parallel^4),
\end{array}
\label{(17)}
\end{eqnarray}
where $c_1:=(h-1)^2\{2\alpha\sigma^2-(h-1)(h-2)\sigma+(h-1)^2\}/\sigma^2$
and
$c_2:=-(h-1)^2\{\alpha(h-2)\sigma^2-(h-1)(\alpha-h)\sigma-(h-1)^2\}/\sigma^2$.

When $\alpha\neq\frac{1-h}{\sigma}$, it shows that $\frac{(\alpha\sigma+h-1)(1-h)}{\sigma}\neq0$ in (\ref{(17)})
and the origin is the unique equilibrium as $\epsilon=0$ and another equilibrium arises from the origin as $\epsilon\neq0$.
Moreover, the stabilities of the equilibria exchange as $\epsilon$ varies from negative to positive. Thus, $E_k$ is a saddle-node as $\epsilon=0$ and system (\ref{(3)}) undergoes a transcritical bifurcation at $E_k$ for $(k, \sigma, \alpha)\in\mathcal{S}_{1}\cup\mathcal{S}_{41}\cup\mathcal{L}_{41}\cup\mathcal{S}_{42}$ (\cite[p.149]{Guckenheimer}).
When $\alpha=\frac{1-h}{\sigma}$, it shows that $\frac{(\alpha\sigma+h-1)(1-h)}{\sigma}=0$ and $\frac{2(1-h)^3}{\sigma}\neq0$ in (\ref{(17)})
and the origin is the unique equilibrium as $\epsilon=0$ and the other two equilibria arise from the origin as $\epsilon>0$.
Thus, $E_k$ is a saddle-node as $\epsilon=0$ and system (\ref{(3)}) undergoes a pitchfork bifurcation at $E_k$ for $(k, \sigma, \alpha)\in\mathcal{L}_{1}$ (\cite[p.149]{Guckenheimer}).
The proof is completed.
\end{proof}

As indicated in Theorem \ref{thm1}, system (\ref{(3)}) has a degenerate equilibrium $E_*$ for $0<h<1$ and $(k, \sigma, \alpha)\in\mathcal{S}_{5}$, i.e., $D|_{E_*}=0$. To consider what bifurcation system (\ref{(3)}) undergoes for this degenerate case,
let us  first discuss the sign of the trace $T|_{E_*}$ given in  (\ref{(13)}).
Substituting $\alpha=\alpha_1$ and $x_*$ (given in (\ref{(07)}) and (\ref{(012)}) respectively) into $T|_{E_*}$, we obtain
\begin{eqnarray}
\begin{array}{l}
T|_{E_*}=\frac{(hk-k+9)}{8(1-h)^2(\alpha_1k\sigma+3)}\{(h\sigma-h+1)\sqrt{(kh-k+1)(hk-k+9)}\\
\phantom{f_1(\sigma):=}-(kh^2-hk+h+4)\sigma+3(1-h)(hk-k+1)\}.
\end{array}
\label{(18)}
\end{eqnarray}
For $0<h<1$ and $0<k<k_1$, the sign of $(h\sigma-h+1)\sqrt{(kh-k+1)(hk-k+9)}$ is always positive, but that of $-(kh^2-hk+h+4)\sigma+3(1-h)(hk-k+1)$
is indeterminate.
If $-(kh^2-hk+h+4)\sigma+3(1-h)(hk-k+1)\geq0$ in (\ref{(18)}), i.e., $0<\sigma\leq\frac{3(1-h)(hk-k+1)}{kh^2-hk+h+4}$, it is evident that $T|_{E_*}>0$.
If $-(kh^2-hk+h+4)\sigma+3(1-h)(hk-k+1)<0$, i.e., $\sigma>\frac{3(1-h)(hk-k+1)}{kh^2-hk+h+4}$,
we can derive the following relationship
\begin{eqnarray*}
\begin{array}{l}
f_2(\sigma)=-\frac{1}{8}\big\{\{(h\sigma-h+1)\sqrt{(kh-k+1)(hk-k+9)}\}^2\\
\phantom{f_2(\sigma)=}-\{-(kh^2-hk+h+4)\sigma+3(1-h)(hk-k+1)\}^2\big\}.
\end{array}
\end{eqnarray*}
Based on the fact that $f_2(\sigma)>0$ if $\sigma>\sigma_2$, $f_2(\sigma)<0$ if $0<\sigma<\sigma_2$ and $f_2(\sigma)=0$ if $\sigma=\sigma_2$ as well as the following inequality
\begin{eqnarray*}
\begin{array}{l}
f_2(\frac{3(1-h)(hk-k+1)}{kh^2-hk+h+4})=-\frac{2(h-1)^2(hk-k+9)(hk-k+1)(h^2k-hk+h+1)^2}{(h^2k-hk+h+4)^2}<0,
\end{array}
\end{eqnarray*}
we conclude that $T|_{E_*}>0$ if $\frac{3(1-h)(hk-k+1)}{kh^2-hk+h+4}<\sigma<\sigma_2$, $T|_{E_*}<0$ if $\sigma>\sigma_2$ and $T|_{E_*}=0$ if $\sigma=\sigma_2$.
From the discussion, the sign of $T|_{E_*}$ is obtained for $0<h<1$ and $(k, \sigma, \alpha)\in\mathcal{S}_{5}$, namely, $T|_{E_*}>0$ if $0<\sigma<\sigma_2$, $T|_{E_*}<0$ if $\sigma>\sigma_2$ and $T|_{E_*}=0$ if $\sigma=\sigma_2$.

For $0<h<1$ and $(k, \sigma, \alpha)\in\mathcal{S}_{5}$ with $\sigma\neq\sigma_2$, the following theorem displays that system (\ref{(3)}) undergoes a saddle-node bifurcation at $E_*$.
\begin{theorem}
For $0<h<1$ and
$(k, \sigma, \alpha)\in\mathcal{S}_{5}$ with $\sigma\neq\sigma_2$, $E_*$  is a saddle-node
of system (\ref{(3)}) and a saddle-node bifurcation happens at $E_*$ as $(k, \sigma, \alpha)$ crosses $\mathcal{S}_{5}$.
Moreover, as $(k, \sigma, \alpha)$ changes from $\mathcal{P}_{4}$ to $\mathcal{P}_{51}\cup\mathcal{P}_{52}$, an unstable (resp.,
stable) node $E_1$ and a saddle $E_2$ arise for $(k, \sigma, \alpha)\in\mathcal{P}_{51}$ (resp., $(k, \sigma, \alpha)\in\mathcal{P}_{52}$).
\label{thm3}
\end{theorem}
\begin{proof}
Let $\epsilon=\alpha-\alpha_1$. For sufficiently small $|\epsilon|$,
consider system (\ref{(3)}) suspected by the parameter $\epsilon$.
By translating $E_*$ to the origin $(0,0)$ we can expand the suspected system as follows
\begin{eqnarray}
\left\{
\begin{array}{l}
\frac{dx}{dt}=a_{100}x+a_{010}y+a_{001}\epsilon+a_{200}x^2+a_{020}y^2+a_{110}xy+a_{011}\epsilon y+a_{101}\epsilon x+O(\parallel(x,y,\epsilon)\parallel^3),\\
\frac{dy}{dt}=b_{100}x+b_{010}y+b_{001}\epsilon+b_{110}xy+b_{101}x\epsilon+b_{020}y^2+b_{011}y\epsilon+O(\parallel(x,y,\epsilon)\parallel^3),\\
\displaystyle \frac{d\epsilon}{dt}=0,
\end{array}
\right.
\label{(19)}
\end{eqnarray}
where the coefficients $a_{ijk}$ and $b_{ijk}$ are given in the Appendix
with $\alpha_1$ and $x_*$ given in (\ref{(07)}) and (\ref{(012)}) respectively.
Using the linear transformation
$x=u+\frac{a_{100}}{b_{100}}v+\frac{a_{001}(a_{010}-a_{100})}{a_{100}(a_{100}+b_{010})}\epsilon$ and $y=-\frac{b_{100}}{b_{010}}u+v$
to diagonalize the linear part of the suspected system, we obtain the following form
\begin{eqnarray}
\left\{
\begin{array}{l}
\frac{du}{dt}=p_{001}\epsilon+p_{200}u^2+p_{020}v^2+p_{002}\epsilon^2+p_{110}uv+p_{011}\epsilon v+p_{101}\epsilon u+O(\parallel(u,v,\epsilon)\parallel^3),\\
\frac{dv}{dt}=q_{010}v+q_{200}u^2+q_{020}v^2+q_{002}\epsilon^2+q_{110}uv+q_{101}u\epsilon+q_{011}v\epsilon+O(\parallel(u,v,\epsilon)\parallel^3),\\
\displaystyle \frac{d\epsilon}{dt}=0,
\end{array}
\right.
\label{(20)}
\end{eqnarray}
where $p_{ijk}$ and $q_{ijk}$ are displayed in the Appendix.
By Theorem 1 of \cite{Carr}, system (\ref{(20)}) has a two-dimensional center manifold $W^c: v=h_1(u,\epsilon)$ near the origin,
which is $C^\infty$ and
 tangent to the plane $v=0$ at the origin in the $(u,v,\epsilon)$-space. Let
\begin{eqnarray}
v=h_1(u,\epsilon)=c_{20}u^2+c_{02}\epsilon^2+c_{11}u\epsilon+O(\parallel(u,\epsilon)\parallel^3).
\label{(21)}
\end{eqnarray}
Differentiating both sides of (\ref{(21)}) leads to the equality
$\dot{v}=h_{1_u}\dot{u}+h_{1_\epsilon}\dot{\epsilon}$.
Substituting equations of (\ref{(20)}) into the equality
and comparing the coefficients
of $u^2$, $\epsilon^2$ and $u\epsilon$, we obtain $c_{20}$,
$c_{11}$ and $c_{02}$ given in the Appendix respectively.
System (\ref{(20)}) restricted to center manifold (\ref{(21)}) can be written as
\begin{eqnarray}
\begin{array}{l}
\frac{du}{dt}=d_0(\epsilon)+d_1(\epsilon)u+d_2(\epsilon)u^2+O(|u|^3),
\end{array}
\label{(22)}
\end{eqnarray}
where
\begin{eqnarray*}
\begin{array}{l}
d_0(\epsilon)=-\frac{(a_{100}+b_{100}) b_{001} b_{010}}{b_{100} (a_{100}+b_{010})}\epsilon+O(|\epsilon|^2),\\
d_2(\epsilon)=\frac{a_{020} b_{100}^2+a_{100} b_{010} b_{110}-a_{100} b_{020} b_{100}-a_{110} b_{010} b_{100}+a_{200} b_{010}^2}{(a_{100}+b_{010}) b_{010}}+O(|\epsilon|),\\
d_1(\epsilon)=\frac{-1}{b_{100} (a_{100}+b_{010})^2}(a_{011} a_{100} b_{100}^2+a_{011} b_{010} b_{100}^2+a_{100}^2 b_{010} b_{101}-a_{100}^2 b_{011} b_{100}-a_{100} a_{101} b_{010} b_{100}\\
\phantom{d_1(\epsilon)=}+a_{100} b_{001} b_{010} b_{110}-a_{100} b_{001} b_{100} b_{110}+a_{100} b_{010}^2 b_{101}-a_{100} b_{010} b_{011} b_{100}-a_{101} b_{010}^2 b_{100}\\
\phantom{d_1(\epsilon)=}-a_{110} b_{001} b_{010} b_{100}+a_{110} b_{001} b_{100}^2+2 a_{200} b_{001} b_{010}^2-2 a_{200} b_{001} b_{010} b_{100})\epsilon+O(|\epsilon|^2).
\end{array}
\end{eqnarray*}
In the following, we prove $d_2(0)\neq0$ for $0<h<1$ and $(k, \sigma, \alpha)\in\mathcal{S}_{5}$ with $\sigma\neq\sigma_2$, where
\begin{eqnarray*}
\left.
\begin{array}{l}
d_2(0)=\frac{a_{020} b_{100}^2+a_{100} b_{010} b_{110}-a_{100} b_{020} b_{100}-a_{110} b_{010} b_{100}+a_{200} b_{010}^2}{(a_{100}+b_{010}) b_{010}}.
\end{array}
\right.
\end{eqnarray*}
Denote the numerator of $d_2(0)$ by $\tilde{d_2}(0)$, it follows immediately that
$\tilde{d_2}(0)=-x_*^4 \alpha_1^2 \sigma^3 (k-x_*)^2 (h-1)^2f_3(x_*)/k$ with
$f_3(x_*):=-\alpha_1 \sigma (h+2) x_*^3+\alpha_1 k \sigma (h+3) x_*^2-k (\alpha_1 k \sigma-h-2) x_*-k (k-1)$,
which is a cubic polynomial.
Since $F'(x)$ is a quartic polynomial and $F'(x_*)=0$, we use the Maple command ``prem"  to get the pseudo-remainder of $f_3(x_*)$ divided by $F'(x_*)$.
$$prem(f_3(x_*), F'(x_*), x_*,`m\mbox{'})=\alpha_1^2\sigma^2k(h-1)^2\{(\alpha_1k\sigma(2h+1)+6(h+2))x_*+k(h-4)+9\},$$
where $m=9\alpha_1^2\sigma^2(-1+h)^2$.
Substituting $x_*$ given by (\ref{(012)}) into the pseudo-remainder leads to
$f_3(x_*)=\frac{k}{18(k\alpha_1\sigma+3)(h-1)}\{-9k\sigma(hk-k+3)\alpha_1-36(h-1)k-162\}>0$
for $0<h<1$ and $(k, \sigma, \alpha)\in\mathcal{S}_{5}$. Thus, $\tilde{d_2}(0)<0$ for $0<h<1$ and $(k, \sigma, \alpha)\in\mathcal{S}_{5}$.
In the denominator of $d_2(0)$, $b_{010}>0$ is obvious and $a_{100}+b_{010}$, i.e., the trace $T|_{E_*}$, has been discussed before this theorem.
Hence, for $0<h<1$ and $(k, \sigma, \alpha)\in\mathcal{S}_{5}$ with $\sigma\neq\sigma_2$, we have $d_2(0)>0$ if $\sigma>\sigma_2$ and $d_2(0)<0$ if $0<\sigma<\sigma_2$.

Using the translation $u=w-\frac{d_1(\epsilon)}{2d_2(\epsilon)}$ and  time-rescaling $\tau:=d_2(\epsilon)t$ to system (\ref{(22)}), we get
\begin{eqnarray}
\begin{array}{l}
\frac{dw}{d\tau}=\zeta(\epsilon)+w^2+O(|w|^3),
\end{array}
\label{(23)}
\end{eqnarray}
where
$\zeta(\epsilon):=\{4d_0(\epsilon)d_2(\epsilon)-d_1^2(\epsilon)\}/4d_2^2(\epsilon)$.
The computation yields
$\zeta(0)=0$ and $\zeta'(0)=d_0'(0)/d_2(0)$,
where
$d_0'(0)=-(a_{100}+b_{100}) b_{001} b_{010}/b_{100} (a_{100}+b_{010})$.
It is obvious that both $b_{100}$ and $b_{001}$ are positive. In addition,
$a_{100}+b_{100}=\sigma x_*\{\alpha_1\sigma(k-x_*)x_*+k\}(k-2x_*)/k<0$
because $k-2x_*=-\frac{-k\alpha_1\sigma+2\sqrt{k\alpha_1\sigma(\alpha_1k\sigma+3)}}{\alpha_1\sigma}<0$.
Thus, $\zeta'(0)<0$ for $0<h<1$ and $(k, \sigma, \alpha)\in\mathcal{S}_{5}$ with $\sigma\neq\sigma_2$.

Hence,  the origin is the unique equilibrium of (\ref{(23)}) as
$\epsilon=0$ and two equilibria arise from the origin as $\epsilon$ varies from
$0$ to positive when $\sigma\neq\sigma_2$. Therefore, for $0<h<1$ and $(k, \sigma, \alpha)\in\mathcal{S}_{5}$ with $\sigma\neq\sigma_2$, a saddle-node
bifurcation occurs at $E_*$ as $\alpha$ changes from $\alpha=\alpha_1$ to
$\alpha>\alpha_1$ such that an unstable (stable) node $E_1$ and a saddle $E_2$ emerge from $E_*$ if $\sigma<\sigma_2$ (resp. $\sigma>\sigma_2$).
The proof is completed.
\end{proof}

\section{Hopf Bifurcation at $E_1$}
As indicated in Theorem \ref{thm1}, $E_1(x_1, y_1)$ is of center type for $0<h<1$ and $(k, \sigma, \alpha)\in\mathcal{S}_{11}\cup\mathcal{L}_{21}\cup\mathcal{S}_{31}\cup\mathcal{L}_{41}\cup\mathcal{S}_{51}$, i.e., $T|_{E_1}=0$ and $D|_{E_1}>0$,
where $x_1:=x_0$ given in (\ref{(013)}) and
$y_1:=\sigma k\{k(h-1)^2+h+\sigma-1\}(h\sigma-h+1)/\{k(h-1)^2+\sigma(h+1)\}^2$.
In this section, we show that $E_1$ is a weak focus of multiplicity at most 1 and the Hopf bifurcation occurs at $E_1$.
For convenience, let
$\mathcal{D}:=\mathcal{S}_{11}\cup\mathcal{L}_{21}\cup\mathcal{S}_{31}\cup\mathcal{L}_{41}\cup\mathcal{S}_{51}
=\{(k,\sigma,\alpha)\in\mathbb{R}_+^3:k_1\leq k<k_2, \alpha=\alpha_2 ~\mbox{or}~ k<k_1, \sigma>\sigma_2,\alpha=\alpha_2\}$.
\begin{theorem}
For $0<h<1$ and
$(k, \sigma, \alpha)\in\mathcal{D}$, equilibrium $E_1$ of system (\ref{(3)}) is a stable weak focus of multiplicity 1 and
 one stable limit cycle arises near $E_1$ as $\alpha$ varies from $\alpha=\alpha_2$ to $\alpha>\alpha_2$.
\label{thm4}
\end{theorem}
\begin{proof}
Translating equilibrium $E_1$ to the origin, system (\ref{(3)}) becomes the following system
\begin{eqnarray}
\left\{
\begin{array}{l}
\frac{dx}{dt}=a_{10} x+a_{01} y+a_{20} x^2+a_{11} x y+a_{02} y^2+a_{30} x^3+a_{21} x^2 y+a_{12} x y^2+O(\parallel(x,y)\parallel^4),\\
\frac{dy}{dt}=b_{10} x+b_{01} y+b_{11} x y+b_{02} y^2+b_{12} x y^2+O(\parallel(x,y)\parallel^4),
\end{array}
\right.
\label{(24)}
\end{eqnarray}
where the coefficients are given in the Appendix.
For $0<h<1$ and $(k, \sigma, \alpha)\in\mathcal{D}$, system (\ref{(24)}) has a pair of purely imaginary eigenvalues $\pm \beta$,
where
$\beta:=k\sqrt{\sigma f_2(\sigma)}/\{(k(h-1)^2+\sigma(h+1))\sqrt{1-h}\}$.
The transversal condition of Hopf bifurcation holds because
\begin{eqnarray*}
\begin{array}{l}
\frac{d\tilde{T}|_{E_1}}{d\alpha}|_{\alpha=\alpha_2}=\frac{k\sigma x_1^2(k-x_1)(h-1)^2(h\sigma-h+1)}{k(h-1)^2+\sigma(1+h)}>0.
\end{array}
\end{eqnarray*}
In the following we compute the quantity of focus.
Using the linear transformation $x=\frac{1}{b_{10}}u+\frac{a_{10}}{b_{10}\beta}v$, $y=\frac{1}{\beta}v$ and time-rescaling  $\tau:=\beta t$ to normalize the linear part, we can change  system (\ref{(24)}) into the form
\begin{eqnarray}
\left\{
\begin{array}{l}
\frac{du}{d\tau}=-v+f_{20}u^2+f_{02}v^2+f_{11}uv+f_{30}u^3+f_{21}u^2v+f_{12}uv^2+f_{03}v^3+O(\parallel(u,v)\parallel^4),\\
\frac{dv}{d\tau}=u+g_{02}v^2+g_{11}uv+g_{03}v^3+g_{12}uv^2+O(\parallel(u,v)\parallel^4),
\end{array}
\right.
\label{(25)}
\end{eqnarray}
where the coefficients are given in the Appendix.
The following is devoted to the center-focus determination by the successive function method (\cite{Zhang}).
We can obtain the first order focal value
\begin{eqnarray}
\begin{array}{l}
g:=\frac{x_1y_1\{(1+h)\sigma+k(h-1)^2\}G(\sigma)}{8b_{10}\beta^2\sqrt{\beta}(h\sigma-h+1)^3(1-h)^3\{(h+1)\sigma+k(h-1)^2\}},
\end{array}
\label{(28)}
\end{eqnarray}
where $G(\sigma):=L_4\sigma^4+L_3\sigma^3+L_2\sigma^2+L_1\sigma+L_0$ with
\begin{eqnarray*}
\begin{array}{l}
L_4:= h \{h (h-1) (h^2-h+2) k+h^3+3 h+4\} \{h (h-1)^2 k+h^2-h-2\},\\
L_3:=(h-1) \{3 h^3 (h-1)^4 k^3-h^2 (h-3) (4 h^2-3 h-5) (h-1)^2 k^2\\
\phantom{L3:=}-h (h-1) (h+1) (8 h^3-25 h^2+20 h+17) k-(4 h^3-13 h^2+16 h-3) (h+1)^2\},\\
L_2:=(h-1)^3 \{h^2 (3 h-4) (h-1)^2 k^3+h (h-1) (6 h^3+3 h^2-17 h-8) k^2\\
\phantom{L3:=}+(12 h^2-21 h+8) (h+1)^2 k+3 (2 h-1) (h+1)^2\},\\
L_1:=k (h-1)^4 \{h^3 (h-1)^4 k^4+h^2 (8 h+5) (h-1)^3 k^3+h (2 h+1) (11 h+3) (h-1)^2 k^2\\
\phantom{L3:=}+(h-1) (24 h^3+23 h^2+10 h+7) k+(h+1) (9 h^2+2 h+5)\},\\
L_0:=k^2 (h-1)^6 (h k-k+1) \{2 h^2 (h-1)^2 k^2+h (5 h+2) (h-1) k+3 h^2+3 h+2\}.
\end{array}
\end{eqnarray*}
The sign of $g$ is determined by that of $G(\sigma)$.
We first show $G(\sigma)<0$ for $(h, k,\sigma)\in\{(h, k,\sigma)\in\mathbb{R}_+^3:h<1, k_1\leq k<k_2\}$ by proving that all coefficients $L_i$ $(i=0, 1, 2, 3, 4)$ of $G(\sigma)$ are nonpositive.
It is easy to check that $L_4<0$ and $L_0\leq0$.
In fact, the third factor is negative and the other two are positive in $L_4$ and
the third factor is nonpositive and the other three are positive in $L_0$.
To prove $L_1<0$, let the last factor of $L_1$ be $L_{11}(k)$.
Since $L_{11}(k)$ is negative at the both endpoints of the interval $[k_1, k_2)$,  Lemma 3.1 of \cite{YangLu} indicates that the number of the roots for $L_{11}(k)$ in the interval $(k_1, k_2)$ is equal to that of positive roots for
\begin{eqnarray*}
\begin{array}{l}
\Phi(z):=(1+z)^4L_{11}(\frac{hz+h+1}{h(1-h)(1+z)})\\
\phantom{\Phi(z):}=-\frac{1}{h}\{2hz^4+(4h^2+12h+7)z^3+2(4h^2+11h+9)z^2+4(h^2+5h+5)z+8(h+1)\}.
\end{array}
\end{eqnarray*}
It is easily seen that $\Phi(z)$ has no positive root. We thus infer that $L_{11}(k)<0$, implying $L_1=k(h-1)^4L_{11}(k)<0$.
To see $L_2<0$, let the second factor of $L_2$ be $L_{21}(k)$
and the derivative of which be
\begin{eqnarray*}
\begin{array}{l}
L'_{21}(k):=3h^2 (3 h-4) (h-1)^2 k^2+2h (h-1) (6 h^3+3 h^2-17 h-8) k+(12 h^2-21 h+8) (h+1)^2.
\end{array}
\end{eqnarray*}
The facts that the leading coefficient of $L'_{21}(k)$ is negative and $L'_{21}(k)$ is positive at the endpoints of the interval $[k_1, k_2)$ imply that $L'_{21}(k)>0$ for $k_1\leq k<k_2$. Furthermore, since $L_{21}(k_1)=\frac{6h+5}{1-h}>0$, we deduce $L_{21}(k)>0$. Hence, $L_2=(h-1)^3L_{21}(k)<0$.
To show $L_3<0$, let the second factor of $L_3$ be $L_{31}(k)$. Analysis similar to that in the proof of $L_{21}(k)$ shows
that $L_{31}(k)>0$. Hence, $L_3=(h-1)L_{31}(k)<0$.
Consequently, it follows that $G(\sigma)<0$ for $(h, k,\sigma)\in\{(h, k,\sigma)\in\mathbb{R}_+^3:h<1, k_1\leq k<k_2\}$.

We proceed to show $G(\sigma)<0$ for $(h, k,\sigma)\in\{(h, k,\sigma)\in\mathbb{R}_+^3:h<1, k<k_1, \sigma>\sigma_2\}$.
In order to avoid discussing the monotonicity of $G(\sigma)$, we make the transformation $\sigma=\rho+\sigma_2$ to transform the problem of determining the sign of $G(\sigma)$ in the interval $(\sigma_2, +\infty)$ to
the issue of determining the sign of  $\tilde{G}(\rho)$ in the interval $(0, +\infty)$, where
$\tilde{G}(\rho):=\tilde{L}_4\rho^4+\tilde{L}_3\rho^3+\tilde{L}_2\rho^2+\tilde{L}_1\rho+\tilde{L}_0$
with
\begin{eqnarray*}
\begin{array}{l}
\tilde{L}_4:= h \{h (h-1) (h^2-h+2) k+(h+1) (h^2-h+4)\} \{h (h-1)^2 k+(h+1) (h-2)\},\\
\tilde{L}_3:=4 h \{h (h-1) (h^2-h+2) k+(h+1) (h^2-h+4)\} \{h (h-1)^2 k+(h+1) (h-2)\} \sigma_2\\
\phantom{\tilde{L}_3:=}+(h-1) \{3 h^3 (h-1)^4 k^3-h^2 (h-3) (4 h^2-3 h-5) (h-1)^2 k^2-h (h-1) (h+1) (8 h^3-25 h^2\\
\phantom{\tilde{L}_3:=}+20 h+17) k-(4 h^3-13 h^2+16 h-3) (h+1)^2\},\\
\tilde{L}_2:=(h-1) (h^2 k-h k+h+1)^2 \{\big(3 h (h-1) (2 h^2+h+1) k+3 (2 h+3) (h+1)^2\big) \sigma_2+(h-1)^2 \\
\phantom{\tilde{L}_3:=}\times\big((6 h^2-3 h+8) k+6 h-3\big)\},\\
\tilde{L}_1:=-(h-1)^2 (h^2 k-h k+h+1)^3 \{\big(h (4 h^2+5 h-1) (h-1)^2 k^2+(h-1) (16 h^3+44 h^2+17 h-5) k\\
\phantom{\tilde{L}_3:=}+3 (4 h+5) (h+1)^2\big) \sigma_2+k (h-1)^2 \big(h (h-1)^3 k^2+(h-1) (10 h^2+h-5) k+9 h^2+2 h-1\big)\},\\
\tilde{L}_0:=(h-1)^3 (h^2 k-h k+h+1)^4 \{\big(2 h^3 (h-1)^3 k^3+(13 h^3+15 h^2-3 h+1) (h-1)^2 k^2+2 (10 h+1)\\
\phantom{\tilde{L}_3:=}\times (h-1)(h+1)^2 k+9 (h+1)^3\big) \sigma_2+k (h-1)^2 \big((h-1) (3 h^2-2 h+1) k+3 (h+1)^2\big) (h k-k+1)\},
\end{array}
\end{eqnarray*}
in which $\tilde{L}_2$, $\tilde{L}_1$ and $\tilde{L}_0$ were reduced by the pseudo-division since $f_2(\sigma_2)=0$.
Likewise, we show $\tilde{G}(\rho)<0$ by proving that all coefficients $\tilde{L}_i$ $(i=0, 1, 2, 3, 4)$ are negative for $0<h<1$ and $0<k<k_1$.
It follows immediately that $\tilde{L}_4<0$ because the last factor of $\tilde{L}_4$ is negative and the others are positive.
To prove $\tilde{L}_0<0$, let the last factor of $\tilde{L}_0$ be $\tilde{L}_{01}(\sigma_2)$,
the constant term of which is positive.
Let $LC_0(k)$ be the leading coefficient of $\tilde{L}_{01}(\sigma_2)$. The fact that
\begin{eqnarray*}
\begin{array}{l}
\Phi_0(z):=(1+z)^3LC_0(\frac{1}{(1-h)(1+z)})=9(h+1)^3z^3+(7h+25)(h+1)^2z^2+2(6h^2+15h+12)z+8
\end{array}
\end{eqnarray*}
has no positive root shows, by Lemma 3.1 of \cite{YangLu}, that $LC_0(k)$ has no root in the interval $(0, k_1)$.
Since $LC_0(0)>0$, we immediately obtain $LC_0(k)>0$.
Therefore, we have $\tilde{L}_{01}(\sigma_2)>0$, which implies $\tilde{L}_0=(h-1)^3 (h^2 k-h k+h+1)^4\tilde{L}_{01}(\sigma_2)<0$.
In the following, we omit the details of the proof about $\tilde{L}_i<0$ ($i=1, 2, 3$).
We claim that $\tilde{L}_1<0$. In fact, let the last factor of $\tilde{L}_1$ be $\tilde{L}_{11}(\sigma_2)$ and obtain which is positive by analyzing the monotonicity. Since $-(h-1)^2 (h^2 k-h k+h+1)^3<0$, we conclude $\tilde{L}_1=-(h-1)^2 (h^2 k-h k+h+1)^3\tilde{L}_{11}(\sigma_2)<0$.  We claim that $\tilde{L}_2<0$. In fact, let the last factor of $\tilde{L}_2$ be $\tilde{L}_{21}(\sigma_2)$ and obtain $\tilde{L}_{21}(\sigma_2)>0$, which can derive $\tilde{L}_2=(h-1) (h^2 k-h k+h+1)^2\tilde{L}_{21}(\sigma_2)<0$. In the same manner, we can see that $\tilde{L}_3<0$.
Consequently, we can assert that $\tilde{G}(\rho)<0$ for $\rho>0$, namely that $G(\sigma)<0$ for $(h, k,\sigma)\in\{(h, k,\sigma)\in\mathbb{R}_+^3:h<1, k<k_1, \sigma>\sigma_2\}$.

We obtain the desired conclusion that the first order focal value $g$ is negative for $0<h<1$ and
$(k, \sigma, \alpha)\in\mathcal{D}$. Therefore, the equilibrium $E_1$ of system (\ref{(3)}) is a stable weak focus of multiplicity 1 and
at most one stable limit cycle arises near $E_1$ from Hopf bifurcation as $\alpha$ varies from $\alpha=\alpha_2$ to $\alpha>\alpha_2$.
\end{proof}

\section{Bogdanov-Takens bifurcation at $E_*$}
As presented before Theorem \ref{thm3}, $E_*(x_*, y_*)$ is degenerate with $D|_{E_*}=0$ and $T|_{E_*}=0$ for
$0<h<1$ and $(k, \sigma, \alpha)\in\mathcal{S}_{5}$ with $\sigma=\sigma_2$, where
$x_*=\frac{hk\sigma_2+k(1-h)}{k(1-h)^2+\sigma_2(h+1)}$ and $y_*=\frac{\sigma_2 k(h\sigma_2-h+1)\{(h-1)(hk-k+1)+\sigma_2\}}{\{k(1-h)^2+\sigma_2(1+h)\}^2}$. Since $\alpha_1=\alpha_2$ if $\sigma=\sigma_2$, we let $\alpha_*:=\alpha_1=\alpha_2$. In the section, we display that $E_*$ is
a cusp and the Bogdanov-Takens bifurcation may occur at $E_*$.
\begin{lemma}
For $0<h<1$ and
$(k, \sigma, \alpha)\in\mathcal{S}_{5}$ with $\sigma=\sigma_2$, the equilibrium $E_*$ of system (\ref{(3)}) is a cusp.
\label{lem2}
\end{lemma}
\begin{proof}
For $0<h<1$ and $(k, \sigma, \alpha)\in\mathcal{S}_{5}$ with $\sigma=\sigma_2$,  system (\ref{(3)}) can be transformed into the following form
by translating $E_*$ to the origin
\begin{eqnarray}
\left\{
\begin{array}{l}
\frac{dx}{dt}=A_{10} x+A_{01} y+A_{20} x^2+A_{11} x y+A_{02} y^2+O(\parallel(x,y)\parallel^3),\\
\frac{dy}{dt}=B_{10} x+B_{01} y+B_{11} x y+B_{02} y^2+O(\parallel(x,y)\parallel^3),
\end{array}
\right.
\label{(29)}
\end{eqnarray}
where the coefficients are given in the Appendix
with $\sigma_2$ given in (\ref{(7)}).
Using the linear transformation $x=-\frac{B_{01}}{B_{10}}u+v$ and $y=u$ combined with the time-rescaling $\tau:=B_{10}t$ to change system (\ref{(29)})
into the canonical form
\begin{eqnarray}
\left\{
\begin{array}{l}
\frac{du}{d\tau}=v+\mathcal{A}_{20} u^2+\mathcal{A}_{11} u v+O(\parallel(u,v)\parallel^3),\\
\frac{dv}{d\tau}=\mathcal{B}_{20} u^2+\mathcal{B}_{11} u v+\mathcal{B}_{02} v^2+O(\parallel(u,v)\parallel^3),
\end{array}
\right.
\label{(30)}
\end{eqnarray}
where
\begin{eqnarray*}
\begin{array}{l}
\mathcal{A}_{20}:=-\frac{B_{01} B_{11}-B_{02} B_{10}}{B_{10}^2},~~
\mathcal{A}_{11}:=\frac{B_{11}}{B_{10}},~~
\mathcal{B}_{11}:=\frac{A_{11} B_{10}-2 A_{20} B_{01}+B_{01} B_{11}}{B_{10}^2},
\mathcal{B}_{02}:=\frac{A_{20}}{B_{10}},\\
\mathcal{B}_{20}:=\frac{A_{02} B_{10}^2-A_{11} B_{01} B_{10}+A_{20} B_{01}^2-B_{01}^2 B_{11}+B_{01} B_{02} B_{10}}{B_{10}^3}.
\end{array}
\end{eqnarray*}
 By the near-identity transformation $u_1:=u$ and $v_1:=v+\mathcal{A}_{20} u^2+\mathcal{A}_{11} u v\cdots$, system (\ref{(30)}) can be written as the Kukles form
\begin{eqnarray}
\begin{array}{l}
\frac{du_1}{d\tau}=v_1,~
\frac{dv_1}{d\tau}=\mathcal{B}_{20} u_1^2+(2\mathcal{A}_{20}+\mathcal{B}_{11}) u_1 v_1+(\mathcal{A}_{11}+\mathcal{B}_{02}) v_1^2+O(\parallel(u_1,v_1)\parallel^3).
\end{array}
\label{(31)}
\end{eqnarray}
Using a further transformation $u_2:=u_1$ and $v_2:=v_1-(\mathcal{A}_{11}+\mathcal{B}_{02}) u_1v_1$ and the time-rescaling $t:=\{1+(\mathcal{A}_{11}+\mathcal{B}_{02})u_2\}\tau$ to eliminate the term of $v_1^2$ in (\ref{(31)}), the system can be changed into
\begin{eqnarray}
\begin{array}{l}
\frac{du_2}{dt}=v_2,~
\frac{dv_2}{dt}=\mathcal{B}_{20} u_2^2+(2\mathcal{A}_{20}+\mathcal{B}_{11}) u_2 v_2+O(\parallel(u_2,v_2)\parallel^3).
\end{array}
\label{(32)}
\end{eqnarray}
We can assert that the coefficients $\mathcal{B}_{20}$ and $2\mathcal{A}_{20}+\mathcal{B}_{11}$ are nonzero for $0<h<1$ and $0<k<k_1$. In fact,
we can obtain
\begin{eqnarray*}
\begin{array}{l}
\mathcal{B}_{20}=-\frac{\{(h+1) \sigma_2+k (h-1)^2\} \{h (h k-k+1)+1\}^2 \{(-2 h+1) \sigma_2+(h-1) (h k-k+3)\}}{k \{(h-1) (h k-k+1)+\sigma_2\}^3 (h-1) (h \sigma_2-h+1)},\\
2\mathcal{A}_{20}+\mathcal{B}_{11}=\frac{(h+1) \sigma_2+k (h-1)^2}{k \sigma_2 \{k (h-1)^2+h+\sigma_2-1\}^2 (1-h) (h \sigma_2-h+1)}g_1(\sigma_2)
\end{array}
\end{eqnarray*}
with
$g_1(\sigma_2):=\{h (h-1) (h^2+2 h-1) k+(h+1) (h^2+2 h-2)\} \sigma_2^2
+(h-1) \{h (h-1)^3 k^2-(4 h+3) (h-1) k-(h+4) (h+1)\} \sigma_2-k (h-1)^3 (h^2 k+h-k+2)$.
It is easy to check that $\mathcal{B}_{20}<0$. Since the pseudo remainder of $g_1(\sigma_2)$ divided by $f_2(\sigma_2)$ is
$2(1-h)\{h(hk-k+1)+1\}^2\{(h^2k(h-1)+(h+1)^2)\sigma_2+k(h-1)^2\}$, which is positive, we obtain $2\mathcal{A}_{20}+\mathcal{B}_{11}>0$.
Then, by the rescaling
$u_3:=(2\mathcal{A}_{20}+\mathcal{B}_{11})^2u_1/\mathcal{B}_{20}$, $v_3:=-(2\mathcal{A}_{20}+\mathcal{B}_{11})^3v_2/\mathcal{B}_{20}^2 $ and $\tau:=-\mathcal{B}_{20}t/(2\mathcal{A}_{20}+\mathcal{B}_{11})$
system (\ref{(32)}) becomes
\begin{eqnarray}
\begin{array}{l}
\frac{du_3}{d\tau}=v_3,~
\frac{dv_3}{d\tau}=u_3^2-u_3 v_3+O(\parallel(u_3,v_3)\parallel^3).
\end{array}
\label{(33)}
\end{eqnarray}
It follows by Theorem 8.4 of \cite{Kuznetsov} that $E_*$ is a cusp
of system (\ref{(3)}) for $0<h<1$ and $(k, \sigma, \alpha)\in\mathcal{S}_{5}$ with $\sigma=\sigma_2$. The proof of this lemma is completed.
\end{proof}

We proceed to display that the Bogdanov-Takens bifurcation may occur at $E_*$ in the following theorem.  We choose $\sigma$ and $\alpha$ as the
bifurcation parameters and unfold the Bogdanov-Takens normal form of codimension 2 when the parameters $(\sigma, \alpha)$ are perturbed near the point
$(\sigma_2, \alpha_*)$.
\begin{theorem}
For $0<h<1$ and
$(k, \sigma, \alpha)\in\mathcal{S}_{5}$ with $\sigma=\sigma_2$,  there is a neighborhood $U$ of the point $(\sigma_2, \alpha_*)$ in the $(\sigma, \alpha)$-space and
four curves
\begin{eqnarray*}
\begin{array}{l}
\mathcal{SN}^+:=\{(\sigma, \alpha)\in U:\alpha=\alpha_1, \sigma>\sigma_2\},~
\mathcal{SN}^-:=\{(\sigma, \alpha)\in U:\alpha=\alpha_1, \sigma<\sigma_2\},\\
\mathcal{H}:=\{(\sigma, \alpha)\in U:\alpha=\alpha_2, \sigma>\sigma_2\},~
\mathcal{HL}:=\{(\sigma, \alpha)\in U:\alpha=\alpha_3, \sigma>\sigma_2\}
\end{array}
\end{eqnarray*}
such that system (\ref{(3)}) undergoes a saddle-node bifurcation near $E_*$  as $(\sigma,\alpha)$ crossing $\mathcal{SN}^+\cup \mathcal{SN}^-$, a
Hopf bifurcation near  $E_*$  as $(\sigma,\alpha)$ crossing $\mathcal{H}$ and a
homoclinic bifurcation near  $E_*$  as $(\sigma,\alpha)$ crossing $\mathcal{HL}$, where
$
\alpha_3:=\alpha_*-\frac{\mu_{101}}{\mu_{110}}(\sigma-\sigma_2)-\{\frac{A(0,0)(\mu_{101}^2\mu_{120}-\mu_{101}\mu_{110}\mu_{111}+\mu_{102}\mu_{110}^2)}{A(0,0)\mu^3_{110}}+\frac{6(\mu_{101}\mu_{210}-\mu_{110}\mu_{201})^2}{25A(0,0)\mu^3_{110}}
\}(\sigma-\sigma_2)^2+O(|\sigma-\sigma_2|^3)
$
with $A(0,0)=\mathcal{B}_{20}$ and $\mu_{lij}$ displayed in the Appendix.
\label{thm5}
\end{theorem}
\begin{proof}
Let $\epsilon_1:=\alpha-\alpha_*$ and $\epsilon_2:=\sigma-\sigma_2$. For sufficiently small $|\epsilon_1|$ and $|\epsilon_2|$,
we can transform system (\ref{(3)}) into the following form by translating $E_*$ to the origin and using the same translation as (\ref{(30)})
\begin{eqnarray}
\left\{
\begin{array}{l}
\frac{dx}{dt}=E_{00}+E_{10} x+E_{01} y+E_{20} x^2+E_{11} x y+O(\parallel(x,y)\parallel^3),\\
\frac{dy}{dt}=F_{00}+F_{10} x+F_{01} y+F_{20} x^2+F_{11} x y+F_{02} y^2+O(\parallel(x,y)\parallel^3),
\end{array}
\right.
\label{(34)}
\end{eqnarray}
where the coefficients are  given in the Appendix.
 With the change of variables $(x, y)\rightarrow(u_1, v_1)$, where $u_1:=x$ and $v_1$ denotes the right side of
the first equation in (\ref{(34)}), system (\ref{(34)}) can be written as the Kukles form,  whose
second order truncation is the following form
\begin{eqnarray}
\left\{
\begin{array}{l}
\frac{du_1}{dt}=v_1,\\
\frac{dv_1}{dt}=\mathcal{F}_{00}(\epsilon_1, \epsilon_2)+\mathcal{F}_{10}(\epsilon_1, \epsilon_2) u_1+\mathcal{F}_{01}(\epsilon_1, \epsilon_2) v_1+\mathcal{F}_{20}(\epsilon_1, \epsilon_2) u_1^2+\mathcal{F}_{11}(\epsilon_1, \epsilon_2) u_1v_1+\mathcal{F}_{02}(\epsilon_1, \epsilon_2) v_1^2,
\end{array}
\right.
\label{(35)}
\end{eqnarray}
where the coefficients are given in the Appendix.
Since
$\mathcal{F}_{11}(0, 0)=2\mathcal{A}_{20}+\mathcal{B}_{11}>0$,
 we can use a parameter-dependent shift $u_2:=u_1+\frac{\mathcal{F}_{01}(\epsilon_1, \epsilon_2)}{\mathcal{F}_{11}(\epsilon_1, \epsilon_2)}$ and $v_2:=v_1$
to vanish the term proportional to $v_1$ in the second equation of system (\ref{(35)}), which
leads to the following system
\begin{eqnarray}
\left\{
\begin{array}{l}
\frac{du_2}{dt}=v_2,\\
\frac{dv_2}{dt}=\frac{\mathcal{F}_{00}(\epsilon_1, \epsilon_2)\mathcal{F}^2_{11}(\epsilon_1, \epsilon_2)+\mathcal{F}^2_{01}(\epsilon_1, \epsilon_2)\mathcal{F}_{20}(\epsilon_1, \epsilon_2)-\mathcal{F}_{01}(\epsilon_1, \epsilon_2)\mathcal{F}_{10}(\epsilon_1, \epsilon_2)\mathcal{F}_{11}(\epsilon_1, \epsilon_2)}{\mathcal{F}^2_{11}(\epsilon_1, \epsilon_2)}\\
\phantom{\frac{dv_2}{dt}=}
-\frac{2\mathcal{F}_{01}(\epsilon_1, \epsilon_2)\mathcal{F}_{20}(\epsilon_1, \epsilon_2)-\mathcal{F}_{10}(\epsilon_1, \epsilon_2)\mathcal{F}_{11}(\epsilon_1, \epsilon_2)}{\mathcal{F}_{11}(\epsilon_1, \epsilon_2)}u_2+\mathcal{F}_{20}(\epsilon_1, \epsilon_2) u_2^2+\mathcal{F}_{11}(\epsilon_1, \epsilon_2) u_2v_2+\mathcal{F}_{02}(\epsilon_1, \epsilon_2) v_2^2.
\end{array}
\right.
\label{(36)}
\end{eqnarray}
Using the near-identity transformation $u_3:=u_2$, $v_3:=v_2-\mathcal{F}_{02}(\epsilon_1, \epsilon_2)u_2v_2$ and time-rescaling $\tau:=(1+\mathcal{F}_{02}(\epsilon_1, \epsilon_2)u_3)t$, system (\ref{(36)}) can be changed into
\begin{eqnarray}
\begin{array}{l}
\frac{du_3}{d\tau}=v_3,~
\frac{dv_3}{d\tau}=\mu_1(\epsilon_1, \epsilon_2)+\mu_2(\epsilon_1, \epsilon_2)u_3+A(\epsilon_1, \epsilon_2)u_3^2+B(\epsilon_1, \epsilon_2)u_3v_3,
\end{array}
\label{(37)}
\end{eqnarray}
where
\begin{eqnarray*}
\begin{array}{l}
\mu_1(\epsilon_1, \epsilon_2):=\frac{\mathcal{F}_{00}(\epsilon_1, \epsilon_2)\mathcal{F}^2_{11}(\epsilon_1, \epsilon_2)+\mathcal{F}^2_{01}(\epsilon_1, \epsilon_2)\mathcal{F}_{20}(\epsilon_1, \epsilon_2)-\mathcal{F}_{01}(\epsilon_1, \epsilon_2)\mathcal{F}_{10}(\epsilon_1, \epsilon_2)\mathcal{F}_{11}(\epsilon_1, \epsilon_2)}{\mathcal{F}^2_{11}(\epsilon_1, \epsilon_2)},\\
\mu_2(\epsilon_1, \epsilon_2):=-2\mathcal{F}_{02}(\epsilon_1, \epsilon_2)\mu_1(\epsilon_1, \epsilon_2)-\frac{2\mathcal{F}_{01}(\epsilon_1, \epsilon_2)\mathcal{F}_{20}(\epsilon_1, \epsilon_2)-\mathcal{F}_{10}(\epsilon_1, \epsilon_2)\mathcal{F}_{11}(\epsilon_1, \epsilon_2)}{\mathcal{F}_{11}(\epsilon_1, \epsilon_2)},\\
A(\epsilon_1, \epsilon_2):=2\mathcal{F}^2_{02}(\epsilon_1, \epsilon_2)\mu_1(\epsilon_1, \epsilon_2)+\mathcal{F}_{20}(\epsilon_1, \epsilon_2)
+2\mathcal{F}_{02}(\epsilon_1, \epsilon_2)
\frac{2\mathcal{F}_{01}(\epsilon_1, \epsilon_2)\mathcal{F}_{20}(\epsilon_1, \epsilon_2)-\mathcal{F}_{10}(\epsilon_1, \epsilon_2)\mathcal{F}_{11}(\epsilon_1, \epsilon_2)}{\mathcal{F}_{11}(\epsilon_1, \epsilon_2)},\\
B(\epsilon_1, \epsilon_2):=\mathcal{F}_{11}(\epsilon_1, \epsilon_2).
\end{array}
\end{eqnarray*}
We can check that
$A(0, 0)=\mathcal{B}_{20}<0$ and $B(0, 0)=2\mathcal{A}_{20}+\mathcal{B}_{11}>0$.
Thus,  by the rescaling
$u_4:=B^2(\epsilon_1, \epsilon_2)u_3/A(\epsilon_1, \epsilon_2)$, $v_4:=-B^3(\epsilon_1, \epsilon_2)v_3/A^2(\epsilon_1, \epsilon_2)$ and $t:=-A(\epsilon_1, \epsilon_2)\tau/B(\epsilon_1, \epsilon_2)$ system (\ref{(37)}) can be changed into
\begin{eqnarray}
\begin{array}{l}
\frac{du_4}{dt}=v_4,~
\frac{dv_4}{dt}=\beta_1(\epsilon_1, \epsilon_2)+\beta_2(\epsilon_1, \epsilon_2)u_4+u_4^2-u_4v_4,
\end{array}
\label{(38)}
\end{eqnarray}
where
\begin{eqnarray}
\begin{array}{l}
\beta_1(\epsilon_1, \epsilon_2):=\frac{B^4(\epsilon_1, \epsilon_2)}{A^3(\epsilon_1, \epsilon_2)}\mu_1(\epsilon_1, \epsilon_2),~~
\beta_2(\epsilon_1, \epsilon_2):=\frac{B^2(\epsilon_1, \epsilon_2)}{A^2(\epsilon_1, \epsilon_2)}\mu_2(\epsilon_1, \epsilon_2).
\end{array}
\label{(39)}
\end{eqnarray}
Because the coefficients $E_{00}$, $E_{10}$, $F_{00}$, $F_{10}$ and $F_{01}$ in system (\ref{(34)}) are equal to zero if $\epsilon_1=0$ and $\epsilon_2=0$,
we can check $\mu_1(0, 0)=0$ and $\mu_2(0, 0)=0$. Consequently, we conclude that $\beta_1(0, 0)=0$ and $\beta_2(0, 0)=0$.
Moreover, the Jacobian  determinant of (\ref{(39)}) at $(0,0)$ is given by
$$
\begin{vmatrix}
\frac{\partial\beta_1(\epsilon_1, \epsilon_2)}{\partial\epsilon_1}&\frac{\partial\beta_1(\epsilon_1, \epsilon_2)}{\partial\epsilon_2}\\
\frac{\partial\beta_2(\epsilon_1, \epsilon_2)}{\partial\epsilon_1}&\frac{\partial\beta_2(\epsilon_1, \epsilon_2)}{\partial\epsilon_2}
\end{vmatrix}_{(\epsilon_1, \epsilon_2)=(0, 0)}
=\frac{B^6(0,0)}{A^5(0,0)}\frac{k(h-1)^2g_2(k)+g_3(k)\sigma_2}{(-h^3k+2h^2k-h^2-hk+h+2)^4},
$$
where
\begin{eqnarray*}
\begin{array}{l}
g_2(k):=2 h^5 (h-1)^5 k^4+(9 h^6+5 h^5-13 h^4-6 h^3-h^2+3 h-1) (h-1)^3 k^3+(15 h^6+28 h^5-20 h^4-60 h^3\\
\phantom{g_2(k):=}-31 h^2+24 h-8) (h-1)^2 k^2+(h-1) (11 h^3-34 h+1) (h+1)^3 k+3 (h-2) (h+2) (h+1)^4,\\
g_3(k):=2 h^6 (h-1)^6 k^5+2 h^5 (8 h^2+3 h-13) (h-1)^4 k^4+(45 h^7+74 h^6-64 h^5-125 h^4-5 h^3-4 h^2\\
\phantom{g_2(k):=}+4 h-1) (h-1)^3 k^3+(59 h^4-10 h^3-143 h^2+40 h-6) (h-1)^2 (h+1)^3 k^2\\
\phantom{g_2(k):=}+(h-1) (37 h^3-2 h^2-118 h+11) (h+1)^4 k+9 (h-2) (h+2) (h+1)^5.
\end{array}
\end{eqnarray*}
We utilize the theory of complete discrimination system for parametric polynomials in \cite{YangLu} to
determine the number of real roots of $g_2(k)$ and $g_3(k)$ in the interval $k\in(0, k_1)$ with $0<h<1$.
Let $k=\frac{1}{(1-h)(1+x^2)}$, the number of real roots for $g_2(k)$ in the interval $(0, k_1)$ is equal to the half number of that for $\tilde{g}_2(x)$ on the total real axis, where
\begin{eqnarray*}
\begin{array}{l}
\tilde{g}_2(x):=3(h-2)(h+2)(h+1)^4x^8+(h^3+12h^2-14h-49)(h+1)^3x^6+(h^5+19h^4\\
\phantom{g_2(k):=}-6h^3-148h^2-171h-83)x^4+(12h^3-40h^2-54h-66)x^2+4(h-5).
\end{array}
\end{eqnarray*}
The discriminant sequence of $\tilde{g}_2(x)$ is
$D:=\{D_1,D_2,D_3,D_4,D_5,D_6,D_7,D_8\}$,
where
\begin{eqnarray*}
\begin{array}{l}
D_1:=(h-2)^2 (h+2)^2 (h+1)^8,\\
D_2:=(2-h) (h+2) (h+1)^7 (h^3+12 h^2-14 h-49) D_1,\\
D_3:=(h+1)^4 D_2 D_{31},\\
D_4:=(h-2) (h+2) (h+1)^{12} D_1 D_{31} D_{41},\\
D_5:=(h-2) (h+2) (h+1)^{12} D_1 D_{41} D_{51},\\
D_6:=(2-h) (h+2) (h+1)^{12} D_1 D_{61} D_{51},\\
D_7:=(2 h-1)^2 (2-h) (h+2) (h+1)^{12} D_1 D_{61} D_{71},\\
D_8:=(-5+h) (h-2) (h+2) (2 h-1)^4 (h+1)^{12} D_1 D_{71}^2
\end{array}
\end{eqnarray*}
with $D_{31}$, $D_{41}$, $D_{51}$, $D_{61}$ and $D_{71}$ listed in the Appendix.
It is obvious that $\mbox{sign}(D_1)=1$ and $\mbox{sign}(D_2)=-1$ for $0<h<1$. To discuss the sign of $D_i$ ($i=3\cdots8$),  we begin by considering the zeros of the single-variable function
$D_{j}$ ($j=31, 41, 51, 61, 71$). Using the Maple command $``realroot(D_{j}, 0.000001)"$ to isolate the real roots of $D_{j}$ in the interval $(0,1)$. We
see that $D_{31}$ has exactly one real root $h_1$ covered by
$I_1:=[5686495/16777216, 177703/524288]$.
By computing $D_{31}$ at the endpoints of $I_1$, i.e., $D_{31}(5686495/16777216)<0$ and   $D_{31}(177703/524288)>0$, we obtain the sign of $D_{31}$ as follows. $D_{31}<0$ if $h\in(0,h_1)$, $D_{31}=0$ if $h=h_1$ and $D_{31}>0$ if $h\in(h_1,1)$.
$D_{41}$ has exactly one real root $h_2$ covered by
$I_2:=[5448295/16777216, 681037/2097152]$.
Since $D_{41}(5448295/16777216)>0$ and $D_{41}(681037/2097152)<0$, the sign of $D_{41}$ is that $D_{41}>0$ if $h\in(0,h_2)$, $D_{41}=0$ if $h=h_2$ and $D_{41}<0$ if $h\in(h_2,1)$.
$D_{51}$ has three real roots $h_3$, $h_4$ and $h_5$ covered by
$I_3:=[3141071/8388608, 6282143/16777216]$, $I_4:=[4311003/8388608, 1077751/2097152]$ and $I_5:=[6104019/8388608, 1526005/2097152]$ respectively.
Similarly, by computing the sign of $D_{51}$ at the endpoints of the intervals $I_3$, $I_4$ and $I_5$
we obtain $D_{51}<0$ if $h\in(0,h_3)\cup(h_4,h_5)$, $D_{51}=0$ if $h=h_3\cup h_4\cup h_5$ and $D_{51}>0$ if $h\in(h_3,h_4)\cup(h_5,1)$.
$D_{61}$ also has exactly three real roots $h_6$, $h_7$ and $h_8$ covered by
$I_6:=[6273181/16777216, 3136591/8388608]$, $I_7:=[4310379/8388608, 1077595/2097152]$ and $I_8:=[6129165/8388608, 3064583/4194304]$ respectively.
We conclude similarly that $D_{61}>0$ if $h\in(0,h_6)\cup(h_7,h_8)$, $D_{61}=0$ if $h=h_6\cup h_7\cup h_8$ and $D_{61}<0$ if $h\in(h_6,h_7)\cup(h_8,1)$.
$D_{71}$ has exactly three real roots $h_9$, $h_{10}$ and $h_{11}$ covered by
$I_9:=[3160567/8388608, 6321135/16777216]$, $I_{10}:=[4283093/8388608, 2141547/4194304]$ and $I_{11}:=[6960901/8388608, 3480451/4194304]$ respectively.
The computation yields that $D_{71}<0$ if $h\in(0,h_9)\cup(h_{10},h_{11})$, $D_{71}=0$ if $h=h_9\cup h_{10}\cup h_{11}$ and $D_{71}>0$ if $h\in(h_9,h_{10})\cup(h_{11},1)$.
Furthermore, we divide interval $(0,1)$ into 13 open subintervals and 12 single points, which are arranged in order as follows by comparing the endpoints of $I_i$ $(i=1, 2\cdots 11)$
\begin{eqnarray*}
\begin{array}{l}
(0,1)=
(0, h_2)\cup h_2\cup (h_2,h_1)\cup h_1\cup (h_1,h_6)\cup h_6\cup (h_6, h_3)\cup h_3\cup(h_3,h_9)\cup h_9\cup(h_9,\frac{1}{2})\cup \frac{1}{2}\cup(\frac{1}{2},h_{10})\\
\phantom{(0,1)=}\cup h_{10}\cup
(h_{10},h_7)\cup h_7\cup (h_7,h_4)\cup h_4\cup(h_4,h_5)\cup h_5\cup(h_5,h_8)\cup h_8\cup(h_8,h_{11})\cup h_{11}\cup(h_{11},1).
\end{array}
\end{eqnarray*}
Consequently, the sign of $D_i$ ($i=3\cdots8$) is displayed as follows
\begin{eqnarray*}
\begin{array}{l}
D_{3}>0, ~\mbox{if}~h\in(0,h_1);~~D_{3}=0, ~\mbox{if}~h=h_1; ~~D_{3}<0, ~\mbox{if}~h\in(h_1,1);\\
D_{4}>0, ~\mbox{if}~h\in(0,h_2)\cup(h_1,1);~~D_{4}=0, ~\mbox{if}~h=h_2\cup h_1; ~~D_{4}<0, ~\mbox{if}~h\in(h_2,h_1);\\
D_{5}>0, ~\mbox{if}~h\in(0,h_2)\cup(h_3,h_4)\cup(h_5,1);~~D_{5}=0, ~\mbox{if}~h=h_2\cup h_3\cup h_4\cup h_5;\\
D_{5}<0, ~\mbox{if}~h\in(h_2,h_3)\cup(h_4,h_5);
~~D_{6}<0, ~\mbox{if}~h\in(0,h_6)\cup(h_3,h_7)\cup(h_4,h_5)\cup(h_8,1);\\
D_{6}=0, ~\mbox{if}~h=h_6\cup h_3\cup h_7\cup h_4\cup h_5\cup h_8; ~~D_{6}>0, ~\mbox{if}~h\in(h_6,h_3)\cup(h_7,h_4)\cup(h_5,h_8);\\
D_{7}<0, ~\mbox{if}~h\in(0,h_6)\cup(h_9,\frac{1}{2})\cup(\frac{1}{2},h_{10})\cup(h_7,h_8)\cup(h_{11},1);\\
D_{7}=0, ~\mbox{if}~h=h_6\cup h_9\cup \frac{1}{2}\cup h_{10}\cup h_7\cup h_8\cup h_{11};~~
D_{7}>0, ~\mbox{if}~h\in(h_6,h_9)\cup(h_{10},h_7)\cup(h_8,h_{11});\\
D_{8}>0, ~\mbox{if}~h\in(0,h_9)\cup(h_9,\frac{1}{2})\cup(\frac{1}{2},h_{10})\cup(h_{10},h_{11})\cup(h_{11},1);~~
D_{8}=0, ~\mbox{if}~h=h_9\cup \frac{1}{2}\cup h_{10}\cup h_{11}.
\end{array}
\end{eqnarray*}
The sign lists of the discriminant sequence $D$ are given as follows
\begin{alignat*}{4}
&[1,-1,1,1,1,-1,-1,1], ~&h&\in(0, h_2),\\
&[1,-1,1,0,0,-1,-1,1], ~&h&=h_2,\\
&[1,-1,1,-1,-1,-1,-1,1], ~&h&\in(h_2,h_1),\\
&[1,-1,0,0,-1,-1,-1,1], ~&h&=h_1,\\
&[1,-1,-1,1,1,0,0,1], ~&h&\in h_7\cup h_8,\\
&[1,-1,-1,1,-1,0,0,1], ~&h&=h_6,\\
&[1,-1,-1,1,-1,1,1,1], ~&h&\in(h_6,h_3),\\
&[1,-1,-1,1,0,0,1,1], ~&h&=h_3,\\
&[1,-1,-1,1,1,-1,1,1], ~&h&\in(h_3,h_9)\cup(h_{10},h_7)\cup(h_8,h_{11}),\\
&[1,-1,-1,1,1,-1,0,0], ~&h&\in h_9\cup\frac{1}{2}\cup h_{10}\cup h_{11},\\
&[1,-1,-1,1,1,-1,-1,1], ~&h&\in(h_9,\frac{1}{2})\cup(\frac{1}{2},h_{10})\cup(h_{11},1),\\
&[1,-1,-1,1,-1,-1,-1,1], ~&h&\in(h_1,h_6)\cup(h_4,h_5),\\
&[1,-1,-1,1,1,1,-1,1], ~&h&\in(h_7,h_4)\cup(h_5,h_8),\\
&[1,-1,-1,1,0,0,-1,1], ~&h&\in h_4\cup h_5.
\end{alignat*}
For $h\in h_2\cup h_1\cup h_6\cup h_3\cup h_7\cup h_8\cup h_4\cup h_5$, the revised sign lists are
\begin{alignat*}{4}
&[1,-1,1,-1,-1,-1,-1,1], ~&h&=h_2,~&&[1,-1,1,1,-1,-1,-1,1], \quad&h&=h_1,\\
&[1,-1,-1,1,-1,1,1,1], ~&h&=h_6,~&&[1,-1,-1,1,-1,-1,1,1], ~&h&=h_3,\\
&[1,-1,-1,1,1,-1,-1,1], ~&h&\in h_7\cup h_8,~&&[1,-1,-1,1,-1,-1,-1,1],~&h&\in h_4\cup h_5.
\end{alignat*}
Thus, the change number of the sign lists of the
discriminant sequence $D$ is 4 and the number
of the non-vanishing numbers of these lists is 8 for $h\in(0,1)/(h_9\cup\frac{1}{2}\cup h_{10}\cup h_{11})$
and the change number is 3 and the number
of the non-vanishing numbers is 6 for $h\in h_9\cup\frac{1}{2}\cup h_{10}\cup h_{11}$. By Theorem 2.1 in \cite{YangLu}, $\tilde{g}_2(x)$ has no root on the total real axis. Hence,  $g_2(k)$ has no real root in the interval $(0, k_1)$. In addition,  $g_2(k)$ is an even function and  $g_2(0)<0$, then we obtain $g_2(k)<0$ for $0<k<k_1$ and $0<h<1$.

Similarly, by using the complete discrimination system of polynomial we can also obtain $g_3(k)<0$ for $0<k<k_1$ and $0<h<1$.
Hence, we conclude that the Jacobian determinant is nonzero, i.e.,  (\ref{(39)}) is locally invertible.
System (\ref{(38)}) therefore is locally equivalent to the universal unfolding system
\begin{eqnarray}
\begin{array}{l}
\frac{dx}{dt}=y,~~
\frac{dy}{dt}=\xi_1+\xi_2x+x^2-xy.
\end{array}
\label{(40)}
\end{eqnarray}
As indicated  in Section 8.4 of \cite{Kuznetsov},  system (\ref{(40)}) undergoes a saddle-node bifurcation as $(\xi_1,\xi_2)$ crossing $\mathcal{SN}^+\cup\mathcal{SN}^-$, where $\mathcal{SN}^+:=\{(\xi_1,\xi_2)\in U:\xi_1=\xi_2^2/4, \xi_2>0\}$ and $\mathcal{SN}^-:=\{(\xi_1,\xi_2)\in U:\xi_1=\xi_2^2/4, \xi_2<0\}$, a Hopf bifurcation as $(\xi_1,\xi_2)$ crossing $\mathcal{H}$, where $\mathcal{H}:=\{(\xi_1,\xi_2)\in U:\xi_1=0, \xi_2<0\}$ and a homoclinic bifurcation as $(\xi_1,\xi_2)$ crossing $\mathcal{HL}$, where $\mathcal{HL}:=\{(\xi_1,\xi_2)\in U:\xi_1=-6\xi_2^2/25+O(|\xi_2|^3), \xi_2<0\}$.

In what follows, we only need to present the bifurcation curve $\mathcal{HL}$ in terms of $\epsilon_1$ and $\epsilon_2$ because the bifurcation curves $\mathcal{SN}$ and $\mathcal{H}$ have already been shown in Theorem \ref{thm3} and Theorem \ref{thm4}.
For convenience, we denote
$\mu_{lij}:=\partial^{i+j}\mu_l(0,0)/\partial^i\epsilon_1\partial^j\epsilon_2$,
$l=1,2$ and $i,j=0,1,2$  given in the Appendix.
We can solve $\epsilon_1$ and $\epsilon_2$ from (\ref{(39)}) as follows
\begin{eqnarray*}
\begin{array}{l}
\epsilon_1=\frac{A^2(0,0)}{B^2(0,0)}\frac{(-h^3k+2h^2k-h^2-hk+h+2)^4}{k(h-1)^2g_2(k)+g_3(k)\sigma_2}
(\frac{A(0,0)}{B^2(0,0)}\mu_{202}\beta_1-\mu_{102}\beta_2)
+O(\parallel(\beta_1,\beta_2)\parallel^2),\\
\epsilon_2=\frac{A^2(0,0)}{B^2(0,0)}\frac{(-h^3k+2h^2k-h^2-hk+h+2)^4}{k(h-1)^2g_2(k)+g_3(k)\sigma_2}
(-\frac{A(0,0)}{B^2(0,0)}\mu_{201}\beta_1+\mu_{101}\beta_2)
+O(\parallel(\beta_1,\beta_2)\parallel^2).
\end{array}
\end{eqnarray*}
Before expressing the bifurcation curve we need to prove $\mu_{110}\neq0$ for $0<h<1$ and $0<k<k_1$, where $\mu_{110}$ is given in the Appendix and has the same sign as
$\tilde{\mu}_{110}:=\zeta_1(k)\sigma_2+\zeta_2(k)$
with
\begin{eqnarray*}
\begin{array}{l}
\zeta_1(k):=2h^4(h-1)^3k^3+(7h^4+12h^3+6h^2-4h+1)(h-1)^2k^2+8h(h-1)(h+1)^3k+3(h+1)^4,\\
\zeta_2(k):=k(h-1)^2\{(h-1)(h^3+3h^2-3h+1)k+(h+1)^3\}(hk-k+1).
\end{array}
\end{eqnarray*}
It is easy to obtain $\zeta_2(k)>0$ for $0<h<1$ and $0<k<k_1$. We next prove $\zeta_1(k)>0$ for $0<h<1$ and $0<k<k_1$.
Let $k=\frac{1}{(1-h)(1+x^2)}$, the number of real roots for $\zeta_1(k)$ in the interval $(0, k_1)$ is equal to the half number of that for $\tilde{\zeta}_1(k)$ on the total real axis (\cite{YangLu}), where
$\tilde{\zeta}_1(k):=3(h+1)^4x^6+(h+9)(h+1)^3x^4+2(6h^2+8h+5)x^2+4$.
Obviously, $\tilde{\zeta}_1(k)$ has no real root, implying that $\zeta_1(k)$ has no root in the interval $(0, k_1)$. In addition, $\zeta_1(k)$ is an odd function and $\zeta_1(0)>0$, we obtain $\zeta_1(k)>0$ for $0<h<1$ and $0<k<k_1$.

For the bifurcation curve $\mathcal{HL}$, we consider $\Xi:=\beta_1+\frac{6}{25}\beta_2^2+O(|\beta_2|^3)$.
Since $\frac{\partial\Xi}{\partial\epsilon_1}=\frac{B^4(0,0)}{A^3(0,0)}\mu_{110}\neq0$,
by the implicit function theorem,
there exists a unique function $\epsilon_1(\epsilon_2)$ such that $\epsilon_1(0)=0$ and $\Xi(\epsilon_1(\epsilon_2),\epsilon_2)=0$, which can be obtained as an expansion
\begin{eqnarray*}
\begin{array}{l}
\epsilon_1(\epsilon_2)=-\frac{\mu_{101}}{\mu_{110}}\epsilon_2
-\frac{25A(0,0)(\mu_{101}^2\mu_{120}-\mu_{101}\mu_{110}\mu_{111}+\mu_{102}\mu_{110}^2)+6(\mu_{101}\mu_{210}-\mu_{110}\mu_{201})^2}{25A(0,0)\mu^3_{110}}\epsilon_2^2+O(|\epsilon_2|^3).
\end{array}
\end{eqnarray*}
Further, on the curve $\Xi=0$, we have
$\epsilon_2=-\frac{A^2(0,0)\mu_{110}}{B^2(0,0)(\mu_{101}\mu_{210}-\mu_{110}\mu_{201})}\beta_2+O(|\beta_2|^2)$,
in which  the coefficient of $\beta_2$ is negative, implying that
 $\epsilon_2>0$ if $\beta_2<0$ and $\epsilon_2<0$ if $\beta_2>0$. Therefore, we obtain
\begin{eqnarray*}
\begin{array}{l}
\mathcal{HL}:=\{(\epsilon_1, \epsilon_2)\in U:\epsilon_1=-\frac{\mu_{101}}{\mu_{110}}\epsilon_2-\{\frac{A(0,0)(\mu_{101}^2\mu_{120}-\mu_{101}\mu_{110}\mu_{111}+\mu_{102}\mu_{110}^2)}{A(0,0)\mu^3_{110}}+\frac{6(\mu_{101}\mu_{210}-\mu_{110}\mu_{201})^2}{25A(0,0)\mu^3_{110}}
\}\epsilon_2^2\\
\phantom{\mathcal{HL}:=}+O(|\epsilon_2|^3), \epsilon_2>0\}.
\end{array}
\end{eqnarray*}
With the transformation $\epsilon_1=\alpha-\alpha_1$ and $\epsilon_2=\sigma-\sigma_2$, we can rewrite the above bifurcation curve $\mathcal{HL}$ as in Theorem \ref{thm5}. The proof of this theorem is completed.
\end{proof}

\section{Numerical simulation and discussion}
In this paper we qualitatively investigate prey-predator system (\ref{(2)}) with foraging facilitation among predators including the number and properties of the equilibria (Theorem \ref{thm1}) as well as the bifurcations of equilibria such as transcritical and pitchfork bifurcations (Theorem \ref{thm2}), saddle-node bifurcation (Theorem \ref{thm3}), Hopf bifurcation (Theorem \ref{thm4}) and Bogdanov-Takens bifurcation (Theorem \ref{thm5}).
In spite that both the saddle-node and Hopf bifurcations are discussed above, they are also exhibited in the Bogdanov-Takens bifurcation.
As indicated in Theorem \ref{thm5}, the neighborhood $U$ of point $(\sigma_2,\alpha_*)$ is divided into four regions, i.e., $U=\mathcal{SN}^+\cup\mathcal{SN}^-\cup\mathcal{H}\cup\mathcal{HL}\cup\mathcal{R}_1\cup\mathcal{R}_2\cup\mathcal{R}_3\cup\mathcal{R}_4$, where
\begin{eqnarray*}
\begin{array}{l}
\mathcal{R}_1:=\{(\sigma, \alpha)\in U: \alpha<\alpha_1\},\\
\mathcal{R}_2:=\{(\sigma, \alpha)\in U:
\alpha_1<\alpha<
\alpha_2, \sigma>\sigma_2
\},\\
\mathcal{R}_3:=\{(\sigma, \alpha)\in U:
\alpha_2<\alpha<
\alpha_3, \sigma>\sigma_2
\},\\
\mathcal{R}_4:=\{(\sigma, \alpha)\in U:
\alpha>\alpha_1, \sigma\leq\sigma_2
\}
\cup\{(\sigma, \alpha)\in U:
\alpha>\alpha_3, \sigma>\sigma_2
\}.
\end{array}
\end{eqnarray*}
Accordingly, the dynamical behaviors of system (\ref{(3)}) near the cusp $E_*$ for parameters in neighborhood $U$ of point $(\sigma_2,\alpha_*)$ are listed in Table 2.
\begin{table}[h]
\tbl{Dynamical behaviors near $E_*$.}
{\begin{tabular}{l l l l}\\[-2pt]
\toprule
$(h, k)$ & $(\sigma,\alpha)$ & Equilibria and properties  & Closed orbits and homoclinic orbits \\[6pt]
\hline\\[-2pt]
{} & $\mathcal{R}_{1}$ & No equilibria  & No \\[1pt]
{} & $\mathcal{SN}^+\cup\mathcal{SN}^-$
& $E_*$(saddle node) & No \\[1pt]
{} & $\mathcal{R}_{2}$
& $E_1$(stable focus or node) $E_2$(saddle) & No \\[1pt]
$(0, 1)\times(0,k_1)$ & $\mathcal{H}$
& $E_1$(stable weak focus) $E_2$(saddle) & No\\[1pt]
{} &$\mathcal{R}_3$
& $E_1$(unstable focus)$E_2$(saddle) & A stable limit cycle\\[1pt]
{} & $\mathcal{HL}$
& $E_1$(unstable focus) $E_2$(saddle) & A homoclinic orbit\\[1pt]
{} & $\mathcal{R}_{4}$
& $E_1$(unstable focus or node) $E_2$(saddle) & No \\[1pt]
{} & $(\sigma_2,\alpha_*)$
& $E_*$(cusp) & No\\[1pt]
\botrule
\end{tabular}}
\end{table}
We next offer some examples to demonstrate the dynamical behaviors of system (\ref{(3)}). Let $h=0.5$ and $k=1$, we have $\sigma_2=0.55$, $\alpha_*=15.94$ and $E_*=(0.72, 0.11)$. The bifurcation diagram of the Bogdanov-Takens bifurcation is displayed in Fig. \ref{Figure 2}.
\begin{figure}[h]
\begin{center}
\includegraphics[width=8cm,height=6cm]{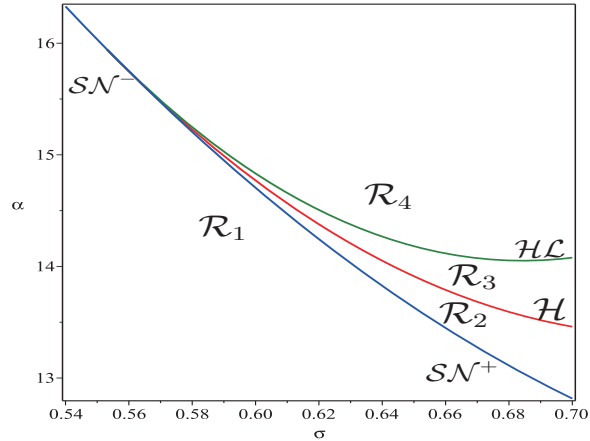}\
\caption { The bifurcation diagram of Bogdanov-Takens bifurcation. }\label{Figure 2}
\end{center}
\end{figure}
When $(\sigma, \alpha)=(0.62, 14.2)\in\mathcal{R}_1$, system (\ref{(3)}) has no equilibrium except the saddle $E_0$ and the stable node $E_k$ (Fig. \ref{Figure 3} (a)). For $(\sigma, \alpha)=(0.62, 14.3)\in\mathcal{R}_2$, system (\ref{(3)}) has four equilibria, i.e., the stable node $E_k$, the stable focus $E_1$ and the saddles $E_0$ and $E_2$ as shown in Fig. \ref{Figure 3} (b). The rise of equilibria $E_1$ and $E_2$ is due to the saddle-node bifurcation.
When $(\sigma, \alpha)=(0.62, 14.42)\in\mathcal{R}_3$, system (\ref{(3)}) has a stable limit cycle and four equilibria, i.e., the stable node $E_k$, the unstable focus $E_1$ and the saddles $E_0$ and $E_2$ as shown in Fig. \ref{Figure 3} (c). The rise of the limit cycle is induced by the Hopf bifurcation.
For $(\sigma, \alpha)=(0.62, 14.55)\in\mathcal{R}_4$, system (\ref{(3)}) only has four equilibria, i.e., the stable node $E_k$, the unstable focus $E_1$ and the saddles $E_0$ and $E_2$ as shown in Fig. \ref{Figure 3} (d). The disappearance of the limit cycle is induced by the homoclinic bifurcation.
Theorem \ref{thm4} describes that one stable limit cycle arises near $E_1$ induced by the Hopf bifurcation as $0<h<1$ and $(k, \sigma, \alpha)$ varies from $\mathcal{S}_{11}\cup\mathcal{L}_{21}\cup\mathcal{S}_{31}$ to $\mathcal{P}_{11}\cup\mathcal{S}_{21}\cup\mathcal{P}_{31}$.
Table 1 shows that system (\ref{(3)}) has two saddles $E_0$ and $E_k$ and an unstable focus or node $E_1$ when $0<h<1$ and $(k, \sigma, \alpha)\in\mathcal{P}_{11}\cup\mathcal{S}_{21}\cup\mathcal{P}_{31}$. Thus, the stable limit cycle is the $\omega$-limit set of the positive solutions. For example, let $h=0.5$ and $(k, \sigma, \alpha)=(5.5, 1, 0.1)\in\mathcal{P}_{11}$, system (\ref{(3)}) has  two saddles $E_0$ and $E_k$ and an unstable focus or node $E_1$ surrounded by a stable limit cycle as shown in Fig. \ref{Figure 4}.
\begin{figure}[h]
\centering\subfigure[]{\label{fig:subfig:a}
\includegraphics[width=5.5cm,height=5.5cm]{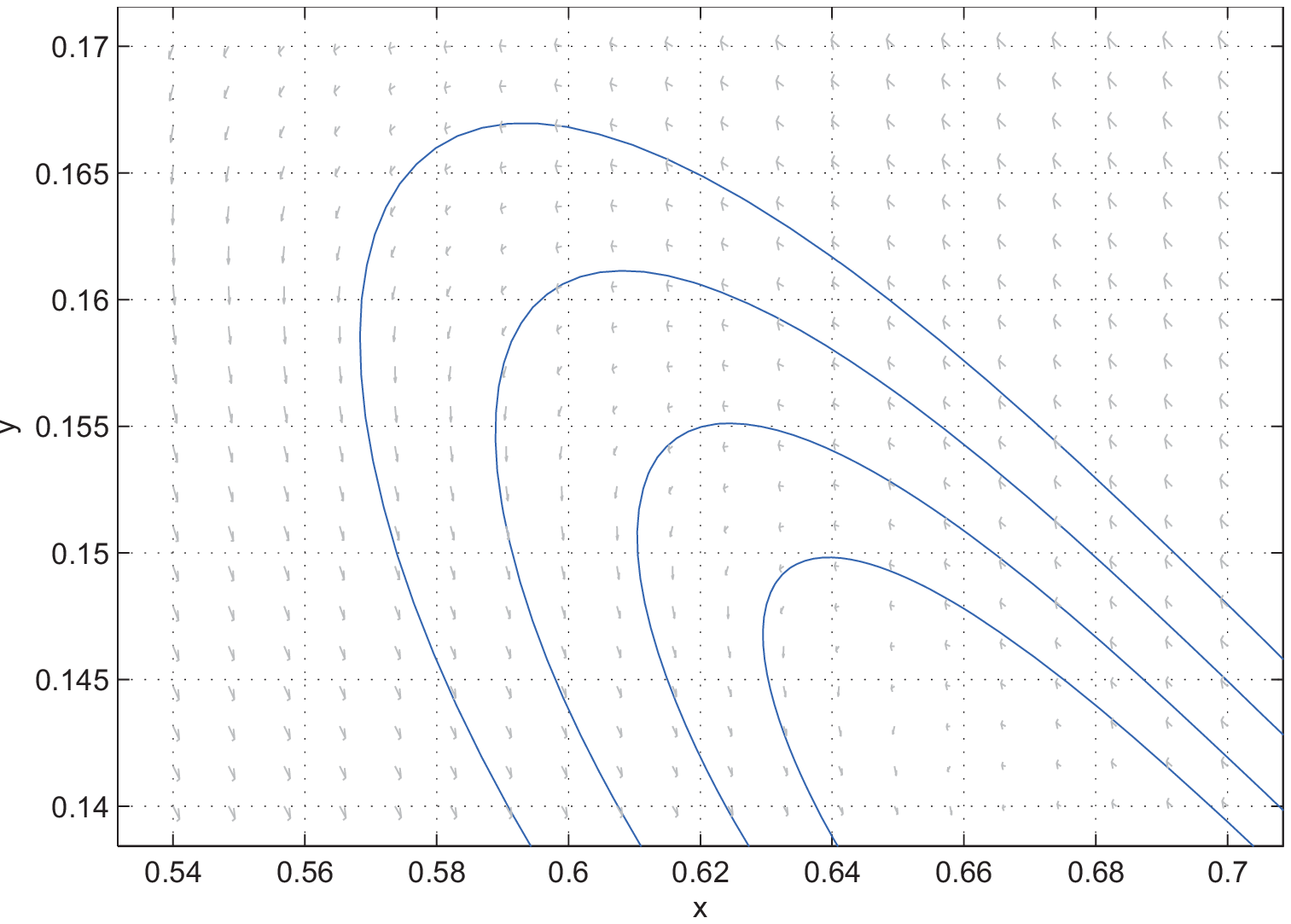}}
\subfigure[]{\label{fig:subfig:b}
\includegraphics[width=5.5cm,height=5.5cm]{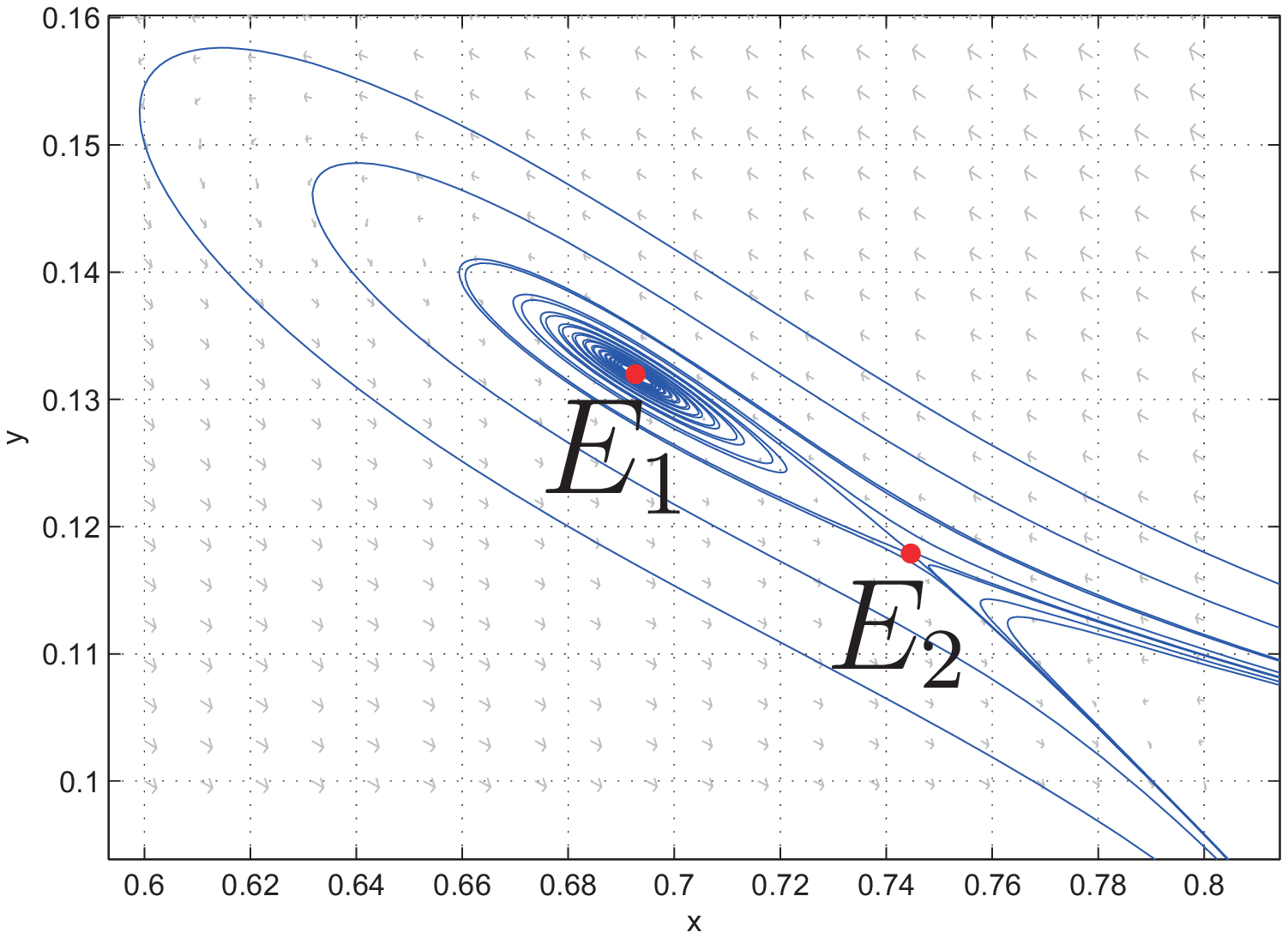}}\\
\subfigure[]{\label{fig:subfig:c}
\includegraphics[width=5.5cm,height=5.5cm]{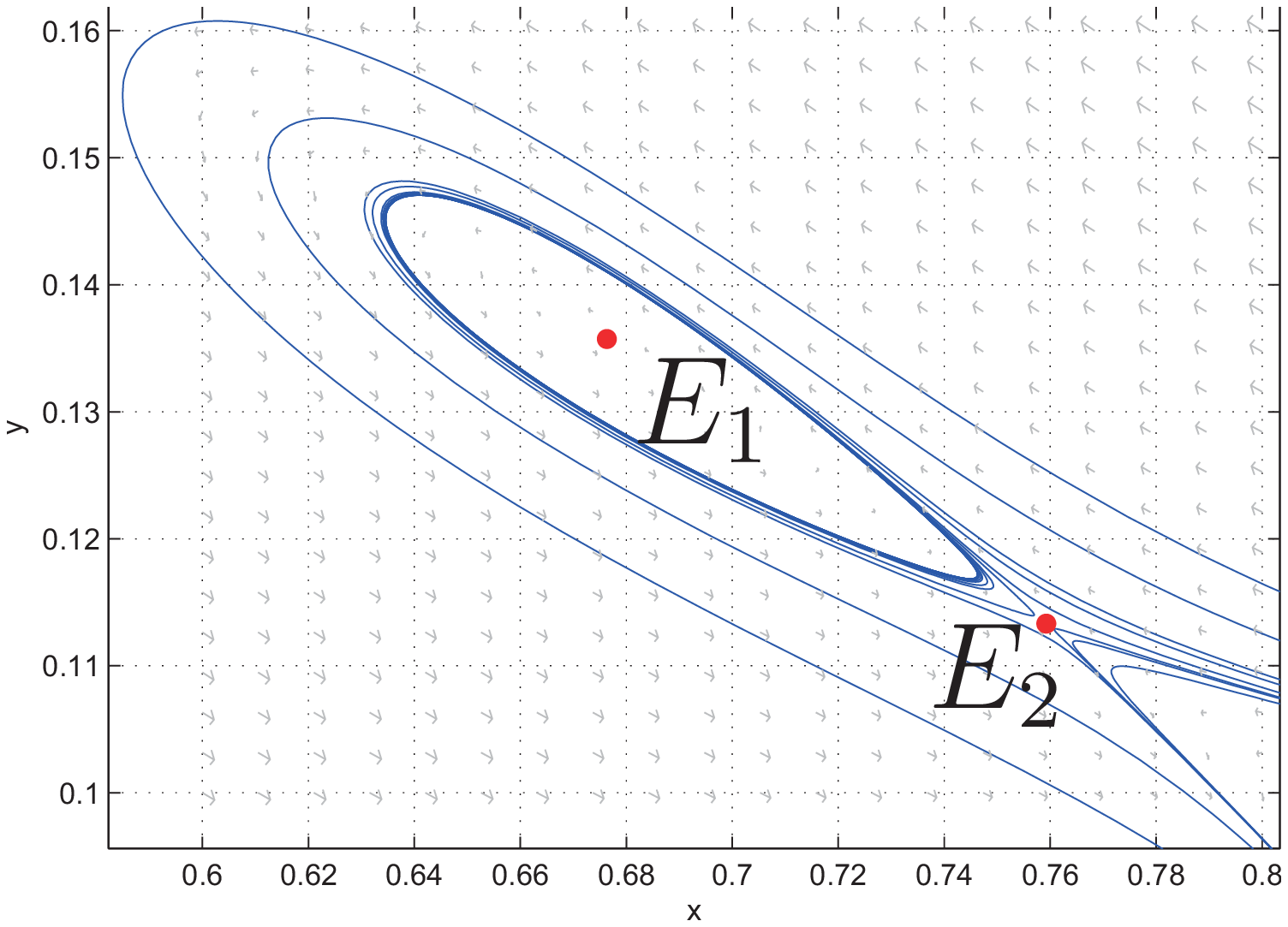}}
\subfigure[]{\label{fig:subfig:d}
\includegraphics[width=5.5cm,height=5.5cm]{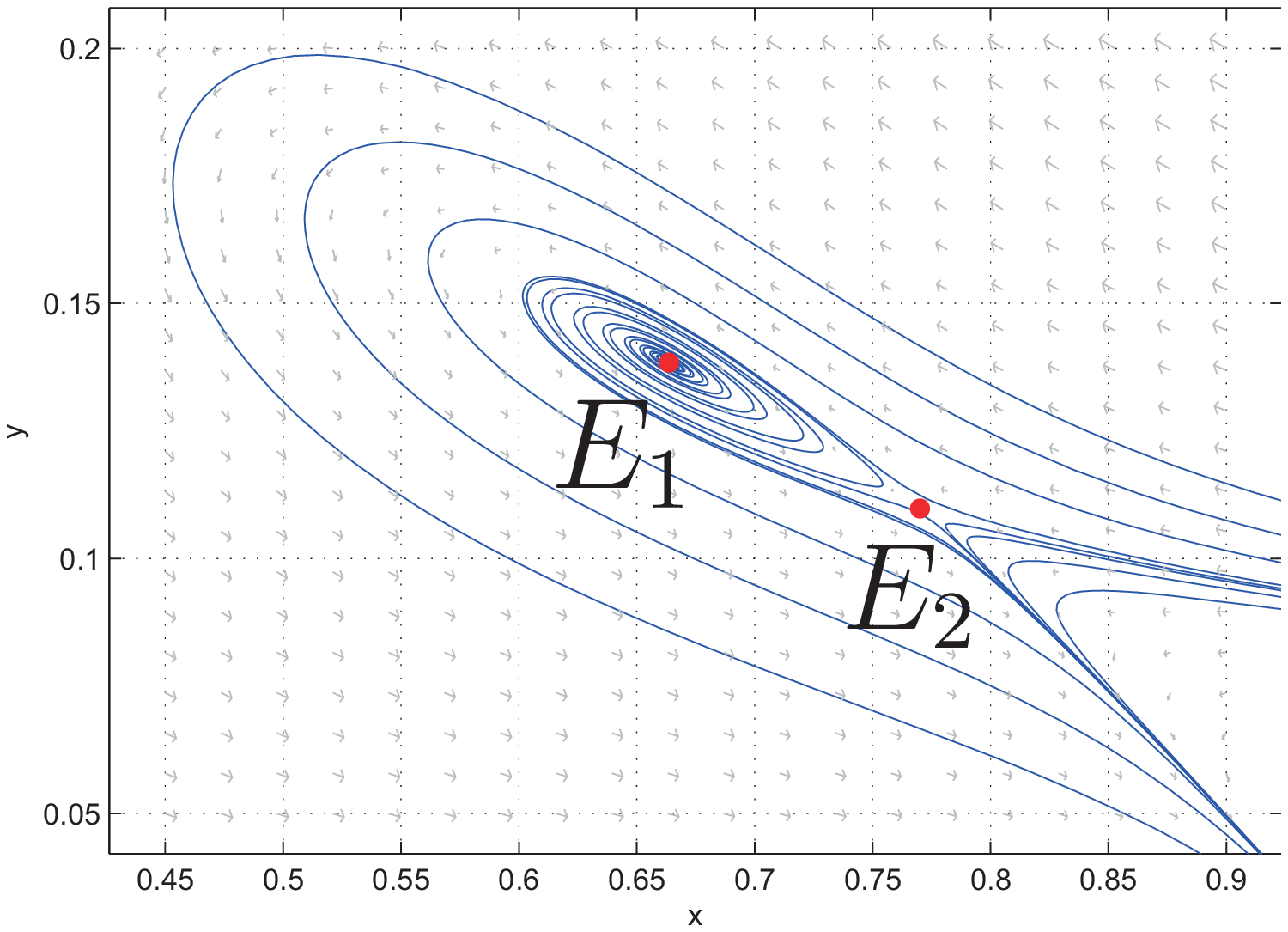}}
\caption{When $h=0.5$ and $k=1$, the dynamical behaviors of system (\ref{(3)}) near the cusp $E_*$ for parameters $(\sigma,\alpha)$ in neighborhood $U$ of point $(\sigma_2,\alpha_*)$ are as follows. (a): No equilibrium as $(\sigma,\alpha)=(0.62, 14.2)$.
(b): Stable focus $E_1$ and saddle $E_2$ as $(\sigma,\alpha)=(0.62, 14.3)$.
(c): Unstable focus $E_1$ surrounded by a stable limit cycle and saddle $E_2$ as $(\sigma,\alpha)=(0.62, 14.42)$.
(d): Unstable focus $E_1$ and saddle $E_2$ as $(\sigma,\alpha)=(0.62, 14.55)$.}
\label{Figure 3}
\end{figure}

\begin{figure}[h]
\begin{center}
\includegraphics[width=8cm,height=6cm]{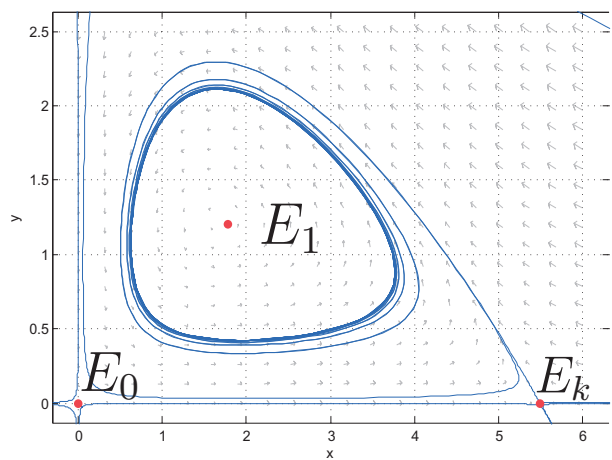}\
\caption {System (\ref{(3)}) has two saddles $E_0$ and $E_k$ and an unstable focus or node $E_1$ surrounded by a stable limit cycle when $h=0.5$, $\sigma=1$, $\alpha=0.1$ and $k=5.5$.} \label{Figure 4}
\end{center}
\end{figure}
From an ecological point of view, the foraging facilitation among predators is an interesting phenomenon to understand the dynamics of the prey-predator interactions in ecosystems, and it is more realistic and reasonable to take into account this factor in the prey-predator system.
The qualitative results of system (\ref{(2)}) indicate that prey-predator system (\ref{(2)}) with foraging facilitation has richer dynamic behaviors than the Rosenzweig-MacArthur system because system (\ref{(00)}) only undergoes the transcritical and Hopf bifurcation.
The analysis of system (\ref{(2)}) reveals that population can be stabilized at the predator free equilibrium or the coexistence equilibrium with increasing the foraging facilitation $\alpha$ as the environmental capacity of prey is relatively low. How the population evolves in time depends on the initial conditions. The foraging facilitation is then beneficial for population persistence and promotes ecosystem diversity.
Cooperative predators can survive in a less favorable and less productive environment, in which sufficient preys are available and the survival is more robust for higher levels of cooperation.
The bistability of the system implies that the predator population goes extinct for low initial predator densities, which actually is the phenomenon of Allee effect in the predators (\cite{CourchampBerec}).  Therefore, the foraging facilitation is a mechanism for inducing Allee effects in predators.
For low environmental capacity of prey and weak foraging facilitation, the prey population is too small to sustain the predator population even though the foraging facilitation of predators exists.
Nevertheless, the foraging facilitation can have not only positive but also negative effects for predators.
For very strong foraging facilitation, the population goes to extinction due to the excessive hunting of prey population by predator population.
The destabilization of the system appears due to the Hopf bifurcation even the homoclinic bifurcation that causes splitting of
the stable cycle, thus ending the oscillation and consequently causing the extinction of the predators. The overexploitation can therefore backfire and result in the extinction of predators because of the increased predation pressure.
It is well known that the Rosenzweig-MacArthur system demonstrates the paradox of enrichment caused by the
Hopf bifurcation (\cite{Rosenzweig1}), which means that a stable oscillation  bifurcates from a stable equilibrium once the environmental carrying capacity of the prey exceeds a critical value. The qualitative results of system (\ref{(2)}) reveal that this typical phenomenon of system (\ref{(00)}) is inherited even if the foraging facilitation is introduced.
We also can observe that the predators will go to extinct if the handing time of the predators $h$ is too long such as $h\geq1$.
By means of bifurcation analysis of prey-predator system (\ref{(2)}), we have proved that hunting cooperation is not always beneficial for the predator population.
Such studies of bifurcations may give insights into the important changes of dynamical behaviors of the system caused by small perturbation of parameters. The results of bifurcations provide some thresholds to control the qualitative properties of the prey-predator system.

\nonumsection{Acknowledgments}
The author is grateful to the associate editor and reviewers for their valuable comments and suggestions that have significantly improved the presentation and quality of the  manuscript.

\nonumsection{Appendix}
\begin{eqnarray*}
\begin{array}{l}
a_{100}:=\frac{x_* \sigma (\alpha_1 k \sigma x_*-\alpha_1 \sigma x_*^2+k) (h k-h x_*-x_*)}{k},~~ a_{001}:=\frac{x_*^3 \sigma^2 (k-x_*)^2 (h-1)}{k},\\
a_{010}:=x_* (\alpha_1 h k \sigma x_*-\alpha_1 h \sigma x_*^2-2 \alpha_1 k \sigma x_*+2 \alpha_1 \sigma x_*^2-k),\\
a_{110}:=2 \alpha_1 h k \sigma x_*-3 \alpha_1 h \sigma x_*^2-2 \alpha_1 k \sigma x_*+2 \alpha_1 \sigma x_*^2-k,~~
a_{101}:=\frac{x_*^2 \sigma^2 (k-x_*) (2 h k-3 h x_*-k+x_*)}{k},\\
a_{200}:=\frac{\sigma (\alpha_1 h k^2 \sigma x_*-4 \alpha_1 h k \sigma x_*^2+3 \alpha_1 h \sigma x_*^3+h k^2-3 h k x_*-k)}{k},~~
a_{011}:=x_*^2 \sigma (k-x_*) (h-2),~a_{020}:=-x_* \alpha_1 k,\\
\end{array}
\end{eqnarray*}
\begin{eqnarray*}
\begin{array}{l}
b_{100}:=\frac{-\sigma (k-x_*) x_* (-1+h) (\alpha_1 k \sigma x_*-\alpha_1 \sigma x_*^2+k)}{k},~~
b_{010}:=-\sigma (k-x_*) x_*^2 \alpha_1 (-1+h),~b_{001}:=\frac{-x_*^3 \sigma^2 (k-x_*)^2 (-1+h)}{k},\\
b_{020}:=-k \alpha_1 x_* (-1+h),~~
b_{110}:=-(-1+h) (2 \alpha_1 k \sigma x_*-2 \alpha_1 \sigma x_*^2+k),~b_{101}:=\frac{-\sigma^2 (k-x_*)^2 x_*^2 (-1+h)}{k},\\
b_{011}:=-2 \sigma (k-x_*) x_*^2 (-1+h),~
a_{10}:=-\frac{(h k (h-1)+h+1) \sigma k}{(h+1) \sigma+k (h-1)^2},~~
a_{01}:=-\frac{(h (h-1)^2 k+(h+1) (h-2)) \sigma-k (h-1)^2 k}{(h-1) ((h+1) \sigma+k (h-1)^2)},\\
a_{20}:=-\frac{(h (h-1)^2 k-h^2 \sigma+(2 h+1) (h-1)) \sigma}{(h-1) (h \sigma-h+1)},~~
a_{02}:= \frac{(h^2 k-h k+h+1) ((h+1) \sigma+k (h-1)^2)}{((h-1) (h k-k+1)+\sigma) (h \sigma-h+1) (h-1)},\\
a_{11}:=\frac{-1}{((h-1) (h k-k+1)+\sigma) (h \sigma-h+1) (h-1)}\{(-h (h-1)(k (h-1)^2+h)-2) \sigma^2+(h-1) (h (2 h-1) (h-1)^2 k^2\\
\phantom{a_{11}:=}+(h-1) (5 h^2-h-3) k+(h+1) (3 h-2)) \sigma-k (h-1)^3 (h k-k+1)\},\\
a_{30}:=\frac{((h+1) \sigma+k (h-1)^2) h \sigma}{(h \sigma-h+1) (h-1) k},~~
a_{21}:=-\frac{((-2 h+1) \sigma+k (h-1)^2+3 h-3) ((h+1) \sigma+k (h-1)^2) (h^2 k-h k+h+1) h \sigma}{k ((h-1) (h k-k+1)+\sigma) (h \sigma-h+1)^2 (h-1)},\\
a_{12}:=\frac{(k (-h^2+h)-h-1) (k (h-1)^2+h \sigma+\sigma)^2}{k (1-h) (h \sigma-h+1)^2 (k (h-1)^2+h+\sigma-1)},~~
b_{10}:=\frac{\{(h-1) (h k-k+1)+\sigma\} k \sigma}{(h+1) \sigma+k (h-1)^2},~~
b_{01}:=\frac{(h^2 k-h k+h+1) \sigma k}{(h+1) \sigma+k (h-1)^2},\\
b_{11}:=\frac{\{h k (h-1)+2 h+2\} \sigma+k (h-1)^2}{h \sigma-h+1},~~
b_{02}:=\frac{(h^2 k-h k+h+1) ((h+1) \sigma+k (h-1)^2)}{((h-1) (h k-k+1)+\sigma) (h \sigma-h+1)},\\
b_{12}:=\frac{\{(h+1) \sigma+k (h-1)^2\}^2 (h^2 k-h k+h+1)}{k ((h-1) (h k-k+1)+\sigma) (h \sigma-h+1)^2},~~
f_{11}:=\frac{a_{11} b_{10}-2 a_{20} b_{01}+b_{01} b_{11}}{b_{10} \beta^2},~~
f_{20}:=\frac{a_{20}}{b_{10} \beta},\\
f_{02}:=\frac{a_{02} b_{10}^2-a_{11} b_{01} b_{10}+a_{20} b_{01}^2-b_{01}^2 b_{11}+b_{01} b_{02} b_{10}}{b_{10} \beta^3},~~
f_{30}:=\frac{a_{30}}{b_{10}^2 \beta},~~
f_{03}:=-\frac{b_{01} (a_{12} b_{10}^2-a_{21} b_{01} b_{10}+a_{30} b_{01}^2+b_{01} b_{10} b_{12})}{b_{10}^2 \beta^4},\\
f_{21}:=\frac{a_{21} b_{10}-3 a_{30} b_{01}}{b_{10}^2 \beta^2},~~
f_{12}:=\frac{a_{12} b_{10}^2-2 a_{21} b_{01} b_{10}+3 a_{30} b_{01}^2+b_{01} b_{10} b_{12}}{b_{10}^2 \beta^3},~~
g_{11}:=\frac{b_{11}}{b_{10} \beta},~~
g_{02}:=-\frac{b_{01} b_{11}-b_{02} b_{10}}{b_{10} \beta^2},\\
g_{12}:=\frac{b_{12}}{b_{10} \beta^2},~~
g_{03}:=-\frac{b_{12} b_{01}}{b_{10} \beta^3},~~
A_{10}:=-\frac{\{h (h k-k+1)+1\} \sigma_2 k}{(h+1) \sigma_2+(h-1)^2 k},~~
A_{01}:=-\frac{\{(h (h-1)^2 k+h^2-h-2) \sigma_2-(h-1)^2 k\} k}{(h-1) \{(h+1) \sigma_2+(h-1)^2 k\}},\\
A_{20}:=-\frac{\sigma_2 \{h (h-1)^2 k-h^2 \sigma_2+(2 h+1) (h-1)\}}{(h-1) (h \sigma_2-h+1)},~~
A_{02}:=\frac{\{h (h k-k+1)+1\} \{(h+1) \sigma_2+(h-1)^2 k\}}{\{(h-1)^2 k+h+\sigma_2-1\} (h \sigma_2-h+1) (h-1)},\\
A_{11}:=\frac{-1}{\{(h-1)^2 k+h+\sigma_2-1\} (h \sigma_2-h+1) (h-1)}\{(-h (h-1)^3 k-h^3+h^2-2) \sigma_2^2+(h-1) (h (2 h-1) (h-1)^2 k^2\\
\phantom{A_{11}:=}+(h-1) (5 h^2-h-3) k+3 h^2+h-2) \sigma_2-k (h-1)^3 (h k-k+1)\},\\
B_{10}:=\frac{\{(h-1)^2 k+h+\sigma_2-1\} k \sigma_2}{(h+1) \sigma_2+(h-1)^2 k},~~
B_{01}:=\frac{\{h (h k-k+1)+1\} \sigma_2 k}{(h+1) \sigma_2+(h-1)^2 k},~
B_{11}:=\frac{\{h k (h-1)+2 h+2\} \sigma_2+(h-1)^2 k}{h \sigma_2-h+1},\\
B_{02}:=\frac{\{h (h k-k+1)+1\} \{(h+1) \sigma_2+(h-1)^2 k\}}{\{(h-1)^2 k+h+\sigma_2-1\} (h \sigma_2-h+1)},~
\mathcal{F}_{00}(\epsilon_1, \epsilon_2):=\frac{E_{00}^2 F_{02}-E_{00} E_{01} F_{01}+E_{01}^2 F_{00}}{E_{01}},\\
\mathcal{F}_{10}(\epsilon_1, \epsilon_2):=\frac{-1}{E_{01}^2}\{E_{00}^2 E_{11} F_{02}+E_{00} E_{01}^2 F_{11}-2 E_{00} E_{01} E_{10} F_{02}-E_{01}^3 F_{10}+E_{01}^2 E_{10} F_{01}-E_{01}^2 E_{11} F_{00}\},\\
\mathcal{F}_{01}(\epsilon_1, \epsilon_2):=-\frac{E_{00} E_{11}+2 E_{00} F_{02}-E_{01} E_{10}-E_{01} F_{01}}{E_{01}},~\mathcal{F}_{02}(\epsilon_1, \epsilon_2):=\frac{E_{11}+F_{02}}{E_{01}},\\
\mathcal{F}_{11}(\epsilon_1, \epsilon_2):=\frac{1}{E_{01}^2}\{E_{00} E_{11}^2+2 E_{00} E_{11} F_{02}+2 E_{01}^2 E_{20}+E_{01}^2 F_{11}-E_{01} E_{10} E_{11}-2 E_{01} E_{10} F_{02}\},\\
\mathcal{F}_{20}(\epsilon_1, \epsilon_2):=\frac{-1}{E_{01}^3}\{E_{00}^2 E_{11}^2 F_{02}-2 E_{00} E_{01}^2 E_{20} F_{02}+2 E_{00} E_{01} E_{10} E_{11} F_{02}-E_{00} E_{01} E_{11}^2 F_{01}-E_{01}^4 F_{20}\\
\phantom{\mathcal{F}_{20}(\epsilon_1, \epsilon_2)=}
+E_{01}^3 E_{10} F_{11}-E_{01}^3 E_{11} F_{10}+E_{01}^3 E_{20} F_{01}-E_{01}^2 E_{10}^2 F_{02}\},\\
p_{001}:=-\frac{(a_{100}+b_{100}) b_{001} b_{010}}{b_{100} (a_{100}+b_{010})},~
p_{200}:=\frac{(a_{020} b_{100}^2+a_{100} b_{010} b_{110}-a_{100} b_{020} b_{100}-a_{110} b_{010} b_{100}+a_{200} b_{010}^2)}{(a_{100}+b_{010}) b_{010}},\\
p_{020}:=\frac{(a_{020} b_{100}^2+a_{100}^2 a_{200}-a_{100}^2 b_{110}+a_{100} a_{110} b_{100}-a_{100} b_{020} b_{100}) b_{010}}{b_{100}^2 (a_{100}+b_{010})},\\
p_{002}:=\frac{b_{001} b_{010} (b_{010}-b_{100}) (a_{100}^2 b_{101}-a_{100} a_{101} b_{100}+a_{100} b_{010} b_{101}-a_{101} b_{010} b_{100}+a_{200} b_{001}  b_{010}-a_{200} b_{001} b_{100})}{(a_{100}+b_{010})^3 b_{100}^2},\\
p_{110}:=\frac{2 a_{020} b_{100}^2-a_{100}^2 b_{110}+a_{100} a_{110} b_{100}-2 a_{100} a_{200} b_{010}+a_{100} b_{010} b_{110}-2 a_{100} b_{020} b_{100}-a_{110} b_{010} b_{100}}{b_{100} (a_{100}+b_{010})},\\
p_{011}:=\frac{-1}{b_{100}^2 (a_{100}+b_{010})^2}(a_{011} a_{100} b_{100}^2+a_{011} b_{010} b_{100}^2-a_{100}^3 b_{101}+a_{100}^2 a_{101} b_{100}-a_{100}^2 b_{010} b_{101}\\
\phantom{p_{011}:=}-a_{100}^2 b_{011} b_{100}+a_{100} a_{101} b_{010} b_{100}-2 a_{100} a_{200} b_{001} b_{010}+2 a_{100} a_{200} b_{001} b_{100}+a_{100} b_{001} b_{010} b_{110}\\
\phantom{p_{011}:=}-a_{100} b_{001} b_{100} b_{110}-a_{100} b_{010} b_{011} b_{100}-a_{110} b_{001} b_{010} b_{100}+a_{110} b_{001} b_{100}^2) b_{010},\\
p_{101}:=\frac{-1}{b_{100} (a_{100}+b_{010})^2}(a_{011} a_{100} b_{100}^2+a_{011} b_{010} b_{100}^2+a_{100}^2 b_{010} b_{101}-a_{100}^2 b_{011} b_{100}-a_{100} a_{101} b_{010} b_{100}\\
\phantom{p_{011}:=}+a_{100} b_{001} b_{010} b_{110}-a_{100} b_{001} b_{100} b_{110}+a_{100} b_{010}^2 b_{101}-a_{100} b_{010} b_{011} b_{100}-a_{101} b_{010}^2 b_{100}\\
\phantom{p_{011}:=}-a_{110} b_{001} b_{010} b_{100}+a_{110} b_{001} b_{100}^2+2 a_{200} b_{001} b_{010}^2-2 a_{200} b_{001} b_{010} b_{100}),\\
q_{010}:=a_{100}+b_{010},~
q_{200}:=\frac{(a_{020} b_{100}^2-a_{110} b_{010} b_{100}+a_{200} b_{010}^2-b_{010}^2 b_{110}+b_{010} b_{020} b_{100}) b_{100}}{(a_{100}+b_{010}) b_{010}^2},\\
q_{020}:=\frac{a_{020} b_{100}^2+a_{100}^2 a_{200}+a_{100} a_{110} b_{100}+a_{100} b_{010} b_{110}+b_{010} b_{020} b_{100}}{b_{100} (a_{100}+b_{010})},\\
q_{002}:=-\frac{b_{001} (b_{010}-b_{100}) (a_{100} a_{101} b_{100}+a_{100} b_{010} b_{101}+a_{101} b_{010} b_{100}-a_{200} b_{001} b_{010}+a_{200} b_{001} b_{100}+b_{010}^2 b_{101})}{(a_{100}+b_{010})^3 b_{100}},\\
q_{110}:=-\frac{2 a_{020} b_{100}^2+a_{100} a_{110} b_{100}-2 a_{100} a_{200} b_{010}+a_{100} b_{010} b_{110}-a_{110} b_{010} b_{100}-b_{010}^2 b_{110}+2 b_{010} b_{020} b_{100}}{b_{010} (a_{100}+b_{010})},\\
q_{101}:=\frac{-1}{b_{010} (a_{100}+b_{010})^2}(a_{011} a_{100} b_{100}^2+a_{011} b_{010} b_{100}^2-a_{100} a_{101} b_{010} b_{100}-a_{100} b_{010}^2 b_{101}+a_{100} b_{010} b_{011} b_{100}\\
\phantom{p_{011}:=}-a_{101} b_{010}^2 b_{100}-a_{110} b_{001} b_{010} b_{100}+a_{110} b_{001} b_{100}^2+2 a_{200} b_{001} b_{010}^2-2 a_{200} b_{001} b_{010} b_{100}\\
\phantom{p_{011}:=}-b_{001} b_{010}^2 b_{110}+b_{001} b_{010} b_{100} b_{110}-b_{010}^3 b_{101}+b_{010}^2 b_{011} b_{100}),\\
q_{011}:=\frac{1}{b_{100} (a_{100}+b_{010})^2}(a_{011} a_{100} b_{100}^2+a_{011} b_{010} b_{100}^2+a_{100}^2 a_{101} b_{100}+a_{100}^2 b_{010} b_{101}+a_{100} a_{101} b_{010} b_{100}\\
\phantom{p_{011}:=}-2 a_{100} a_{200} b_{001} b_{010}+2 a_{100} a_{200} b_{001} b_{100}+a_{100} b_{010}^2 b_{101}+a_{100} b_{010} b_{011} b_{100}-a_{110} b_{001} b_{010} b_{100}\\
\phantom{p_{011}:=}+a_{110} b_{001} b_{100}^2-b_{001} b_{010}^2 b_{110}+b_{001} b_{010} b_{100} b_{110}+b_{010}^2 b_{011} b_{100}),\\
\end{array}
\end{eqnarray*}
\begin{eqnarray*}
\begin{array}{l}
c_{11}:=\frac{1}{(a_{100}+b_{010})^4 b_{010}}(a_{011} a_{100}^2 b_{100}^2+2 a_{011} a_{100} b_{010} b_{100}^2+a_{011} b_{010}^2 b_{100}^2+2 a_{020} a_{100} b_{001} b_{100}^2+2 a_{020} b_{001} b_{100}^3\\
\phantom{c_{11}=}-a_{100}^2 a_{101} b_{010} b_{100}-a_{100}^2 b_{010}^2 b_{101}+a_{100}^2 b_{010} b_{011} b_{100}-2 a_{100} a_{101} b_{010}^2 b_{100}-3 a_{100} a_{110} b_{001} b_{010} b_{100}\\
\phantom{c_{11}=}+a_{100} a_{110} b_{001} b_{100}^2+4 a_{100} a_{200} b_{001} b_{010}^2-2 a_{100} a_{200} b_{001} b_{010} b_{100}-3 a_{100} b_{001} b_{010}^2 b_{110}\\
\phantom{c_{11}=}+2 a_{100} b_{001} b_{010} b_{020} b_{100}+a_{100} b_{001} b_{010} b_{100} b_{110}-2 a_{100} b_{010}^3 b_{101}+2 a_{100} b_{010}^2 b_{011} b_{100}-a_{101} b_{010}^3 b_{100}\\
\phantom{c_{11}=}-a_{110} b_{001} b_{010}^2 b_{100}-a_{110} b_{001} b_{010} b_{100}^2+2 a_{200} b_{001} b_{010}^3-b_{001} b_{010}^3 b_{110}-b_{001} b_{010}^2 b_{100} b_{110}\\
\phantom{c_{11}=}+2 b_{001} b_{010} b_{020} b_{100}^2-b_{010}^4 b_{101}+b_{010}^3 b_{011} b_{100}),\\
c_{20}:=\frac{-b_{100}}{(a_{100}+b_{010})^2 b_{010}^2}(a_{020} b_{100}^2-a_{110} b_{010} b_{100}+a_{200} b_{010}^2-b_{010}^2 b_{110}+b_{010} b_{020} b_{100}),\\
c_{02}:=\frac{-b_{001}}{(a_{100}+b_{010})^6 b_{100}}(a_{011} a_{100}^3 b_{100}^2+2 a_{011} a_{100}^2 b_{010} b_{100}^2+a_{011} a_{100}^2 b_{100}^3+a_{011} a_{100} b_{010}^2 b_{100}^2+2 a_{011} a_{100} b_{010} b_{100}^3\\
\phantom{c_{11}=}+a_{011} b_{010}^2 b_{100}^3+2 a_{020} a_{100}^2 b_{001} b_{100}^2+4 a_{020} a_{100} b_{001} b_{100}^3+2 a_{020} b_{001} b_{100}^4-2 a_{100}^3 a_{101} b_{010} b_{100}\\
\phantom{c_{11}=}+a_{100}^3 a_{101} b_{100}^2-2 a_{100}^3 b_{010}^2 b_{101}+a_{100}^3 b_{010} b_{011} b_{100}+a_{100}^3 b_{010} b_{100} b_{101}-5 a_{100}^2 a_{101} b_{010}^2 b_{100}\\
\phantom{c_{11}=}+2 a_{100}^2 a_{101} b_{010} b_{100}^2-3 a_{100}^2 a_{110} b_{001} b_{010} b_{100}+a_{100}^2 a_{110} b_{001} b_{100}^2+5 a_{100}^2 a_{200} b_{001} b_{010}^2\\
\phantom{c_{11}=}-4 a_{100}^2 a_{200} b_{001} b_{010} b_{100}+a_{100}^2 a_{200} b_{001} b_{100}^2-3 a_{100}^2 b_{001} b_{010}^2 b_{110}+2 a_{100}^2 b_{001} b_{010} b_{020} b_{100}\\
\phantom{c_{11}=}+a_{100}^2 b_{001} b_{010} b_{100} b_{110}-5 a_{100}^2 b_{010}^3 b_{101}+2 a_{100}^2 b_{010}^2 b_{011} b_{100}+2 a_{100}^2 b_{010}^2 b_{100} b_{101}+a_{100}^2 b_{010} b_{011} b_{100}^2\\
\phantom{c_{11}=}-4 a_{100} a_{101} b_{010}^3 b_{100}+a_{100} a_{101} b_{010}^2 b_{100}^2-a_{100} a_{110} b_{001} b_{010}^2 b_{100}-4 a_{100} a_{110} b_{001} b_{010} b_{100}^2+a_{100} a_{110} b_{001} b_{100}^3\\
\phantom{c_{11}=}+4 a_{100} a_{200} b_{001} b_{010}^3-a_{100} b_{001} b_{010}^3 b_{110}-4 a_{100} b_{001} b_{010}^2 b_{100} b_{110}+4 a_{100} b_{001} b_{010} b_{020} b_{100}^2\\
\phantom{c_{11}=}+a_{100} b_{001} b_{010} b_{100}^2 b_{110}-4 a_{100} b_{010}^4 b_{101}+a_{100} b_{010}^3 b_{011} b_{100}+a_{100} b_{010}^3 b_{100} b_{101}+2 a_{100} b_{010}^2 b_{011} b_{100}^2\\
\phantom{c_{11}=}-a_{101} b_{010}^4 b_{100}-a_{110} b_{001} b_{010}^2 b_{100}^2-a_{110} b_{001} b_{010} b_{100}^3+a_{200} b_{001} b_{010}^4+a_{200} b_{001} b_{010}^2 b_{100}^2\\
\phantom{c_{11}=}-b_{001} b_{010}^3 b_{100} b_{110}-b_{001} b_{010}^2 b_{100}^2 b_{110}+2 b_{001} b_{010} b_{020} b_{100}^3-b_{010}^5 b_{101}+b_{010}^3 b_{011} b_{100}^2),\\
E_{00}:=-\frac{k^3 \sigma_2 (h \sigma_2-h+1)^3 (h^2 k-2 h k+h+k+\sigma_2-1) (h-1) \epsilon_1}{(h^2 k-2 h k+h \sigma_2+k+\sigma_2)^4},\\
E_{10}:=\frac{k^2 (h-1) \epsilon_1 (h \sigma_2-h+1)^2 (h^2 k \sigma_2-2 h^2 k-h k \sigma_2+4 h k-h \sigma_2-2 k-\sigma_2)}{(h^2 k-2 h k+h \sigma_2+k+\sigma_2)^3},\\
E_{01}:=\frac{-k^2 \sigma_2 (h-1) (h \sigma_2-h+1)^2 (h^2 k-2 h k+h+k+\sigma_2-1) \epsilon_1+(h^2 k-2 h k+h \sigma_2+k+\sigma_2)^3}{(h^2 k-2 h k+h \sigma_2+k+\sigma_2)^3} ,\\
E_{20}:=-\{-k^2 (h-1) (h \sigma_2-h+1)^2 (h^2 k-2 h k+h+k+\sigma_2-1) ((2 h^2 k-2 h k+h+1) \sigma_2-k (h-1)^2) \epsilon_1\\
\phantom{E_{20}:=}+\sigma_2 (h^2 k-h k+h+1)^2 (h^2 k-2 h k+h \sigma_2+k+\sigma_2)^2\}/\{\sigma_2 k (h^2 k-2 h k+h+k+\sigma_2-1)^2 \\
\phantom{E_{20}:=}\times(h \sigma_2-h+1) (h^2 k-2 h k+h \sigma_2+k+\sigma_2)\},\\
E_{11}:=\{-2 k^2 \sigma_2 (h-1) (h \sigma_2-h+1)^2 (h^2 k-2 h k+h+k+\sigma_2-1) \epsilon_1+((h^2 k-h k+2 h+2) \sigma_2+k (h-1)^2)\\
\phantom{E_{20}:=}\times(\sigma_2 (h+1)+k (h-1)^2)^2\}/\{k \sigma_2 (h^2 k-2 h k+h \sigma_2+k+\sigma_2) (h^2 k-2 h k+h+k+\sigma_2-1)\\
\phantom{E_{20}:=}\times(h \sigma_2-h+1)\} ,\\
F_{00}:=\{(h \sigma_2-h+1) k h k^2 \sigma_2 (h-1) (h \sigma_2-h+1)^2 (h^2 k-2 h k+h+k+\sigma_2-1) \epsilon_1 \epsilon_2-k^2 \sigma_2^2 (h-1)^2\\
\phantom{E_{20}:=}\times(h \sigma_2-h+1)^2 (h k-k-\sigma_2+2) \epsilon_1-(h^2 k-2 h k+h \sigma_2+k+\sigma_2)^3 \epsilon_2\}/\{\sigma_2 (h-1) (h^2 k\\
\phantom{E_{20}:=}-2 h k+h \sigma_2+k+\sigma_2)^4\},\\
F_{10}:=-\{k^2 (h-1)^2 (h \sigma_2-h+1)^2 (h^2 k-2 h k+h+k+\sigma_2-1) ((h k-h-1) \sigma_2^2+(h+1) (h^3 k^2-2 h^2 k^2\\
\phantom{E_{20}:=}+h^2 k+h k^2-k+2) \sigma_2-h k (h-1)^2 (h k-k+1)) \epsilon_1 \epsilon_2+k^2 \sigma_2 (h-1) (h \sigma_2-h+1)^2 (h^3 k+h^2\\
\phantom{E_{20}:=}-3 h k-h \sigma_2+3 h+2 k+2 \sigma_2-4)(h^2 k-2 h k+h \sigma_2+k+\sigma_2) (h^2 k-2 h k+h+k+\sigma_2-1) \epsilon_1\\
\phantom{E_{20}:=}-(h-1) (h^2 k-h k+h+1) (h k-k-\sigma_2+2) (h^2 k-2 h k+h \sigma_2+k+\sigma_2)^3 \epsilon_2-((-h^3 k+2 h^2 k\\
\phantom{E_{20}:=}-h^2-h k+h+2) \sigma_2^2+(h-1) (h^2 k-h k+3 h+3) (h k-k+1) \sigma_2+k (h-1)^3 (h k-k+1))(h^2 k\\
\phantom{E_{20}:=}-2 h k+h \sigma_2+k+\sigma_2)^3 \}/\{\sigma_2 (h^2 k-2 h k+h+k+\sigma_2-1)^2 (h-1) (h^2 k-2 h k+h \sigma_2+k+\sigma_2)^3\} ,\\
F_{01}:=-\{-(h^3 k-h^2 k+h^2-h k+h+k+\sigma_2-2) k^2 \sigma_2 (h-1) (h \sigma_2-h+1)^2 (h^2 k-2 h k+h+k+\sigma_2\\
\phantom{E_{20}:=}-1) \epsilon_2 \epsilon_1+(h^3 k-h^2 k+h^2-h k+h+k+\sigma_2-2) (h^2 k-2 h k+h \sigma_2+k+\sigma_2)^3 \epsilon_2\}/\{\sigma_2 (h^2 k-2 h k\\
\phantom{E_{20}:=}+h+k+\sigma_2-1) (h-1)(h^2 k-2 h k+h \sigma_2+k+\sigma_2)^3\} ,\\
F_{11}:=\{-h k^2 (h-1) (h \sigma_2-h+1)^3 ((-4 h^2 k+6 h k-3 h-2 k-3) \sigma_2+k (h-1)^2 (2 h k-2 k+3)) (h^2 k\\
\phantom{E_{20}:=}-2 h k+h+k+\sigma_2-1) \epsilon_1 \epsilon_2-k^2 \sigma_2 (h-1) (h \sigma_2-h+1)^2 (h^2 k-2 h k+h+k+\sigma_2-1)\\
\phantom{E_{20}:=}\times ((-2 h^2 (2 h-1) (h-1) k-(h+1) (3 h^2-2)) \sigma_2^2+(h-1) (2 h^2 (h-1)^2 k^2+(9 h^3-7 h^2+2 h-4) k\\
\phantom{E_{20}:=}+5 h^2+9 h+4) \sigma_2+k (h-1)^3 (2 h k-h-2 k+4)) \epsilon_1-(h+2) (h \sigma_2-h+1)(h^2 k-h k+h+1) \\
\phantom{E_{20}:=}(h^2 k-2 h k+h \sigma_2+k+\sigma_2)^3 \epsilon_2+((-h^4 k-h^3 k-h^3+3 h^2 k-3 h^2-h k+2) \sigma_2^2+(h-1) (h^2 k-h k\\
\phantom{E_{20}:=}+3 h+3) (h^2 k+h-k+2) \sigma_2+k (h-1)^3 (h^2 k+h-k+2))(h^2 k-2 h k+h \sigma_2+k+\sigma_2)^3\}\\
\phantom{E_{20}:=}/\{\sigma_2 k (h^2 k-2 h k+h \sigma_2+k+\sigma_2)^2 (h^2 k-2 h k+h+k+\sigma_2-1)^2 (h \sigma_2-h+1) (h-1)\} ,\\
F_{02}:=-\{(\sigma_2+\epsilon_2)(-h k^2 \sigma_2 (h-1) (h \sigma_2-h+1)^2 (h^2 k-2 h k-2 h \sigma_2+3 h+k+\sigma_2-3) (h^2 k-2 h k+h\\
\phantom{E_{20}:=}+k+\sigma_2-1) \epsilon_1+(h^3 k-2 h^2 k-h^2 \sigma_2+2 h^2+h k-h-1) (h^2 k-2 h k+h \sigma_2+k+\sigma_2)^3)\}\\
\phantom{E_{20}:=}/\{k \sigma_2 (h^2 k-2 h k+h+k+\sigma_2-1) (h \sigma_2-h+1) (h-1) (h^2 k-2 h k+h \sigma_2+k+\sigma_2)^2\} ,\\
F_{20}:=\{h k^2 (h-1) (h \sigma_2-h+1)^2 (h^2 k-h k+h+1) (h^2 k-2 h k+h+k+\sigma_2-1) ((-2 h^3 k+3 h^2 k-h^2\\
\phantom{E_{20}:=}-h k-2 h-1) \sigma_2^2+k (h-1)^3 (h k+3) \sigma_2-k (h-1)^3 (2 h k-2 k+3)) \epsilon_1 \epsilon_2+k^2 (h-1) (h \sigma_2-h+1)^2\\
\end{array}
\end{eqnarray*}
\begin{eqnarray*}
\begin{array}{l}
\phantom{E_{20}:=}\times (h^2 k-2 h k+h+k+\sigma_2-1) ((-h^3 (2 h-1) (h-1)^2 k^2+(-3 h^5+5 h^3-2 h) k-(h^2+h-1)\\
\phantom{E_{20}:=}\times (h+1)^2) \sigma_2^3+(h-1) (-(h^2+h-1) (h+1)^2 k^3+2 h (3 h^2+h+2) (h-1)^2 k^2+(6 h^4+3 h^3-5 h^2\\
\phantom{E_{20}:=}-3 h-1) k+h^3+4 h^2+5 h+2) \sigma_2^2+k (h-1)^3 (2 h (h-1)^2 k^2-(h-1) (3 h^2-2 h+1) k\\
\phantom{E_{20}:=}-3 h (h+1)) \sigma_2-k^2 (h-1)^5 (h^2 k+h-k+2)) \epsilon_1+(h^2 k-h k+h+1)^2 (h^3 k-2 h^2 k+h^2+h k\\
\phantom{E_{20}:=}+2 h \sigma_2-2 h+1) (h^2 k-2 h k+h \sigma_2+k+\sigma_2)^3 \epsilon_2-\sigma_2 (h^2 k-h k+h+1)^2 (h^2 k-2 h k-2 h \sigma_2\\
\phantom{E_{20}:=}+3 h+k+\sigma_2-3) (h^2 k-2 h k+h \sigma_2+k+\sigma_2)^3 \}/\{\sigma_2 k (h^2 k-2 h k+h+k+\sigma_2-1)^3 \\
\phantom{E_{20}:=}\times(h \sigma_2-h+1) (h-1) (h^2 k-2 h k+h \sigma_2+k+\sigma_2)^2\},\\
D_{31}:= h^8+18 h^7+13 h^6-98 h^5+60 h^4+154 h^3+93 h^2+702 h-255,\\
D_{41}:= h^{13}+33 h^{12}+305 h^{11}-227 h^{10}-5629 h^9-1655 h^8+16622 h^7-21154 h^6-43339 h^5+33581 h^4\\
\phantom{D_{41}:=}+17581 h^3-31623 h^2-38277 h+14965,\\
D_{51}:= h^{18}+16 h^{17}-34 h^{16}+1496 h^{15}+3051 h^{14}-18274 h^{13}+239 h^{12}+298960 h^{11}+160771 h^{10}\\
\phantom{D_{41}:=}-960080 h^9+325619 h^8+2233144 h^7-857259 h^6-2020738 h^5+675405 h^4+1195888 h^3\\
\phantom{D_{41}:=}-957364 h^2+271188 h-28445,\\
D_{61}:= 6 h^{21}+10 h^{20}-1901 h^{19}+13365 h^{18}+43756 h^{17}-122354 h^{16}-1053075 h^{15}-202967 h^{14}\\
\phantom{D_{41}:=}+4549769 h^{13}-695287 h^{12}-14495889 h^{11}-2284869 h^{10}+29927221 h^9-1756299 h^8-58026517 h^7\\
\phantom{D_{41}:=}+16372127 h^6+44351321 h^5-9554651 h^4-29330922 h^3+21456856 h^2-5879673 h+605877,\\
D_{71}:= 8 h^{22}-28 h^{21}-2731 h^{20}+29614 h^{19}-7745 h^{18}-521638 h^{17}-839519 h^{16}+6657768 h^{15}\\
\phantom{D_{41}:=}+12874736 h^{14}-21638968 h^{13}-29546678 h^{12}+38082396 h^{11}+22852074 h^{10}-32346916 h^9\\
\phantom{D_{41}:=}-11283238 h^8+24860424 h^7-8757992 h^6-717148 h^5+354673 h^4+621190 h^3-382169 h^2\\
\phantom{D_{41}:=}+87034 h-7531.\\
\mu_{110}:=\{2 k^3 (h \sigma_2-h+1)^3 (h-1)^3 (h^2 k-h k+h+1) ((2 h^4 (h-1)^3 k^3+(7 h^4+12 h^3+6 h^2-4 h+1)\\
\phantom{\mu_{110}:=}\times (h-1)^2 k^2+8 h (h-1) (h+1)^3 k+3 (h+1)^4) \sigma_2+k (h-1)^2((h-1) (h^3+3 h^2-3 h+1) k\\
\phantom{\mu_{110}:=}+(h+1)^3) (h k-k+1))\}/\{(-h^3 k+2 h^2 k-h^2-h k+h+2)^2 (h^2 k-2 h k+h \sigma_2+k+\sigma_2)^4 \\
\phantom{\mu_{110}:=}\times(2 h-2) (h^3 k \sigma_2-h^2 k \sigma_2+h^2 k+h^2 \sigma_2-2 h k+2 h \sigma_2+k+\sigma_2)\} ,\\
\mu_{101}:=
\{k ((h^4 k-h^3 k+h^3+6 h-2) \sigma_2-(h-1) (h^3 k-5 h^2 k+h^2+6 h k-4 h-2 k+4)) (-h^3 k+2 h^2 k\\
\phantom{\mu_{110}:=}-h^2-h k+h+2)\}/\{(h-1)^2 ((2 h^4 (h-1)^3 k^3+(7 h^4+12 h^3+6 h^2-4 h+1) (h-1)^2 k^2\\
\phantom{\mu_{110}:=}+8 h (h-1) (h+1)^3 k+3 (h+1)^4) \sigma_2+k (h-1)^2 (h^4 k+2 h^3 k+h^3-6 h^2 k+3 h^2+4 h k+3 h\\
\phantom{\mu_{110}:=}-k+1) (h k-k+1))\},\\
\mu_{210}:=
\{4 k^2 (h-1)^5 (h^2 k-h k+h+1)^3 ((h^8 (h^2-2 h+3) (h-1)^7 k^7+h^7 (11 h^3-9 h^2+14 h+44) (h-1)^6 k^6\\
\phantom{\mu_{110}:=}+2 h^6 (23 h^4+7 h^3+32 h^2+159 h+140) (h-1)^5 k^5+2 h^2 (50 h^8+68 h^7+139 h^6+589 h^5+924 h^4\\
\phantom{\mu_{110}:=}+448 h^3+28 h^2-8 h+1) (h-1)^4 k^4+h (h+1) (125 h^8+169 h^7+478 h^6+1990 h^5+3092 h^4\\
\phantom{\mu_{110}:=}+1626 h^3+236 h^2-62 h+8) (h-1)^3 k^3+(h+1) (91 h^9+208 h^8+558 h^7+2298 h^6+4660 h^5\\
\phantom{\mu_{110}:=}+4398 h^4+1888 h^3+306 h^2-60 h+8) (h-1)^2 k^2+2 h (h-1) (18 h^3-33 h^2+151 h+76)(h+1)^6 k\\
\phantom{\mu_{110}:=}+(6 (h^3-2 h^2+10 h+4)) (h+1)^7) \sigma_2+k (h-1)^2 (h k-k+1) (h^7 (h^2-2 h+3) (h-1)^5 k^5\\
\phantom{\mu_{110}:=}+h^6 (6 h^3-h^2+3 h+38) (h-1)^4 k^4+h^2 (14 h^7+17 h^6+17 h^5+180 h^4+134 h^3+42 h^2-14 h+2)\\
\phantom{\mu_{110}:=}\times (h-1)^3 k^3+h (h+1) (16 h^7+21 h^6+42 h^5+312 h^4+254 h^3+162 h^2-52 h+8) (h-1)^2 k^2\\
\phantom{\mu_{110}:=}+(h-1) (h+1) (9 h^8+20 h^7+52 h^6+268 h^5+462 h^4+384 h^3+178 h^2-44 h+8) k+(2 (h^3-2 h^2\\
\phantom{\mu_{110}:=}+10 h+4)) (h+1)^6))\}/\{(-h^3 k+2 h^2 k-h^2-h k+h+2)^5 (h^2 k-2 h k+h \sigma_2+k+\sigma_2)^3 (2 h-2)\\
\phantom{\mu_{110}:=}\times ((h^3 k-h^2 k+h^2+2 h+1) \sigma_2+k (h-1)^2) (h^2 k-2 h k+h+k+\sigma_2-1)\},\\
\mu_{201}:=
\{-2 (h-1)^2 (h k-k+1) (h^2 k-h k+h+1) ((2 h^6 (h-1)^5 k^5+h^5 (25 h+19) (h-1)^4 k^4+h (103 h^5\\
\phantom{\mu_{110}:=}+183 h^4+60 h^3+10 h^2-5 h+1) (h-1)^3 k^3+(h+1) (181 h^5+326 h^4+114 h^3+6 h^2-6 h+2)\\
\phantom{\mu_{110}:=}\times (h-1)^2 k^2+(h-1)(143 h^2-31 h+6) (h+1)^4 k+(6 (7 h-2)) (h+1)^5) \sigma_2+k (h-1)^2 \\
\phantom{\mu_{110}:=}\times(2 h^5 (h-1)^4 k^4+h (18 h^4+11 h^3+6 h^2-4 h+1) (h-1)^3 k^3+(2 (22 h^5+37 h^4+28 h^3-6 h^2+1))\\
\phantom{\mu_{110}:=}\times (h-1)^2 k^2+(h-1) (h+1) (42 h^4+73 h^3+45 h^2-21 h+10) k+(2 (7 h-2)) (h+1)^4))\}\\
\phantom{\mu_{110}:=}/\{(-h^3 k+2 h^2 k-h^2-h k+h+2)^3 \sigma_2^2(2 h-2) ((h^3 k-h^2 k+h^2+2 h+1) \sigma_2\\
\phantom{\mu_{110}:=}+k (h-1)^2) (h^2 k-2 h k+h+k+\sigma_2-1)^2\},\\
\mu_{120}:=\{-(h \sigma_2-h+1)^2 k^2 (h-1)^2 ((2 h^6 (h^2-2 h+3) (h-1)^5 k^5+h^2 (11 h^6-8 h^5+23 h^4+44 h^3+15 h^2\\
\phantom{\mu_{110}:=}-6 h+1) (h-1)^4 k^4+h (h+1) (24 h^6-15 h^5+78 h^4+156 h^3+64 h^2-23 h+4) (h-1)^3 k^3+(h+1)\\
\phantom{\mu_{110}:=}\times (26 h^7+9 h^6+102 h^5+366 h^4+344 h^3+89 h^2-22 h+4) (h-1)^2 k^2+h (h-1) (14 h^3-25 h^2\\
\phantom{\mu_{110}:=}+111 h+60)(h+1)^4 k+(3 (h^3-2 h^2+10 h+4)) (h+1)^5) \sigma_2+k (h-1)^2 (h k-k+1)(h^2 (h^5+h^4\\
\phantom{\mu_{110}:=}-4 h^3+10 h^2-5 h+1) (h-1)^3 k^3+h (h+1) (3 h^5+h^4-h^3+37 h^2-18 h+4) (h-1)^2 k^2+(h-1)\\
\phantom{\mu_{110}:=}\times (h+1) (3 h^6+2 h^5+10 h^4+52 h^3+41 h^2-14 h+4) k+(h^3-2 h^2+10 h+4) (h+1)^4))\}\\
\phantom{\mu_{110}:=}/\{2(h^2 k-h k+h+1) (-h^3 k+2 h^2 k-h^2-h k+h+2)^2 (h^2 k-2 h k+h \sigma_2+k+\sigma_2)^7 (h^3 k \sigma_2\\
\phantom{\mu_{110}:=}-h^2 k \sigma_2+h^2 k+h^2 \sigma_2-2 h k+2 h \sigma_2+k+\sigma_2)^2\},\\
\end{array}
\end{eqnarray*}
\begin{eqnarray*}
\begin{array}{l}
\mu_{111}:=\{-k^3 (h \sigma_2-h+1)^3 (h-1)^3 ((2 h^8 (h-6) (h-1)^7 k^7+2 h^7 (10 h^2-61 h-73)(h-1)^6 k^6\\
\phantom{\mu_{110}:=}+h (79 h^8-485 h^7-1240 h^6-729 h^5-10 h^4+11 h^3-10 h^2+5 h-1)(h-1)^5 k^5+h (165 h^8\\
\phantom{\mu_{110}:=}-1014 h^7-4058 h^6-4580 h^5-1782 h^4-104 h^3-12 h^2+36 h-11) (h-1)^4 k^4+(2 (100 h^9\\
\phantom{\mu_{110}:=}-613 h^8-3376 h^7-5529 h^6-3939 h^5-1147 h^4-77 h^3+19 h^2-6 h-1)) (h-1)^3 k^3+(2 (71 h^5\\
\phantom{\mu_{110}:=}-717 h^4-617 h^3-6 h^2+2 h-1)) (h-1)^2 (h+1)^4 k^2+(h-1) (55 h^4-608 h^3-398 h^2+51 h+2)\\
\phantom{\mu_{110}:=} \times (h+1)^5 k+(9 (h^3-12 h^2-5 h+2)) (h+1)^6) \sigma_2+k (h-1)^2 (h k-k+1) (2 h^7 (h-6) (h-1)^5 k^5\\
\phantom{\mu_{110}:=} +h (11 h^7-70 h^6-116 h^5-5 h^4+5 h^3-6 h^2+4 h-1) (h-1)^4 k^4+h (24 h^7-159 h^6-513 h^5\\
\phantom{\mu_{110}:=} -386 h^4-80 h^3+h^2+31 h-12) (h-1)^3 k^3+(26 h^8-177 h^7-843 h^6-1156 h^5-640 h^4-107 h^3\\
\phantom{\mu_{110}:=} +51 h^2-22 h-2) (h-1)^2 k^2+(h-1) (14 h^4-153 h^3-83 h^2+20 h-4) (h+1)^4 k+(3 (h^3-12 h^2\\
\phantom{\mu_{110}:=} -5 h+2)) (h+1)^5))\}/\{2(-h^3 k+2 h^2 k-h^2-h k+h+2)^3 \sigma_2 (h^2 k-2 h k+h \sigma_2+k+\sigma_2)^4 \\
\phantom{\mu_{110}:=} \times(h^3 k \sigma_2-h^2 k \sigma_2+h^2 k+h^2 \sigma_2-2 h k+2 h \sigma_2+k+\sigma_2)^2\},\\
\mu_{102}:=
\{k (h \sigma_2-h+1) (h^3 k-h^2 k+h^2-h k+h+k+\sigma_2-2) ((3 h^4 (h-1)^3 k^3+(15 h^4+26 h^3-6 h^2+1)\\
\phantom{\mu_{110}:=}\times (h-1)^2 k^2+(h-1) (21 h^4+56 h^3+30 h^2-6 h+8) k+(3 (3 h+1)) (h+1)^3) \sigma_2+k (h-1)^2 \\
\phantom{\mu_{110}:=}\times((6 h^3-5 h^2+h+1) (h-1)^2 k^2+(2 (h-1)) (6 h^3-3 h^2+h+5) k+6 h^3-h^2+h+17))\}\\
\phantom{\mu_{110}:=}/\{4(h^2 k-h k+h+1) (h-1) \sigma_2 (h^2 k-2 h k+h \sigma_2+k+\sigma_2) (h^3 k \sigma_2-h^2 k \sigma_2+h^2 k+h^2 \sigma_2-2 h k\\
\phantom{\mu_{110}:=}+2 h \sigma_2+k+\sigma_2)^2(h^2 k-2 h k+h+k+\sigma_2-1)\},\\
\mu_{202}:=
\{(h^2 k-h k+h+1) (h-1) ((h^7 (4 h^2-7 h+2) (h-1)^7 k^7+h^6 (50 h^3-47 h^2-54 h+20) (h-1)^6 k^6\\
\phantom{\mu_{110}:=}+h (247 h^8-26 h^7-544 h^6-112 h^5+22 h^4+20 h^3-8 h^2+4 h-1) (h-1)^5 k^5+(625 h^9+418 h^8\\
\phantom{\mu_{110}:=}-1678 h^7-1698 h^6-356 h^5-102 h^4+54 h^3+14 h^2-5 h-2) (h-1)^4 k^4+(890 h^9+1237 h^8\\
\phantom{\mu_{110}:=}-2370 h^7-5308 h^6-3444 h^5-1368 h^4-304 h^3+152 h^2-30 h-16) (h-1)^3 k^3+(724 h^9+1501 h^8\\
\phantom{\mu_{110}:=}-1630 h^7-7156 h^6-8048 h^5-4948 h^4-2148 h^3-468 h^2-78 h-28) (h-1)^2 k^2+(h-1)\\
\phantom{\mu_{110}:=}\times (315 h^4-715 h^3-55 h^2-225 h-48) (h+1)^5 k+(3 (19 h^3-50 h^2+5 h-22)) (h+1)^6) \sigma_2\\
\phantom{\mu_{110}:=}+k (h-1)^2 (h^6 (4 h^2-7 h+2) (h-1)^6 k^6+h (39 h^7-50 h^6-7 h^5-19 h^4+15 h^3-5 h^2+3 h-1)\\
\phantom{\mu_{110}:=}\times (h-1)^5 k^5+(135 h^8-85 h^7-143 h^6-163 h^5-45 h^4+47 h^3+15 h^2-9 h-2) (h-1)^4 k^4\\
\phantom{\mu_{110}:=}+(2 (115 h^8+10 h^7-211 h^6-341 h^5-285 h^4+53 h^3+63 h^2-27 h-10)) (h-1)^3 k^3+(210 h^8\\
\phantom{\mu_{110}:=}+175 h^7-508 h^6-1258 h^5-1410 h^4-350 h^3+166 h^2-174 h-56) (h-1)^2 k^2+(h-1)\\
\phantom{\mu_{110}:=}\times (99 h^8+158 h^7-275 h^6-1027 h^5-1365 h^4-805 h^3-193 h^2-233 h-92) k+(19 h^3-50 h^2\\
\phantom{\mu_{110}:=}+5 h-22) (h+1)^5))\}/\{2(-h^3 k+2 h^2 k-h^2-h k+h+2)^3 \sigma_2^2 (h^3 k \sigma_2-h^2 k \sigma_2+h^2 k+h^2 \sigma_2\\
\phantom{\mu_{110}:=}-2 h k+2 h \sigma_2+k+\sigma_2)^2 (h^2 k-2 h k+h+k+\sigma_2-1)^2\}.
\end{array}
\end{eqnarray*}

\end{document}